\documentclass{amsart}%
\usepackage{amsfonts, amsmath, amssymb}%
\usepackage{graphicx, epic, epsf, enumerate, stmaryrd}

\input xy
\xyoption{all}

\newtheorem{theorem}{Theorem}
\theoremstyle{definition}

\newtheorem{definition}{Definition}

\newtheorem{remark}{Remark}

\numberwithin{equation}{section}

\def\co{\colon\thinspace} 

\newcommand{\R}{\mathbb{R}}

\begin{document}

\title{Movie moves for singular link cobordisms in 4-dimensional space}

\author{Carmen Caprau}
\address{Department of Mathematics, California State University, Fresno, CA 93740, USA}
\email{ccaprau@csufresno.edu}
\urladdr{}

\date{}

\begin{abstract}
Two singular links are cobordant if one can be obtained from the other by singular link isotopy together with a combination of births or deaths of simple unknotted curves, and saddle point transformations. A movie description of a singular link cobordism in 4-space is a sequence of singular link diagrams obtained from a projection of the cobordism into 3-space by taking 2-dimensional cross sections perpendicular to a fixed direction.
We present a set of movie moves that are sufficient to connect any two movies of isotopic singular link cobordisms.
\end{abstract}
\maketitle

\section{Introduction}\label{sec:intro}
Much as circles can be embedded in $\R^3$, giving rise to knots and links, surfaces can be embedded in $\R^4$. Knotted surfaces can be studied diagrammatically via a projection (diagram) in $\R^3$. Such a diagram can be cut by planes that are perpendicular to a fixed direction in the 3-space of the projection, giving rise to a movie description for a knotted surface.
Carter and Saito~\cite{CS} found a set of fifteen moves, called \textit{movie moves}, such that two movies represent isotopic knotted surfaces if and only if there is a finite sequence of moves from this set or interchanges of the levels of distant critical points that takes one to the other. The result in~\cite{CS} generalizes Roseman's set of seven moves~\cite{Ros}, which are moves on projections where a height function has not been specified in the 3-space of the projection, and are analogous to Reidemeister moves for classical link diagrams. 

Singular links are embeddings of 4-valent graphs in $\R^3$ considered up to rigid-vertex isotopies, and are interesting objects which enrich the field of knot theory. For rigid-vertex isotopies we refer the reader to Kauffman's work~\cite{Ka1, Ka2}.  We can go one dimension higher to consider cobordisms of singular links and study them up to ambient isotopy relative to boundary. We say that two singular links are cobordant if one can be obtained from the other by singular link isotopy and Morse transformations.

Analogous to knotted surfaces and, more generally, knot/link cobordisms, singular link cobordisms can be studied diagrammatically using their projections in 3-space or by their movie representations. One then may desire to have at hand a set of Reidemeister-type moves that describe diagrams of isotopic singular link cobordisms in 4-space. Even more so, upon fixing a height function on the 3-space of the projection and working with movie descriptions of these projections, one may ask: ``What are the movie moves for isotopic singular links cobordisms?" The purpose of our paper is to answer this question.  

In order to find a complete list of movie moves that are sufficient to construct any isotopy of singular link cobordisms, we develop first an analogy with the Reidemeister-type moves for singular link diagrams. Then, by examining  the generic critical points and transverse intersections in 4-space as an isotopy between singular link cobordisms occurs, we obtain the Roseman-type moves and their movie parametrizations for singular link cobordisms. In addition, by considering all possible changes in folds and in the relative position of the critical points of multiple-point strata and folds (where these changes preserve the topology of the projection) we arrive at a list of twenty-four movie moves for singular link cobordisms.
Our main result is the following Movie-Move Theorem:

\begin{theorem}\label{thm:movie-moves}
Two movies represent isotopic singular link cobordisms if and only if they are related to each other by a finite sequence of moves for movies depicted in Figures~\ref{fig:Roseman}, \ref{fig:Carter-Saito}, and \ref{fig:new-reidclips} or interchanges of the levels of distant critical points.
\end{theorem}

Since singular link cobordisms include classical knot/link cobordisms, the set of movie moves given here will include the Carter-Saito movie moves.  The idea of the proof of our Movie-Move Theorem is similar to that in~\cite{CS, Ca, Ros}. Although some familiarity with~\cite{CS} would be useful, we have strived to have a self-contained paper and detailed exposition.

This paper is organized as follows: In Section~\ref{sec:sing-links} we discuss the critical points and transverse intersections in 3-space as an isotopy of singular links occurs. Then we use this analysis to provide a proof of the classical result involving Reidemeister-type moves for singular link diagrams. The heart of the paper is Section~\ref{sec:cobordisms}, where we introduce singular link cobordisms and describe their generic diagrams in 3-space, analyze the singular strata of such diagrams, and finally prove Theorem~\ref{thm:movie-moves}. We finish with a few remarks and comments in Section~\ref{sec:remarks}.

\section{Singular links and their isotopies}\label{sec:sing-links}

A \textit{singular link} is an immersion of a disjoint union of circles in $\R^3$, which has finitely many singularities that are all transverse double points. A singular link can be regarded as an embedding in $\R^3$ of a 4-valent graph whose vertices have been replaced by rigid disks. Each disk has four strands attached to it, and the cyclic order of these strands is determined via the rigidity of the disk. The second definition is more convenient for our purposes. Although we will refer to our objects as singular links, for simplicity, we will keep in mind that these are merely embeddings in $\R^3$ of 4-valent graphs with rigid vertices. In this paper, we will refer to singular crossings (singularities) in a singular link as \textit{vertices}. If the underlying graph $G$ is connected, a rigid vertex embedding of it is called a singular knot; otherwise it is a singular link. For simplicity, in this paper we will use the generic term `singular links' to refer to both singular links and singular knots.

A singular link \textit{diagram} is a projection of a singular link into a plane. We work only with generic projections, that is, diagrams in which all crossings are transverse double points.
Two singular links $L_1$ and $L_2$ are called \textit{equivalent} (or \textit{ambient isotopic}) if there exists an orientation-preserving homeomorphism of $\R^3$ onto itself that maps $L_1$ to $L_2$.

Figure~\ref{fig:extended-Reid} depicts moves for singular link diagrams which we refer to as the \textit{extended Reidemeister moves}. The following statement is well-known (see for example~\cite{Ka1, Ka2}).
\begin{figure}[ht]
\[\raisebox{-13pt}{\includegraphics[height=0.4in]{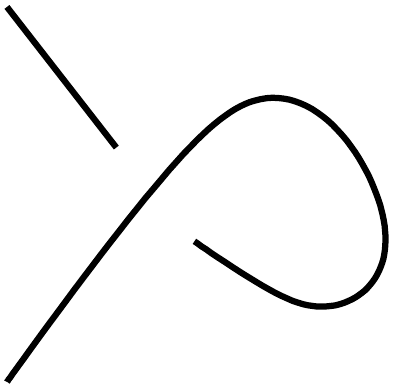}}\,\,  \stackrel{R1}{\longleftrightarrow} \,\, \raisebox{-13pt}{\includegraphics[height=0.4in]{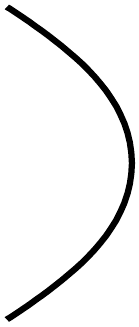}} \,\, \stackrel{R1}{\longleftrightarrow} \,\,   \reflectbox{\raisebox{17pt}{\includegraphics[height=0.4in, angle = 180]{poskink}}} \,\,\quad\,\, 
 \raisebox{-13pt}{\includegraphics[height=0.4in]{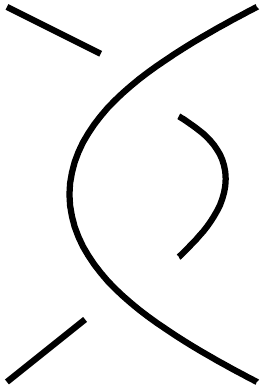}} \,\, \stackrel{R2}{\longleftrightarrow} \,\, \raisebox{-13pt}{\includegraphics[height=0.4in]{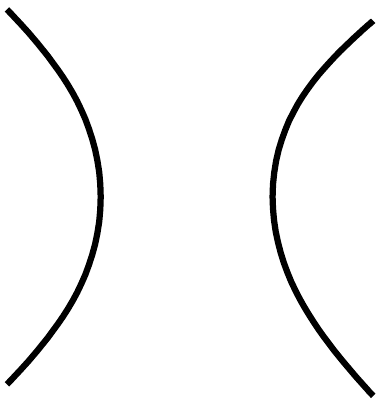}} \,\, \quad \,\, 
\raisebox{-13pt}{\includegraphics[height=0.4in]{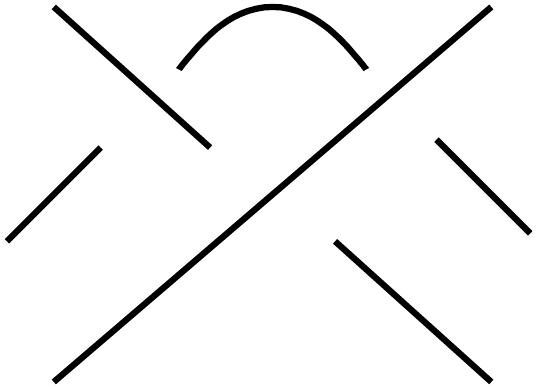}} \,\, \stackrel{R3}{\longleftrightarrow} \,\,   \raisebox{-13pt}{\includegraphics[height=0.4in]{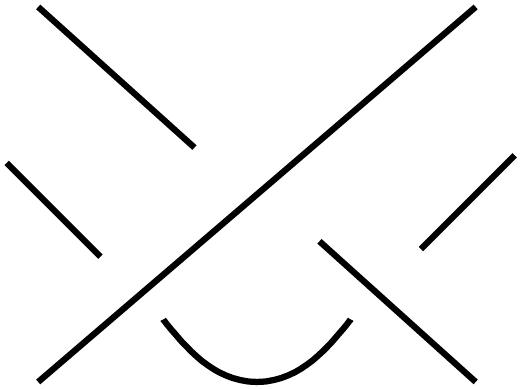}}\]

\[  \raisebox{-12pt}{\includegraphics[height=0.4in]{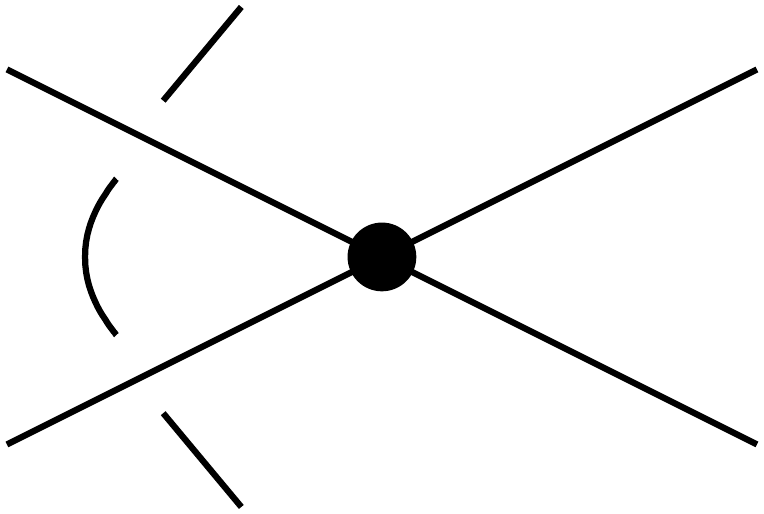}}\,\, \stackrel{R4}{\longleftrightarrow} \,\, \raisebox{-12pt}{\includegraphics[height=0.4in]{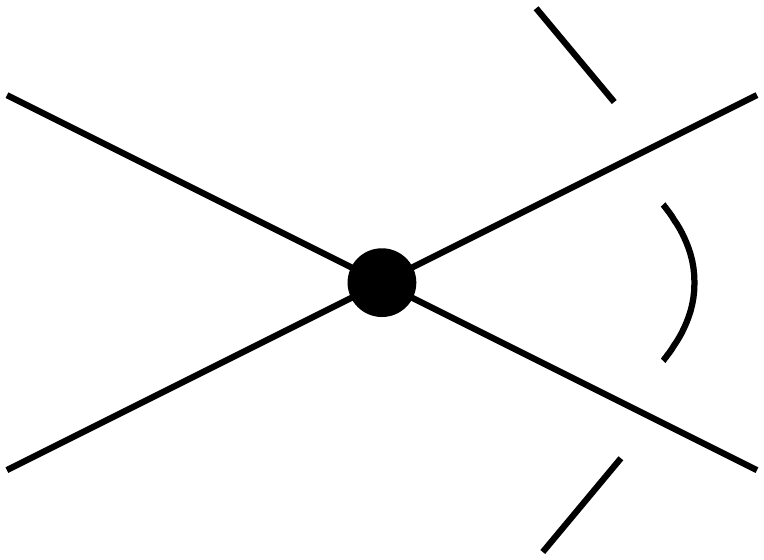}} \,\, \qquad \,\,  \raisebox{-12pt}{\includegraphics[height=0.4in]{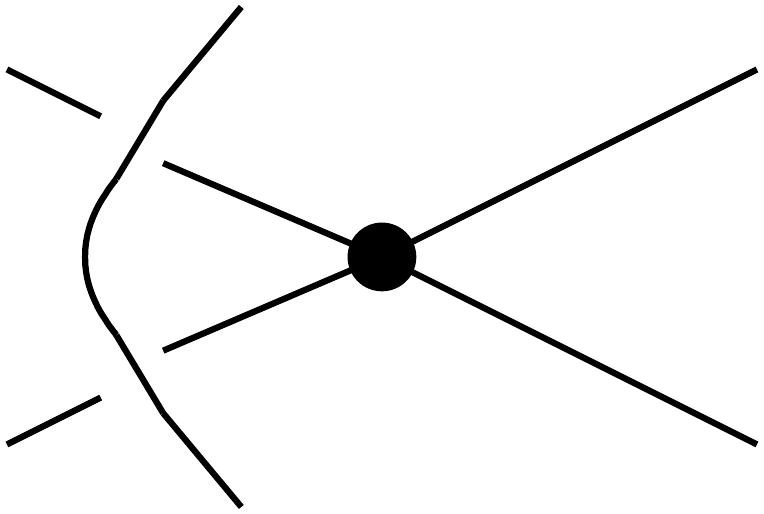}} \,\, \stackrel{R4}{\longleftrightarrow}\,\,  \raisebox{-12pt}{\includegraphics[height=0.4in]{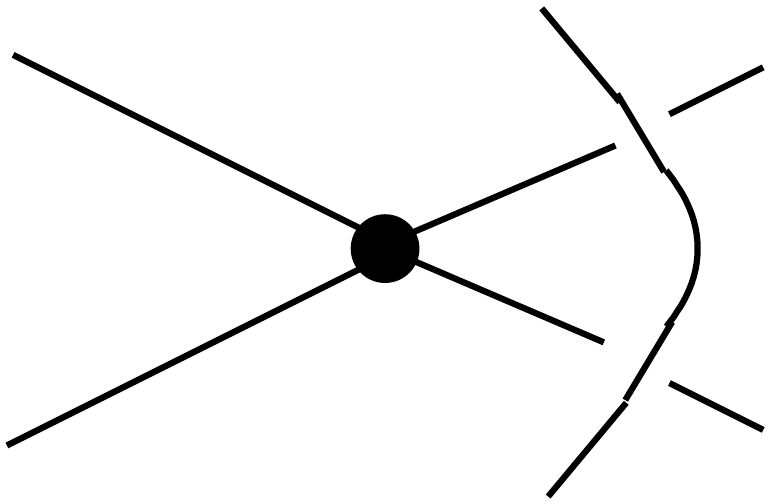}} \]\\
\[ \raisebox{-10pt}{\includegraphics[height=0.35in]{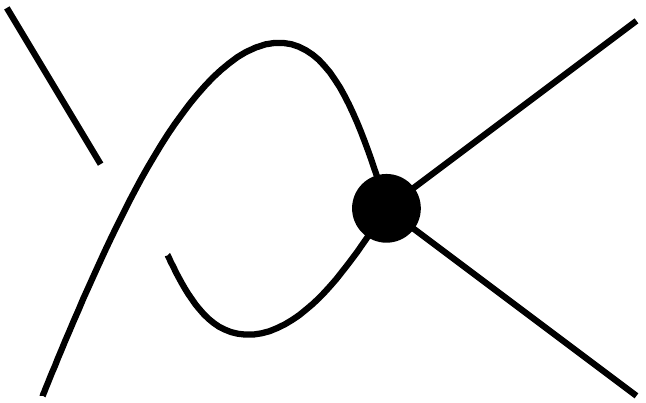}}\,\,\stackrel{R5}{ \longleftrightarrow} \,\,\reflectbox{\raisebox{-10pt}{\includegraphics[height=0.35in]{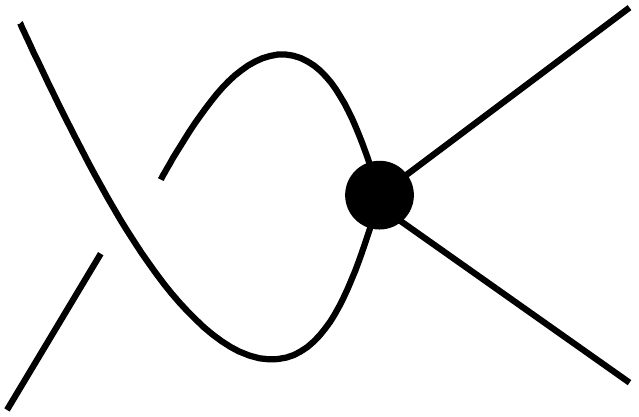}}}\]
\caption{Reidemeister-type moves for singular link diagrams}\label{fig:extended-Reid}
\end{figure}

\begin{theorem}\label{thm:isotopies}
Let $L_1$ and $L_2$ be singular links with diagrams $D_1$ and $D_2$, respectively. Then $L_1$ and $L_2$ are ambient isotopic if and only if there is a finite sequence of planar isotopies and moves $R1$ through $R5$ depicted in Figure~\ref{fig:extended-Reid} taking $D_1$ to $D_2$.
\end{theorem}

\begin{proof}
Although this is a classical result, we sketch its proof below. The main reason for this is to develop arguments which will be used as building blocks and as an analogy to the one-dimensional higher case which is the main point of this paper.

It is easy to see that if two diagrams differ by a finite sequence of the extended Reidemeister moves, then they represent equivalent singular links.

Suppose now that $L_1$ and $L_2$ (with diagrams $D_1$ and $D_2$, respectively) are equivalent singular links, which are rigid-vertex embeddings in $\R^3$ of a graph $G$.
Let $H \co G \times [0, 1] \to \R^3 \times [0, 1]$ denote an isotopy between $L_1$ and $L_2$, and let $p \co \R^3 \times [0, 1] \to \R^2 \times [0, 1]$ denote the projection map $\R^3 \to \R^2$ times the identity map on $[0, 1]$. Classifying the singularities of the function $p \circ H$ provides a proof for the sufficiency of the extended Reidemeister moves to relate diagrams of isotopic singular links. To be specific, each type of the extended Reidemeister move corresponds to a singularity of the function $p \circ H$.

The $0$-dimensional set of the diagram $D_1$ is composed by its vertices and crossing points, which trace out 1-dimensional sets during the isotopy taking $D_1$ to $D_2$. When studying Reidemeister-type moves for singular links, we need to examine the transverse intersections and the Morse critical points of the resulting 1-dimensional sets in $\mathbb{R}^2 \times [0,1]$ as an isotopy occurs. Here we work with the projection $\pi \co \R^2 \times [0, 1] \to [0, 1]$ as a height function. The map $\pi$  can be perturbed so that it is a Morse function on the 1-dimensional sets mentioned above. 

Without loss of generality, we may assume that the critical points and transverse intersections occur at different times during an isotopy. Moreover, by compactness, we may assume that there are only finitely many such critical points and transverse intersections.

In what follows, we will often denote by $X \times [0, 1]$ the 1-dimensional set traced out by a vertex in the isotopy direction and refer to it as an edge. Moreover, the 1-dimensional set traced out by a crossing will be called a double-point arc/curve.

During the isotopy $H$ between the singular links $L_1$ and $L_2$, the topology of the underlying graph $G$ is unchanged. Consequently, the moves for singular link diagrams do not correspond to critical points for the edge set, as these type of critical points would change the topology of the graph. So we only consider the critical points for the double-point set.

A Reidemeister type-I move, denoted by $R1$, corresponds to a critical point (specifically, a \textit{branch point}) of the height function $\pi \circ p \circ H$ on the double-point arc formed as the loop shrinks during the move. Specifically, a branch point occurs as the (isolated) point of intersection in the retinal plane of a double-point arc with a fold line created by the loop during the isotopy. A branch point is also the boundary point for the double-point arc.  A Reidemeister type-II move, denoted by $R2$, involves two crossing points that converge during the isotopy to a (minimum or maximum) critical point of $\pi \circ p \circ H$ on a double-point arc. The \textit{twisted crossing-vertex move}, $R5$, also corresponds to a critical point of a double-point curve; in the retinal plane, a twisted crossing-vertex move represents a transverse intersection between an edge $X \times [0,1]$ and a double-point arc, resulting in an isolated point of the function $p \circ H$, which we refer to as a \textit{crossing-vertex point}.
Therefore, the moves $R1, R2$ and $R5$ correspond to critical points of the double-point set. Figure~\ref{fig:critical-points} illustrates these moves together with the corresponding critical events in the surfaces that represent these isotopies. Note that a red curve represents a double-point curve and a thick black arc depicts an edge $X \times [0,1]$.
\begin{figure}[ht]
\[ \raisebox{5pt}{\includegraphics[height=0.85in]{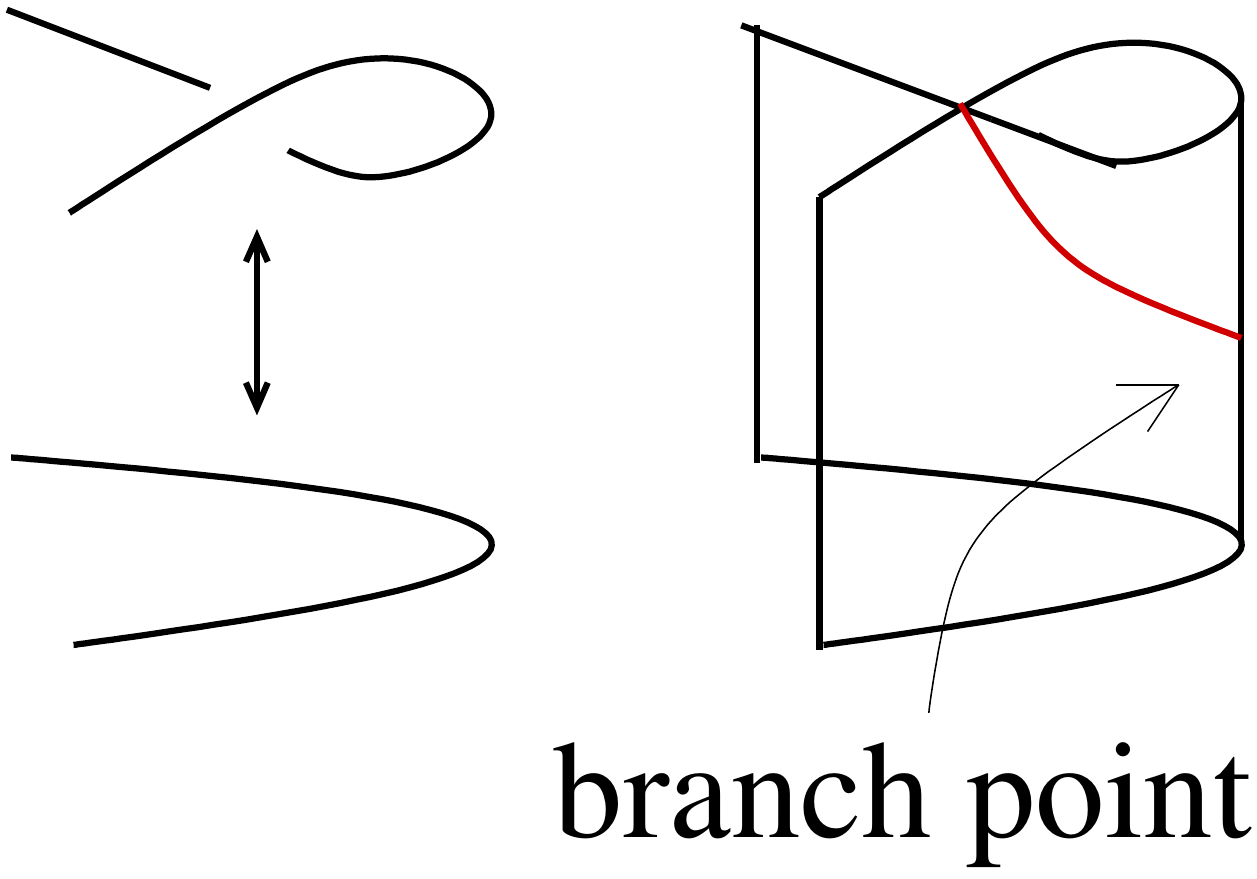}} \hspace{2cm} \raisebox{5pt}{\includegraphics[height=0.85in]{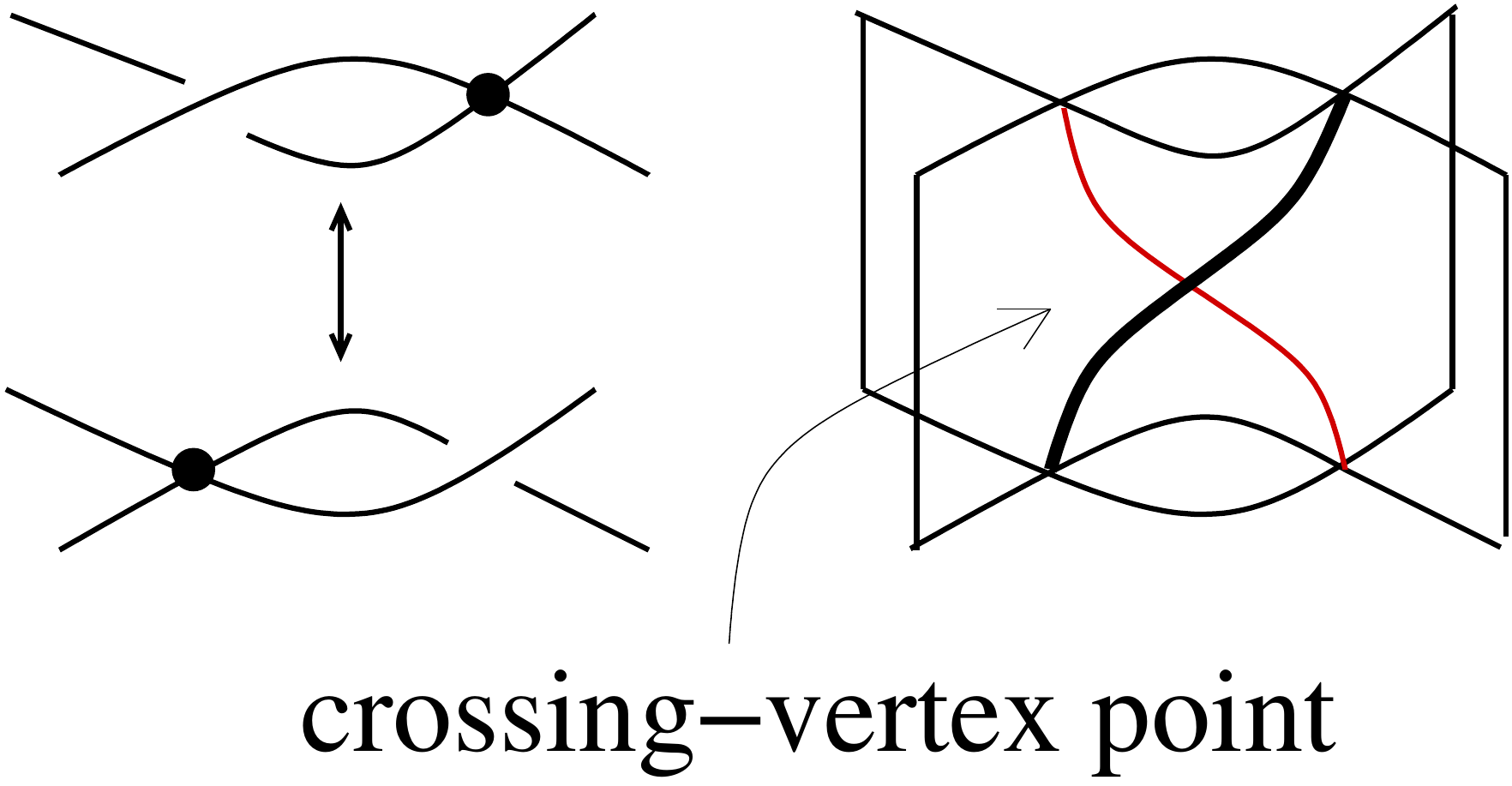}}
 \put(-260, 41){\fontsize{6}{6}R1}
 \put(-107, 41){\fontsize{6}{6}R5}\]
\[ \raisebox{5pt}{\includegraphics[height=0.85in]{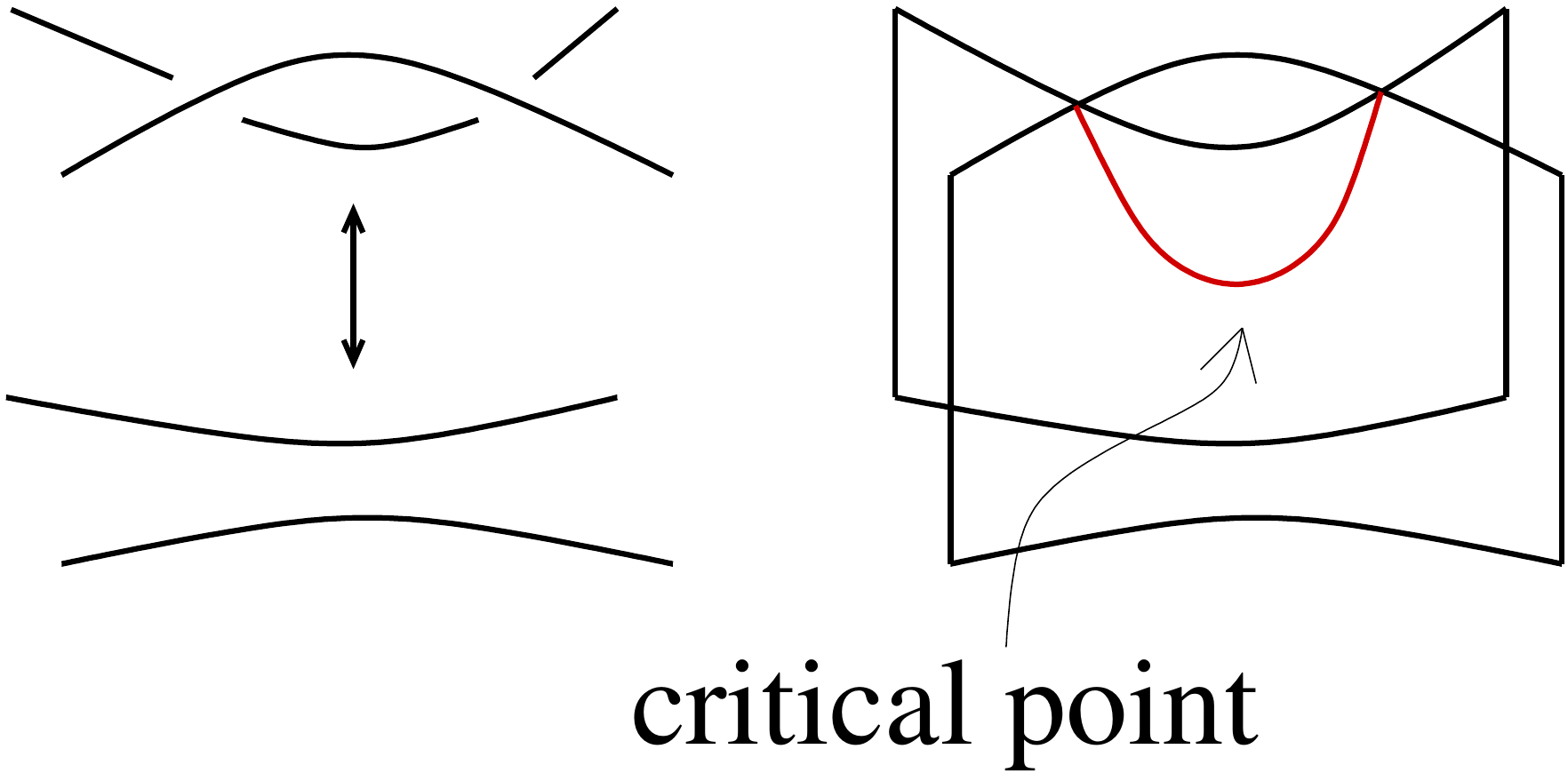}}
 \put(-111, 41){\fontsize{6}{6}R2}\]
 \caption{Moves corresponding to critical points of double-point arcs} \label{fig:critical-points}
\end{figure}

A transverse intersection in $\mathbb{R}^2 \times[0,1]$ is formed between a 1-dimensional set and a 2-dimensional set. The 1-dimensional sets that we need to consider are the edge set and the double-point set.

A Reidemeister type-III move, $R3$, represents an isolated triple point of the map $p \circ H$. During the move $R3$, a crossing point in $D_1$ traces a 1-dimensional set and intersects transversally with a 2-dimensional set in $\mathbb{R}^2 \times[0,1]$ consisting of an arc of the singular link diagram times the isotopy parameter. This move creates an isolated \textit {triple point}, as depicted on the left of Figure~\ref{fig:triple point}. During an extended Reidemeister move of type $R4$, a vertex passes through a transverse arc, either above or below; the edge $X \times [0,1]$ intersects transversally a sheet which is either entirely above or entirely below the edge $X \times [0,1]$. We call the isolated point created during this move a \textit{singular intersection point};  such a point is exemplified on the right of Figure~\ref{fig:triple point}. Therefore, the moves $R3$ and $R4$ correspond to a 1-dimensional set passing transversally through a 2-dimensional set in $\mathbb{R}^2 \times[0,1]$ (the plane of the projection times the isotopy direction). 
\begin{figure}[ht]
\[ \raisebox{5pt}{\includegraphics[height=1in]{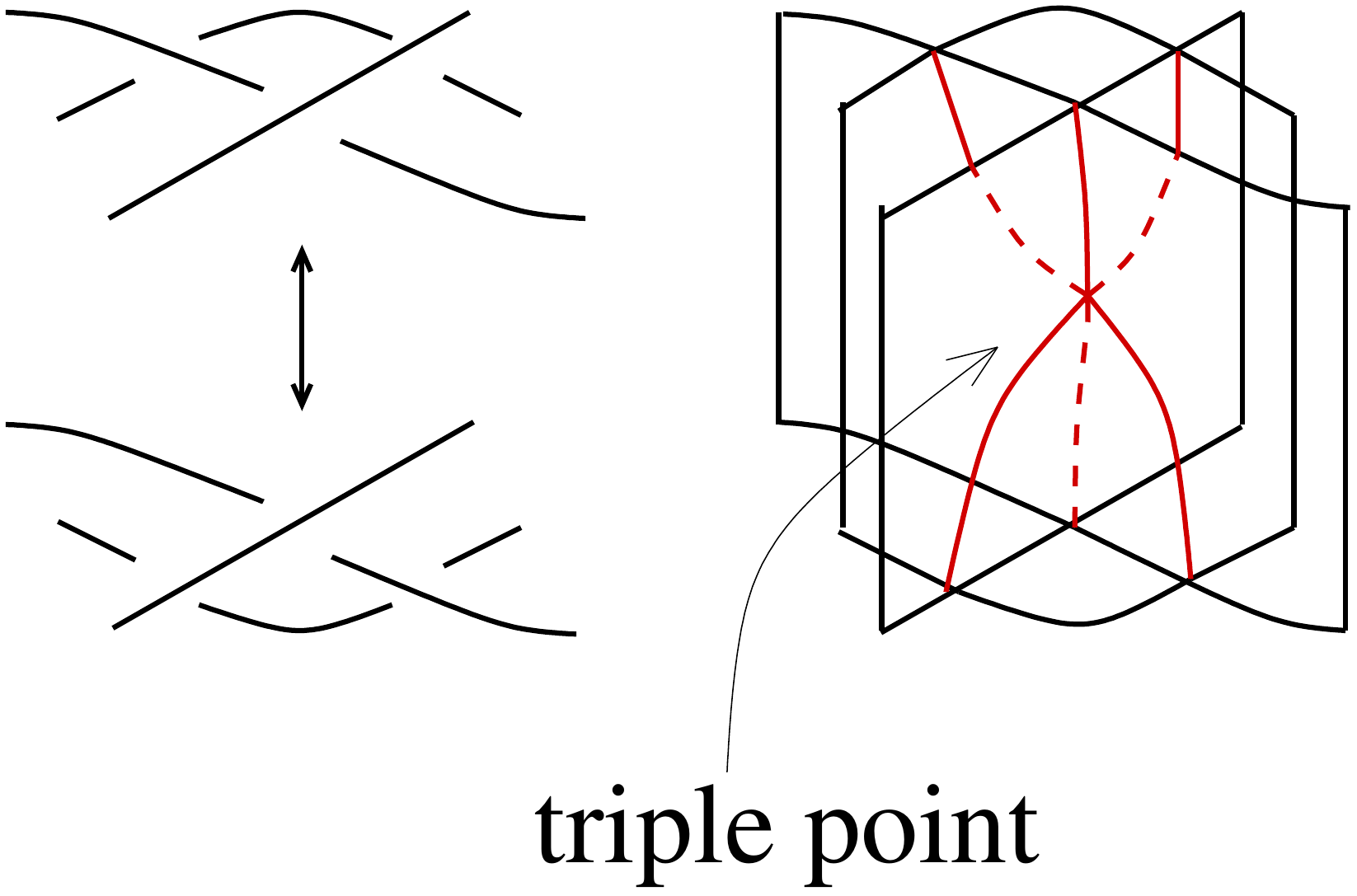}}
 \put(-102, 47){\fontsize{6}{6}R3}
 \hspace{2cm}
  \raisebox{5pt}{\includegraphics[height=1in]{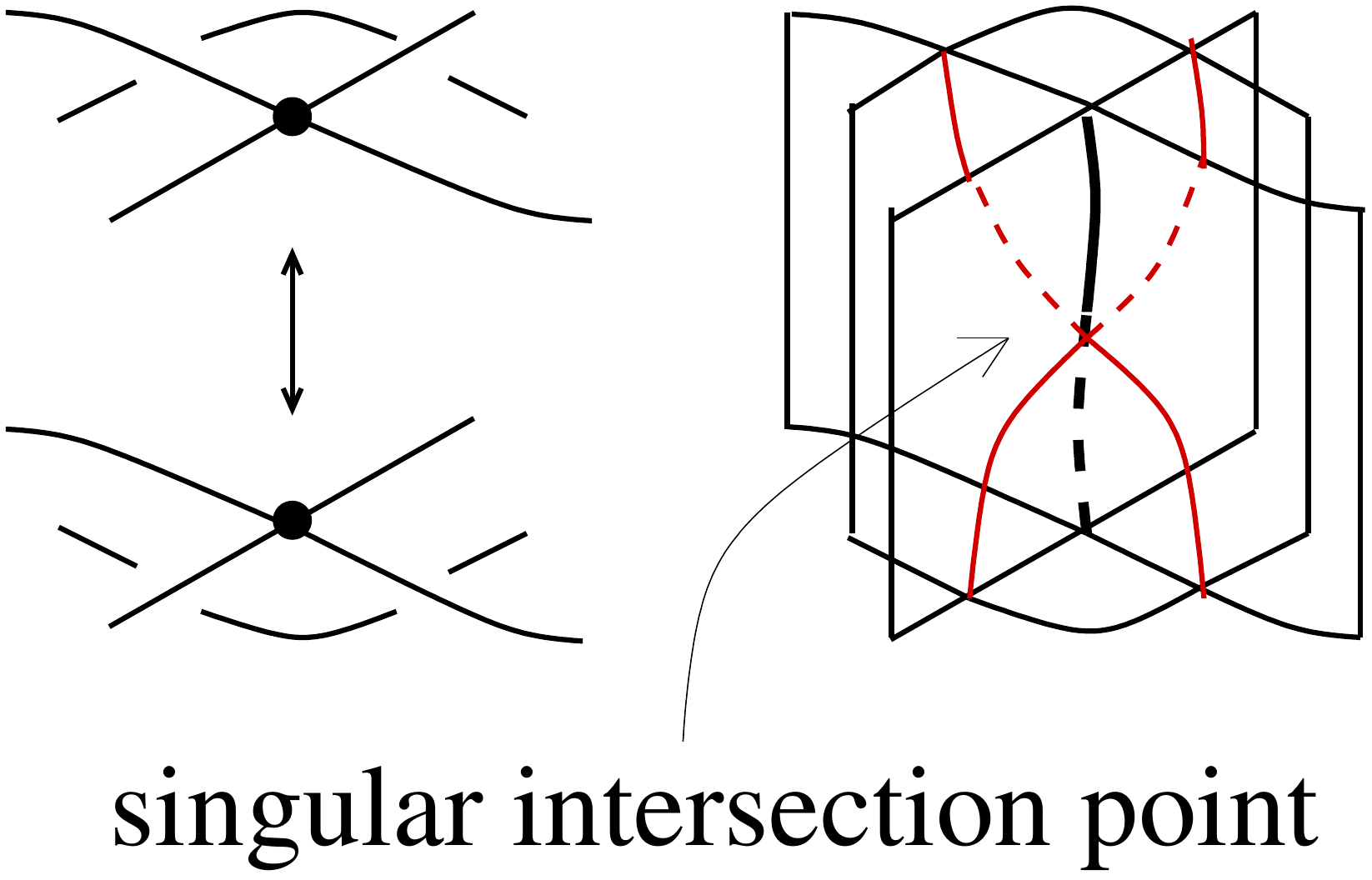}}  \put(-104, 47){\fontsize{6}{6}R4}
\]
\caption{Moves corresponding to transverse intersections in $\mathbb{R}^2 \times [0,1]$} \label{fig:triple point}
\end{figure}

The moves $R1$ through $R5$ exhaust the possibilities for critical points and transverse intersections as an isotopy between diagrams $D_1$ and $D_2$ occurs. Therefore, if diagrams $D_1$ and $D_2$ represent equivalent singular links, then are related (up to topological equivalence) by a finite sequence of the moves $R1$ through $R5$ showed in Figure~\ref{fig:extended-Reid}.
\end{proof}

Similar to the case of classical knots and links, categorical or algebraic descriptions for singular links are obtained by imposing a height function on the plane into which a given singular link is projected. In this case, the extended Reidemeister moves must be supplemented by three moves that take into account the height function. These additional moves are given in Figure~\ref{fig:additional-Reid-moves}, where the move in the middle is given only with one type of crossing. 
\begin{figure}[ht]
\[ \raisebox{-12pt}{\includegraphics[height=0.45in]{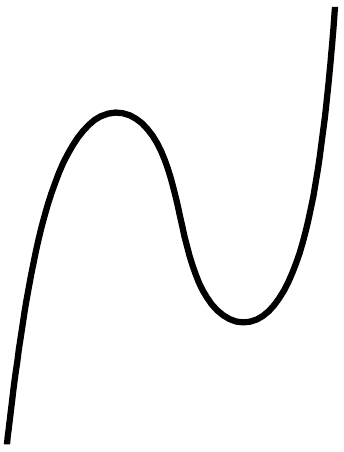}}\, \longleftrightarrow \, \raisebox{-12pt}{\includegraphics[height=0.45in]{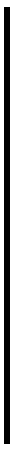}} \, \longleftrightarrow \, \reflectbox{\raisebox{-12pt}{\includegraphics[height=0.45in]{snake}}} \hspace{1cm} 
\raisebox{-12pt}{\includegraphics[height=0.45in]{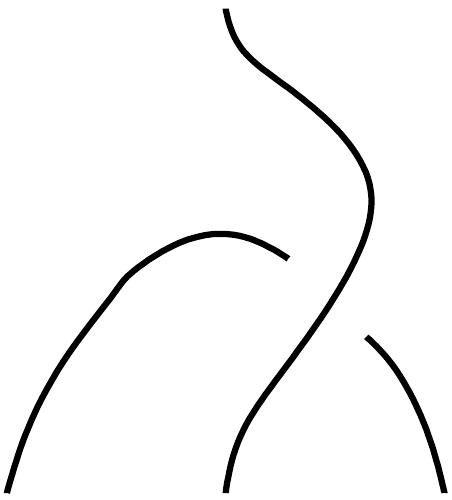}}  \longleftrightarrow  \reflectbox {\raisebox{-12pt}{\includegraphics[height=0.45in]{snake-over}}}
\hspace{1cm} 
\raisebox{-12pt}{\includegraphics[height=0.45in]{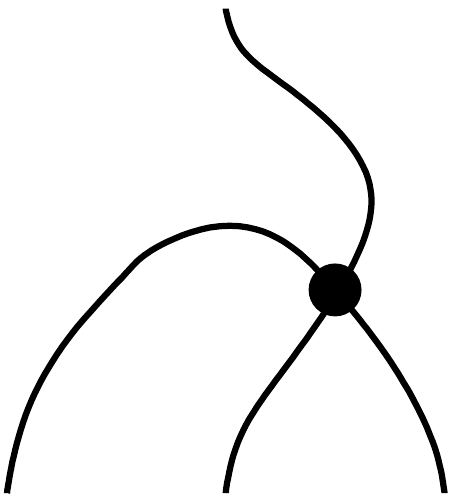}} \longleftrightarrow \reflectbox {\raisebox{-12pt}{\includegraphics[height=0.45in]{snake-vertex}}}
\]
\caption{Additional extended Reidemeister moves with height functions} \label{fig:additional-Reid-moves}
\end{figure}

Note that the moves similar to the second and third moves in Figure~\ref{fig:additional-Reid-moves} but with minima in place of maxima follow from the moves given here together with the exchange of the levels of distant critical points.

Combining the fixed height function on the plane onto which the singular link is projected with the time direction of the isotopy we obtain a projection of the singular link times an interval onto a plane. The singularities of this projection are the cusps and fold lines that are traced out by the maximal and minimal points of the singular link diagram. (Recall that any generic smooth map from a surface to the plane can be approximated by a map which has fold lines and isolated cusp points as its singularities; see~\cite{GG} for definitions and details.)

The cancelation of a local minimum and local maximum corresponds to a cusps singularity of fold lines.
Moving an arc pass a maximum point corresponds to a transverse intersection (in a plane) between a fold line and an edge or between a fold line and a double-point arc. These critical events are illustrated below:
\[\raisebox{10pt}{\includegraphics[height=1in]{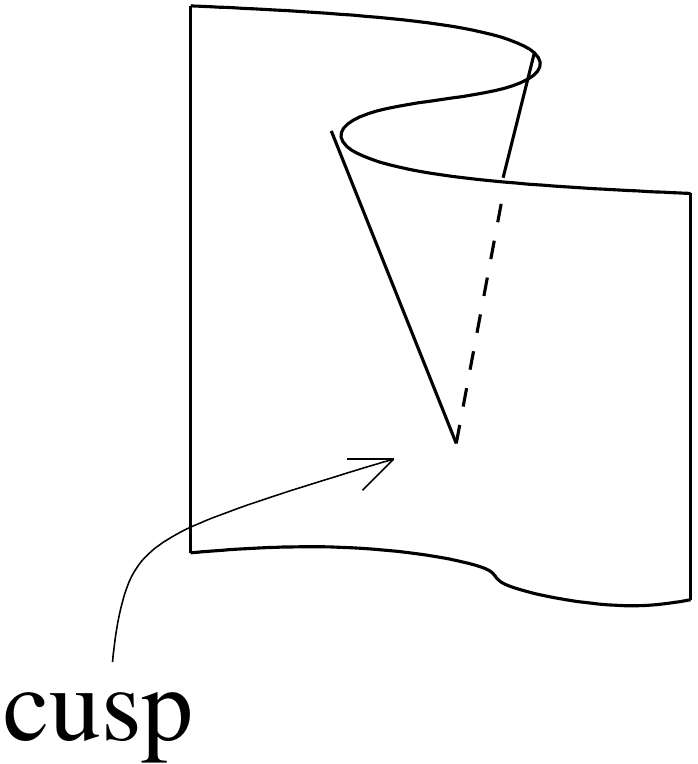}} \hspace{1.5cm} \raisebox{0pt}{\includegraphics[height=1.2in]{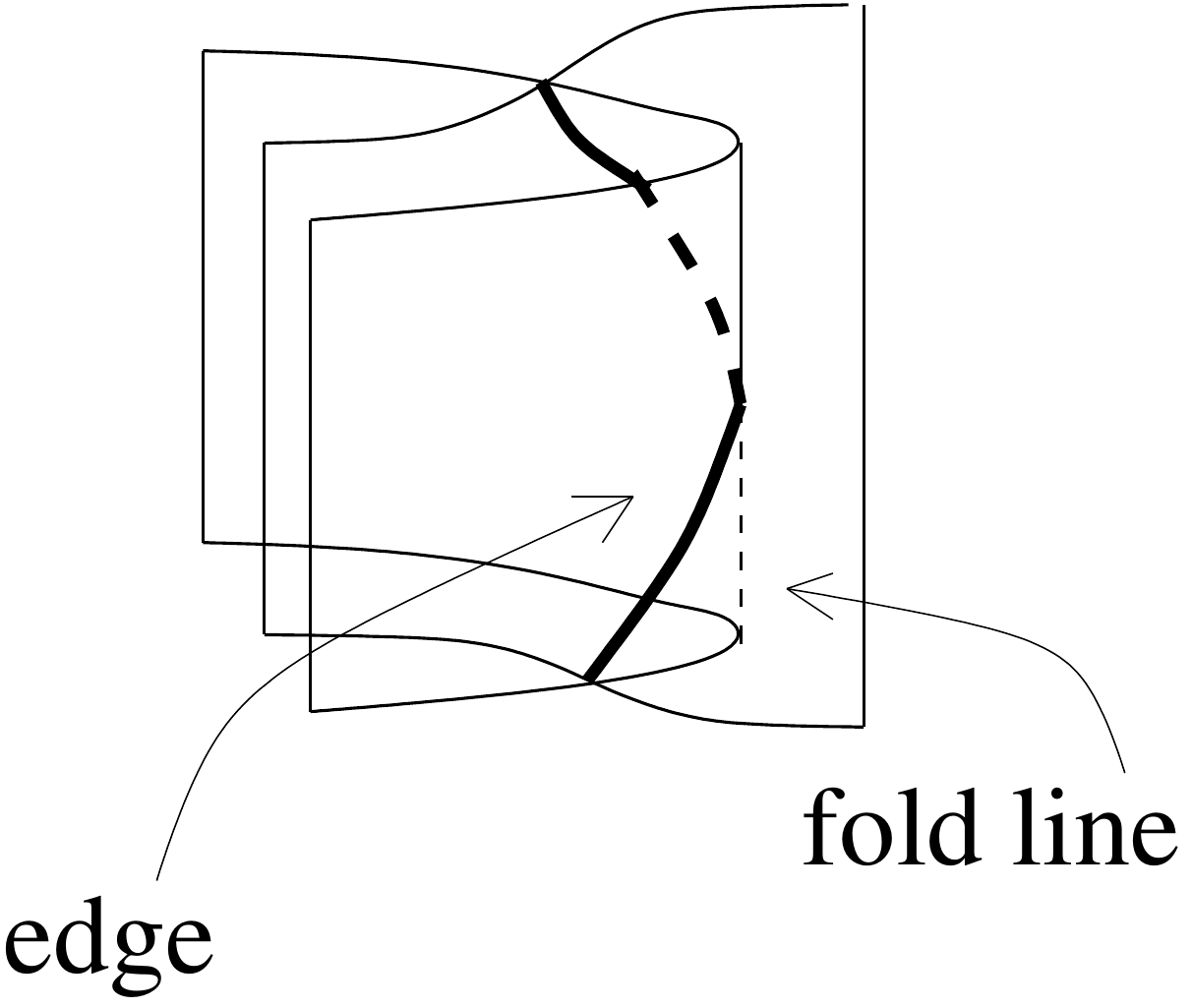}}  \hspace{0.7cm}
\raisebox{-5pt}{\includegraphics[height=1.2in]{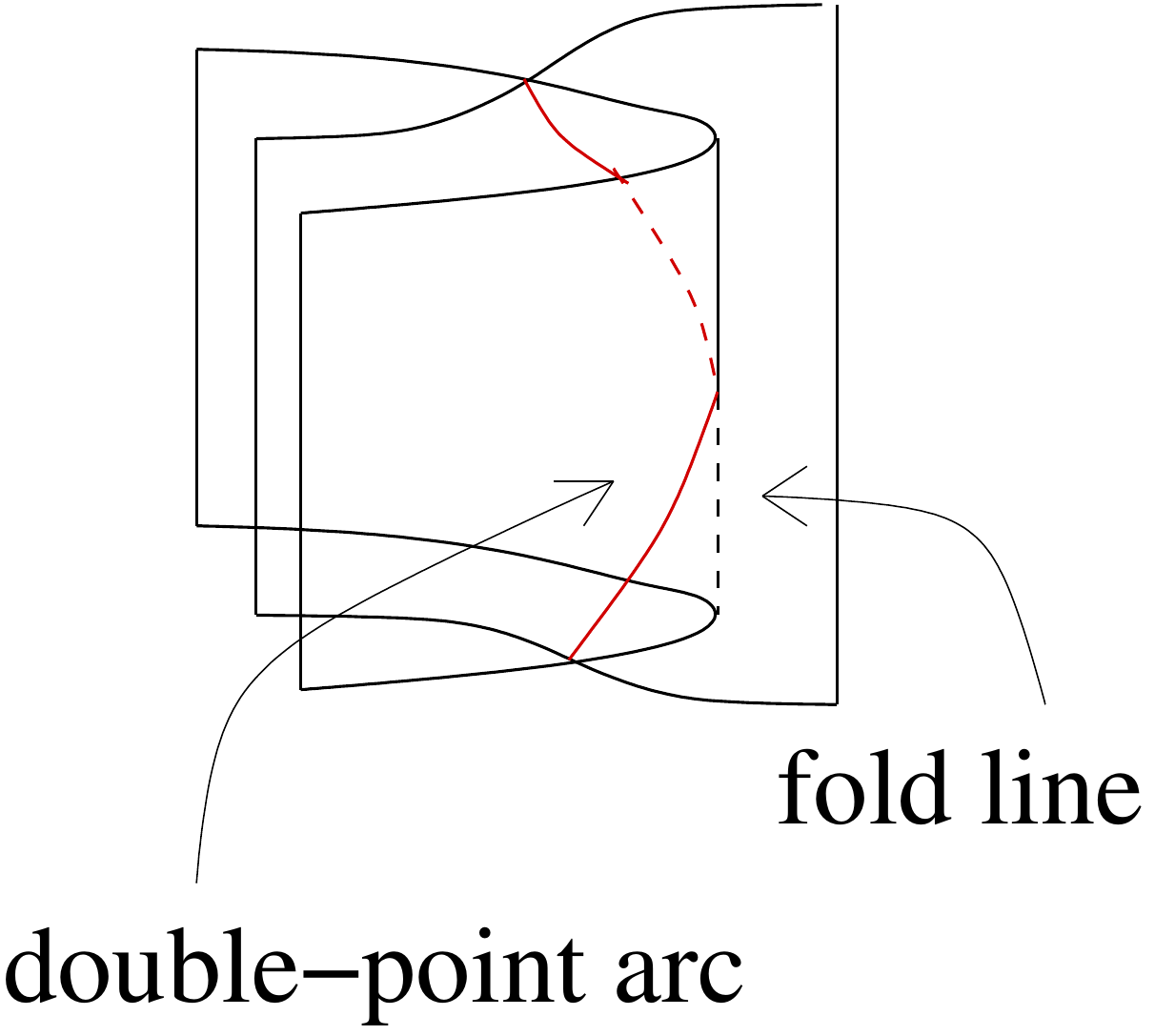}} \]

The moves in Figure~\ref{fig:additional-Reid-moves} cover all types of possible changes involving cusps and folds. Therefore, these additional moves together with moves in which the relative heights of distant critical points are interchanged are sufficient to modify the height function on any singular link (or knotted 4-valent graph).

\textbf{Conclusion.} We have studied the extended Reidemeister moves and regarded them as motions of a singular link in space. Equivalently, we interpreted the extended Reidemeister moves as 2-dimensional topological spaces embedded in 4-space. During such a move, the diagram of a singular link sweeps out a 2-dimensional set in $\mathbb{R}^2 \times [0, 1]$ (a singular surface in $\mathbb{R}^2 \times [0, 1]$) and contains multiple points and singularities. Such a singular surface is a projection onto $\mathbb{R}^2 \times [0, 1]$ of a diagram of a singular link cobordism, concepts that are discussed in the remaining of the paper.


\section{Singular link cobordisms and their isotopies}\label{sec:cobordisms}

In this section we introduce singular link cobordisms in 4-space and study them up to ambient isotopy fixed on the boundary. 

\begin{definition}
We say that two singular links $L_0$ and $L_1$ are \textit{cobordant} if one can be obtained from the other by singular link isotopy plus a combination of births or deaths of simple unknotted curves, and saddle point transformations. Cobordant singular links $L_0$ and $L_1$ are `joined' by a singular surface $C$ embedded in $\R^3 \times [0,1]$ with boundary $\partial C = (L_0 \times \{0\}) \cup (L_1 \times \{1\})$, which we call a \textit{singular link cobordism} between $L_0$ and $L_1$.
\end{definition}

Note that every point in a singular link cobordism has a neighborhood homeomorphic to a neighborhood of a point in $X \times [0, 1]$ (see Figure~\ref{fig:id-vertex}). 

\begin{figure}[ht]
\[\raisebox{-10pt}{\includegraphics[height=0.7in]{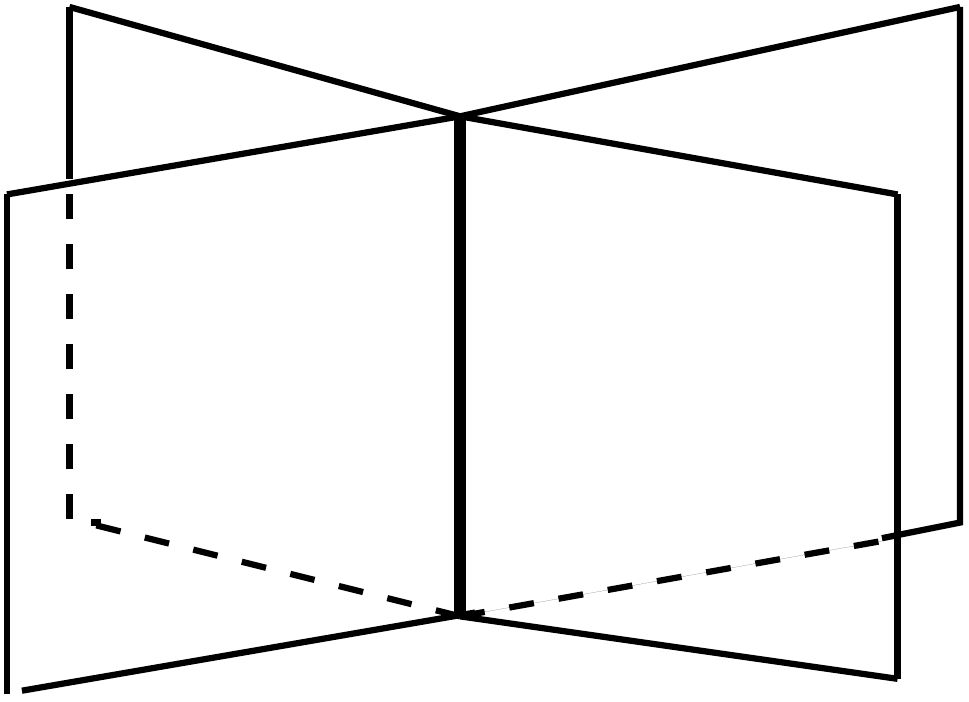}}\]
\caption{$X \times [0, 1]$}\label{fig:id-vertex}
\end{figure}

There is a natural projection from a singular link cobordism $C$ to $[0, 1]$, and the preimage of a regular value $t$ is a singular link $L_t = (\R^3 \times \{ t \}) \cap C$.

\subsection{Movie description of a singular link cobordism}
A singular link cobordism can be studied diagrammatically via a projection in $3$-space. This is analogous to the two-dimensional singular link diagrams. Consider the projection map $p \co \R^3 \times [0, 1] \to \R^2 \times [0, 1]$ which preserves the last coordinate. A \textit{diagram} of a singular link cobordism $C$  is the image, $p(C)$, of $C$ under the projection map $p$ together with crossing information for the image of the double and triple point manifolds. The map $p$ can be perturbed slightly so that the resulting diagram is generic, in the sense that the only 0-dimensional singular points in the interior of the diagram are branch points, crossing-vertex points, triple points, and singular intersection points (as defined in Section~\ref{sec:sing-links}), and such that the 1-dimensional singular strata contains double-point arcs/curves and edges of the form $X \times [0, 1]$. Note that a double-point arc/curve is caused by a transverse intersection between a pair of 2-disks. Also, a branch point is a boundary point of a double-point arc, where this boundary point is an interior point of $p(C)$.

Equivalently, a diagram of a singular link cobordism $C$ is called \textit{generic} if for any point in the image $p(C)$ there is a neighborhood in $\R^2 \times [0, 1]$ in which the image is either the intersection of the coordinate planes, or a branch point, or a crossing-vertex point.
Triple points and singular intersection points have neighborhoods in $p(C)$ which look like three transversely intersecting planes, while interior points of a double-point curve or an edge have  
neighborhoods in $p(C)$ which look like two transversely intersecting planes. 

Note that a generic diagram $p(C)$ of a singular link cobordism $C$ is a 2-dimenional CW-complex. The 0-cells are the vertices of the singular link diagrams in the boundary of $C$ and the 0-dimensional singular points in the interior of $C$ are: branch points, crossing-vertex points, triple points, and singular intersection points. The 1-cells of the diagram $p(C)$ are the edges and the double-point curves. There are four 2-cells meeting at an edge or at a double-point arc/curve of the diagram. We will sometimes refer to the 2-cells in a diagram as the \textit{facets} of the diagram.

From here on we will consider only generic diagrams for singular link cobordisms.

Movies for singular link cobordisms are obtained when a height function is fixed on the 3-space into which a given cobordism is projected. This is similar to the case of embedded surfaces (knotted surfaces) as studied in~\cite{CS}.

\begin{definition}
A projection map $\pi \co \R^2 \times [0, 1] \to [0, 1]$ is called a \textit{generic height function} for a singular link cobordism $C$ with (generic) diagram $p(C)$ if
\begin{enumerate}
\item $\pi_{\vert p(C)}$ has only non-degenerate critical points, where the set of critical points for the edge set is empty, and
\item critical points take distinct critical levels of $\pi$.
\end{enumerate}
\end{definition}
In the above definition we define critical points to include branch points, crossing-vertex points, triple points, and singular intersection points. The triple points and singular intersection points are of special type and we regard them as transverse intersections. In condition (2), we have that the critical points of facets and of double-point arcs/curves in a cobordism are at different critical levels.

\begin{definition}\label{def:generic height func}
A \textit{movie description} of a singular link cobordism $C$ is a (generic) diagram of $C$ with a fixed generic height function.
\end{definition}

In  a movie description of a cobordism $C$ we can cut its diagram $p(C)$ by planes that are perpendicular to the line that defines the height function on $\R^2 \times [0, 1]$. The planes $\R^2 \times \{t\}$ of constant $t \in [0, 1]$ can be taken at non-critical levels of the height direction; the image of the diagram $p(C)$ in the cross sectional planes are singular link diagrams, called \textit{stills}. Any diagram $p(C)$ gives rise to a movie and a movie description defines a diagram of a singular link cobordism.

Note that the time direction in the movie is given by the height function. If we restrict the movie to a small interval of time around a critical level we see two stills in the movie that differ either by an extended Reidemeister move (Reidemeister-type moves for singular link diagrams) or a Morse modification (a birth or death of a simple unknotted curve or a saddle point transformation). We call these changes \textit{elementary singular string interactions} (ESSIs for short) of movie descriptions.

Therefore, two consecutive stills in a movie description of a singular link cobordism differ by an ESSI. 
Moreover, two singular links are cobordant if their diagrams are related via a finite sequence of ESSIs.
This is the reason for imposing  the second part of condition (1) in Definition~\ref{def:generic height func}, namely that there are no critical points for the edge set. 
Equivalently, we do not include the \textit{bigon move} depicted in Figure~\ref{fig:bigon} among the ESSIs. 
\begin{figure}[ht]
\[  \raisebox{5pt}{\includegraphics[height=0.7in]{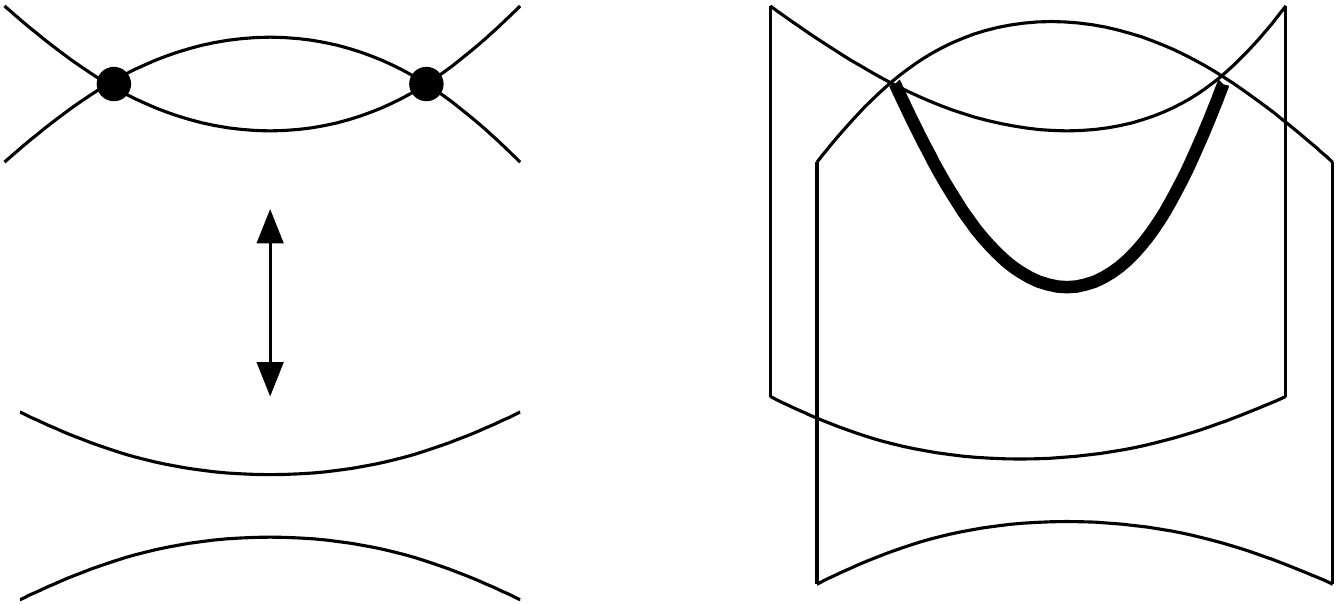}}  \]
\caption{The bigon move is not considered as an ESSI}
\label{fig:bigon}
\end{figure}

\subsection{Isotopies for singular link cobordisms}

We use movie descriptions to represent singular link cobordisms up to ambient isotopy fixed on the boundary. 

Two singular link cobordisms $C_1$ and $C_2$ in $\R^3\times[0, 1]$ are called \textit {ambient isotopic} if there exists an orientation preserving homeomorphism of $\R^3 \times [0,1]$ into itself taking $C_1$ to $C_2$ and keeping the boundary fixed.
It is clear that isotopic singular links are always cobordant. However, the converse is not true in general.

Our goal is to provide the set of movie moves for isotopic singular link cobordisms. Just as singular links contain classical knots and links as a subset, singular link cobordisms include knot and link codordisms in 4-space. Therefore, the moves for movies that relate isotopic singular link cobordisms must contain all of the fifteen Carter-Saito movie moves displayed in Figures~\ref{fig:Roseman} and \ref{fig:Carter-Saito} (see~\cite{CS}).
Note that we are using the namings for the Carter-Saito movies moves as given by Bar-Natan in~\cite{BN}. The movie moves in Figure~\ref{fig:Roseman} are the movie parametrizations of the Roseman moves to diagrams of isotopic knotted surfaces in 4-space (see~\cite{Ros}).

\begin{figure}[ht]
\[\put(27, 70){\fontsize{8}{8}$\text{MM1}$} 
\put(117, 70){\fontsize{8}{8}$\text{MM2}$}
\put(210, 70){\fontsize{8}{8}$\text{MM3}$} 
\put(303, 70){\fontsize{8}{8}$\text{MM4}$}
\raisebox{-5pt}{\includegraphics[height=1in, width = 1in]{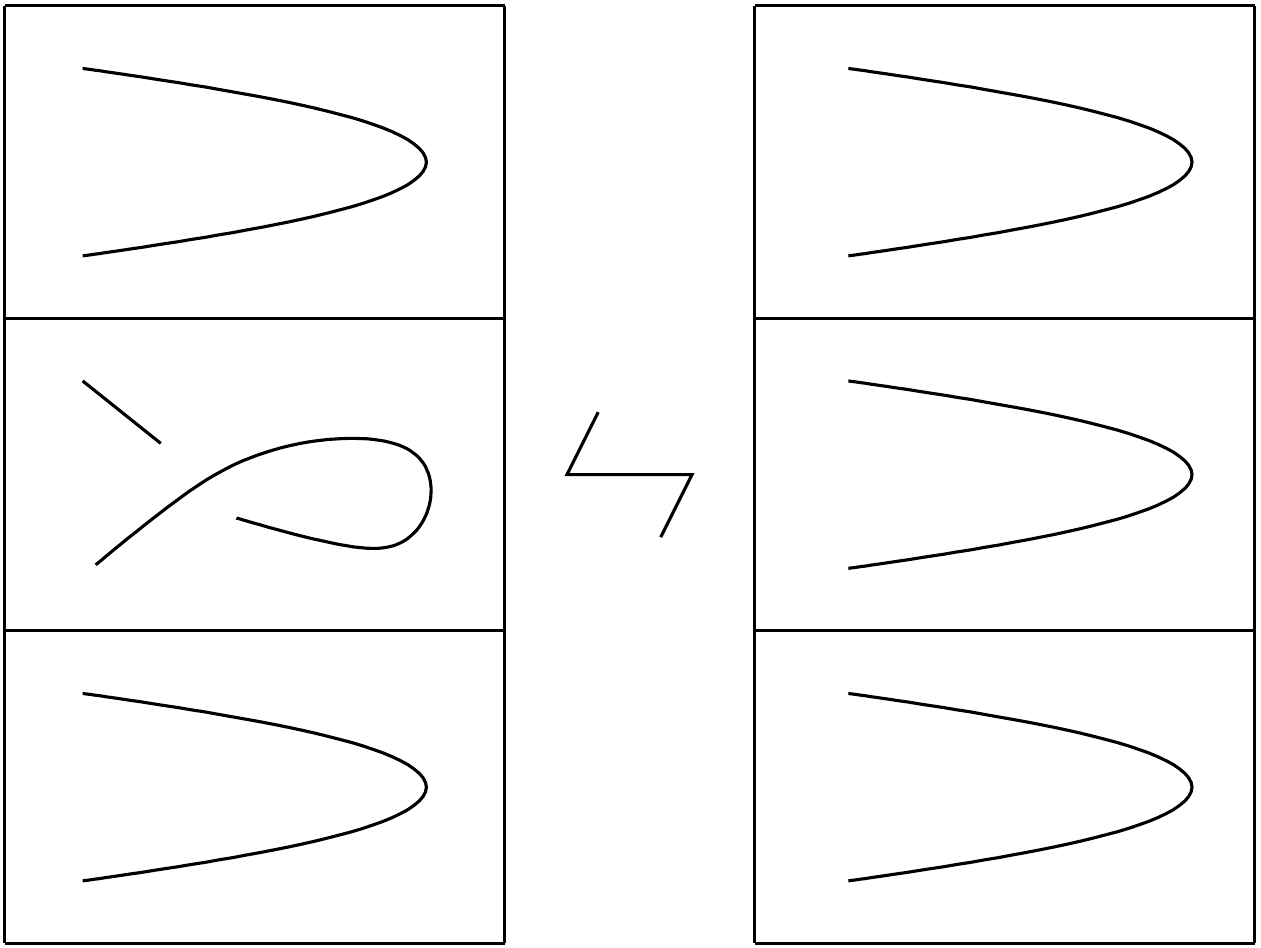}} \hspace{0.7cm} \raisebox{-5pt}{\includegraphics[height=1in, width = 1in]{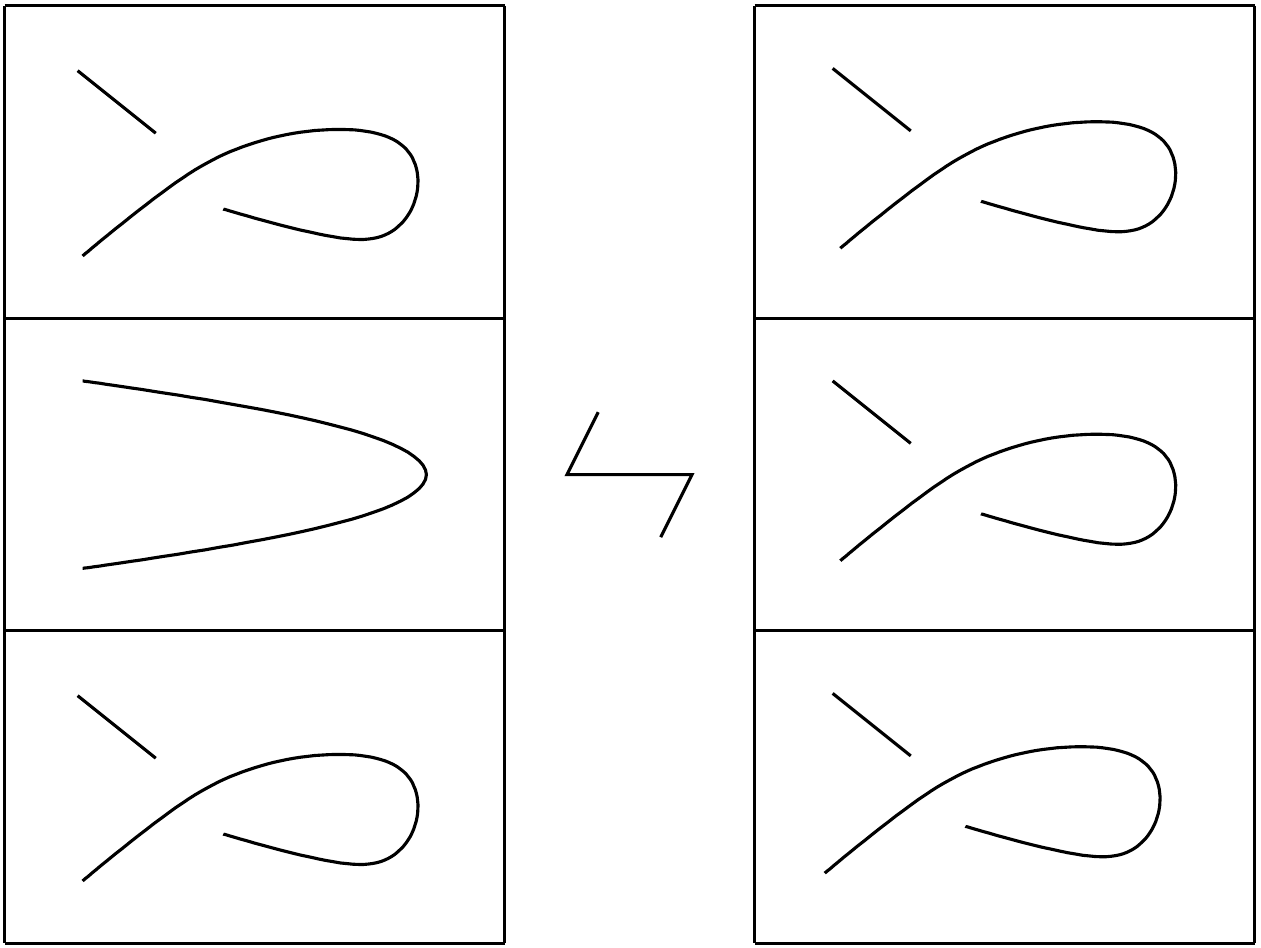}} \hspace{0.7cm} \raisebox{-5pt}{\includegraphics[height=1in, width = 1in]{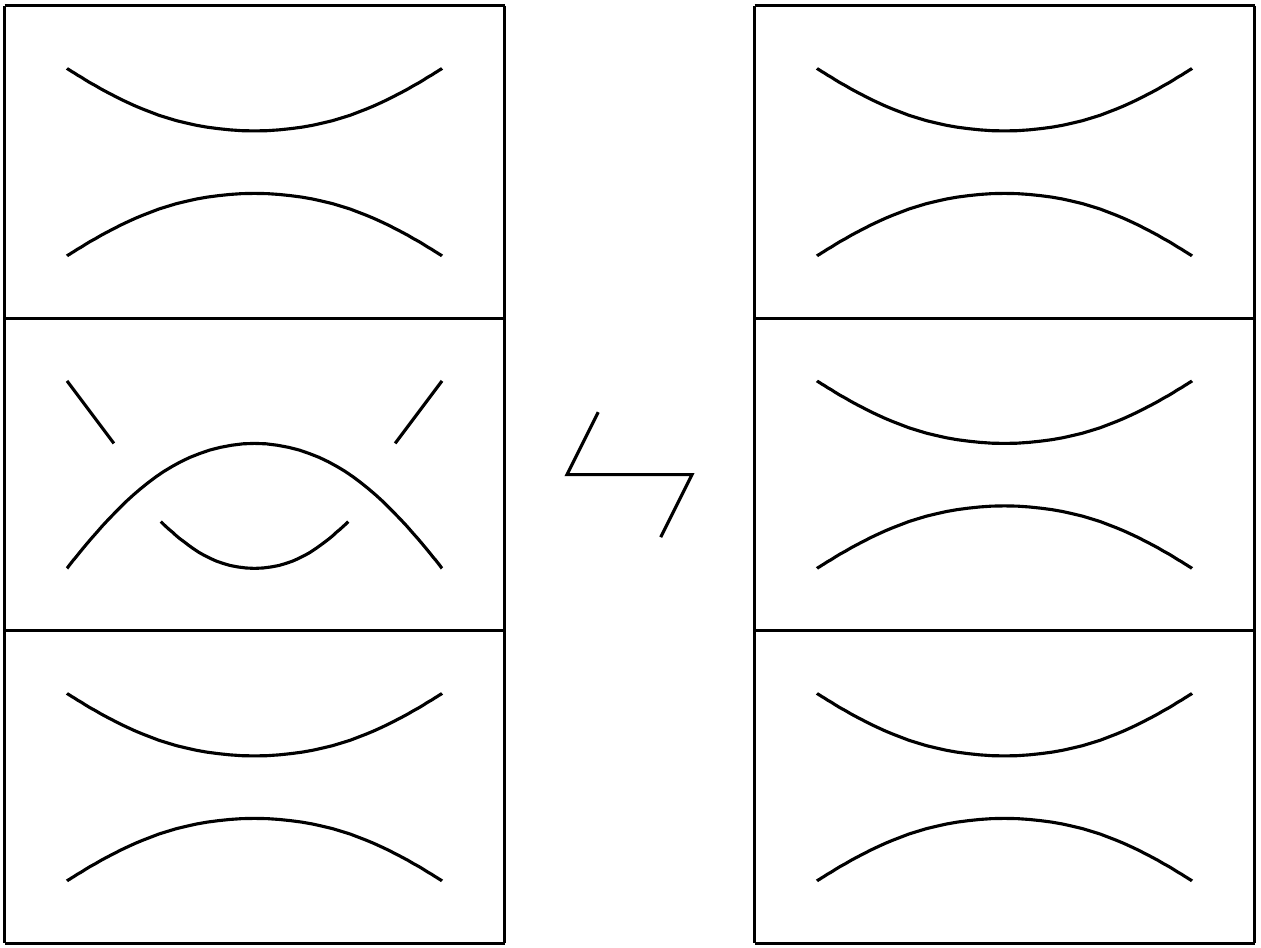}} \hspace{0.7cm} \raisebox{-5pt}{\includegraphics[height=1in, width = 1in]{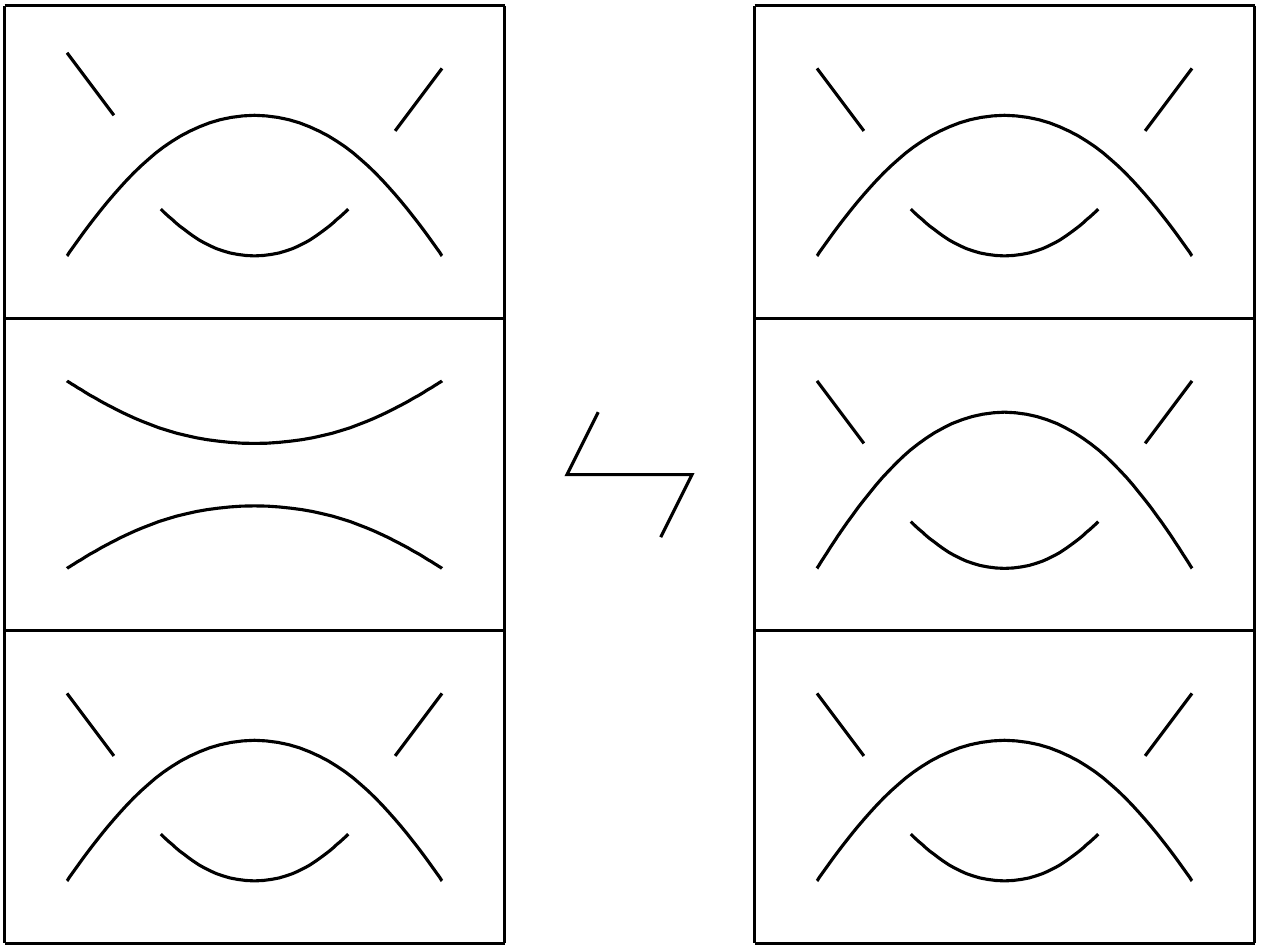}} \]

\[\put(30, -13){\fontsize{8}{8}$\text{MM5}$} 
\put(134, -13){\fontsize{8}{8}$\text{MM8}$}
\put(240, -13){\fontsize{8}{8}$\text{MM10}$}
 \raisebox{-5pt}{\includegraphics[height=1in, width = 1.1in]{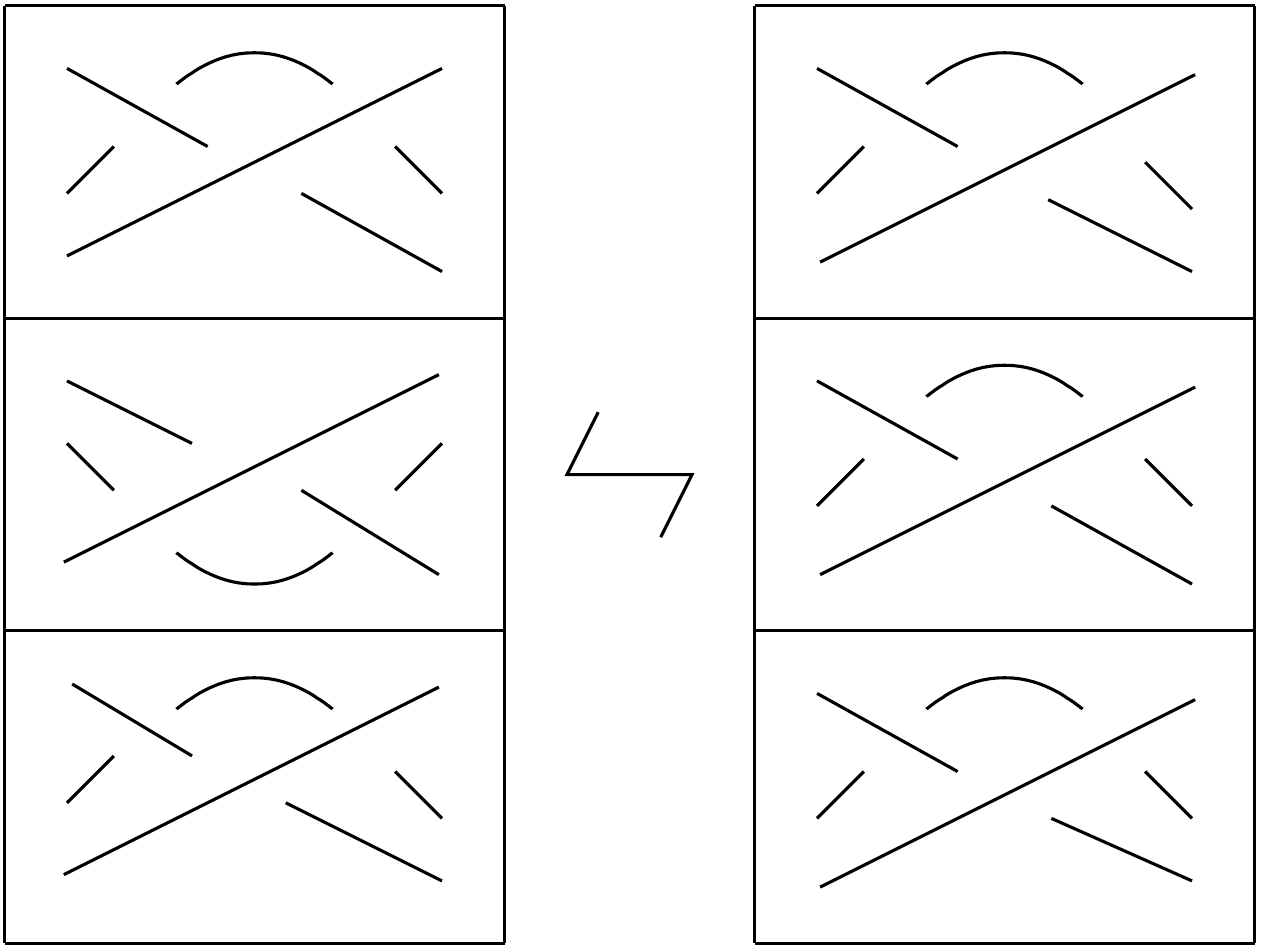}} \hspace{0.7cm} 
\raisebox{-5pt}{\includegraphics[height=1.8in, width = 1.2in]{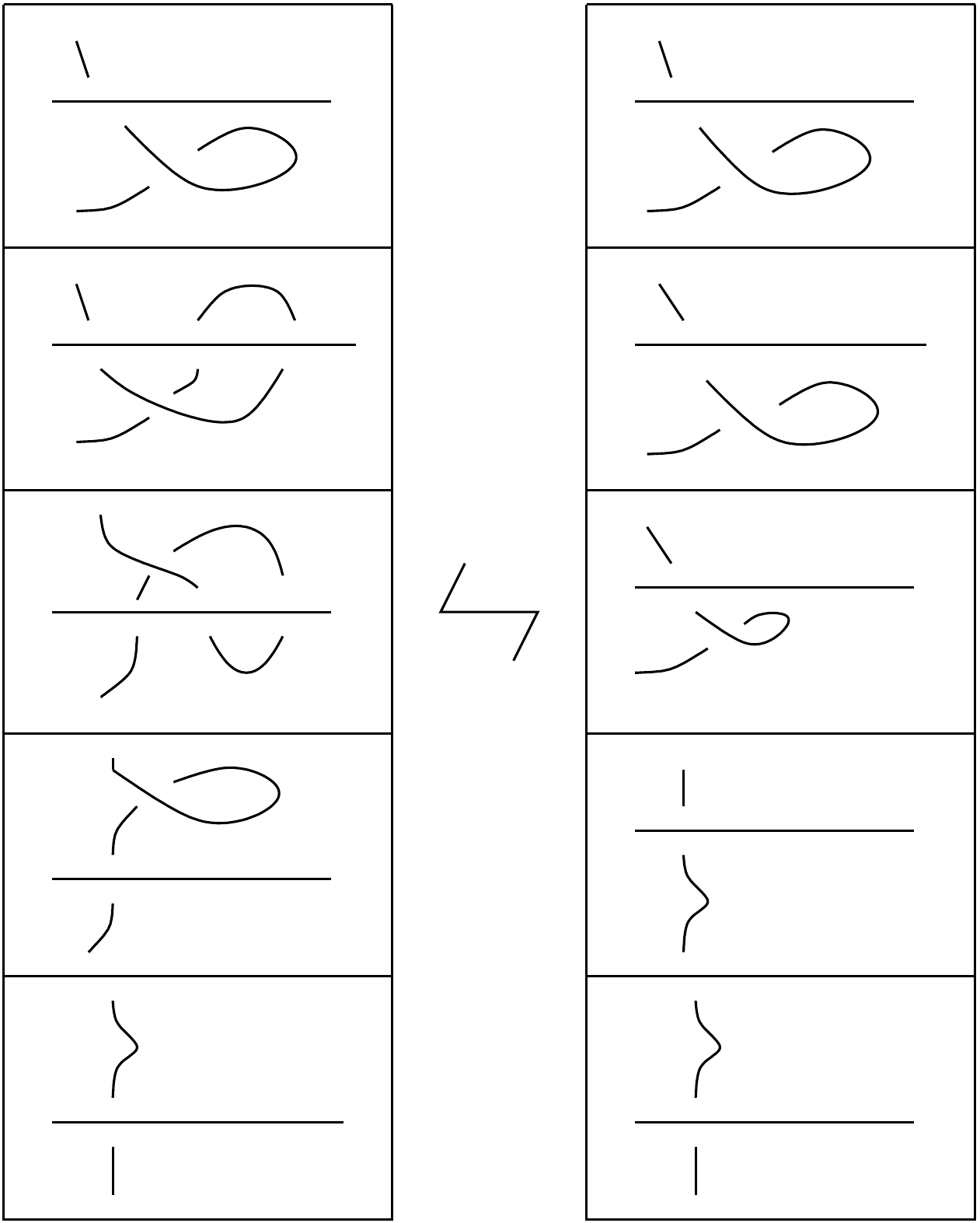}} \hspace{0.7cm} 
\raisebox{-5pt}{\includegraphics[height=1.8in, width = 1.3in]{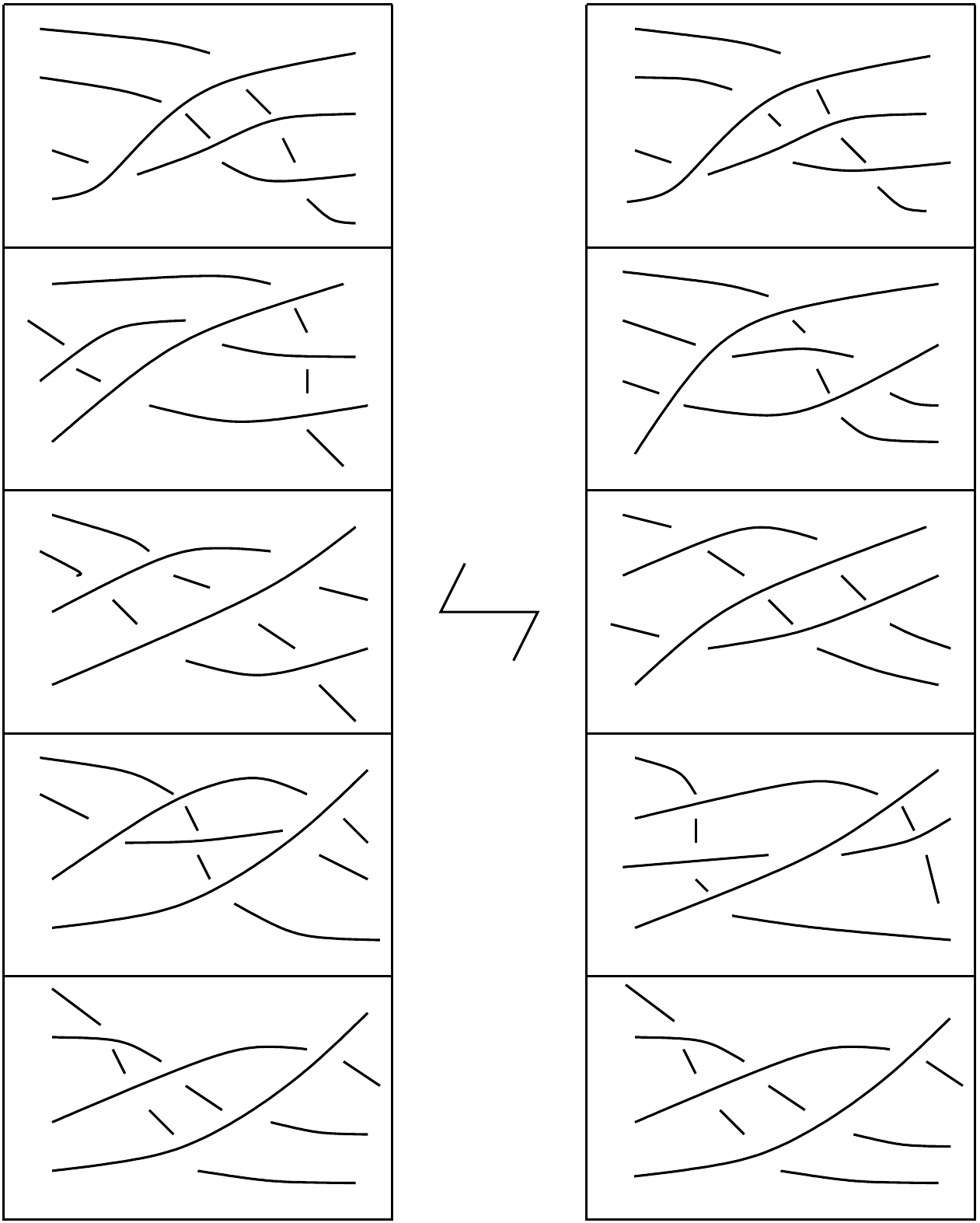}}
\]
\caption{Movie parametrizations of the Roseman moves}\label{fig:Roseman}
\end{figure}

\begin{figure}[ht]
\[\put(27, 70){\fontsize{8}{8}$\text{MM6}$} 
\put(117, 70){\fontsize{8}{8}$\text{MM7}$}
\put(210, 70){\fontsize{8}{8}$\text{MM9}$} 
\put(300, 70){\fontsize{8}{8}$\text{MM11}$} 
\raisebox{-5pt}{\includegraphics[height=1in, width = 1in]{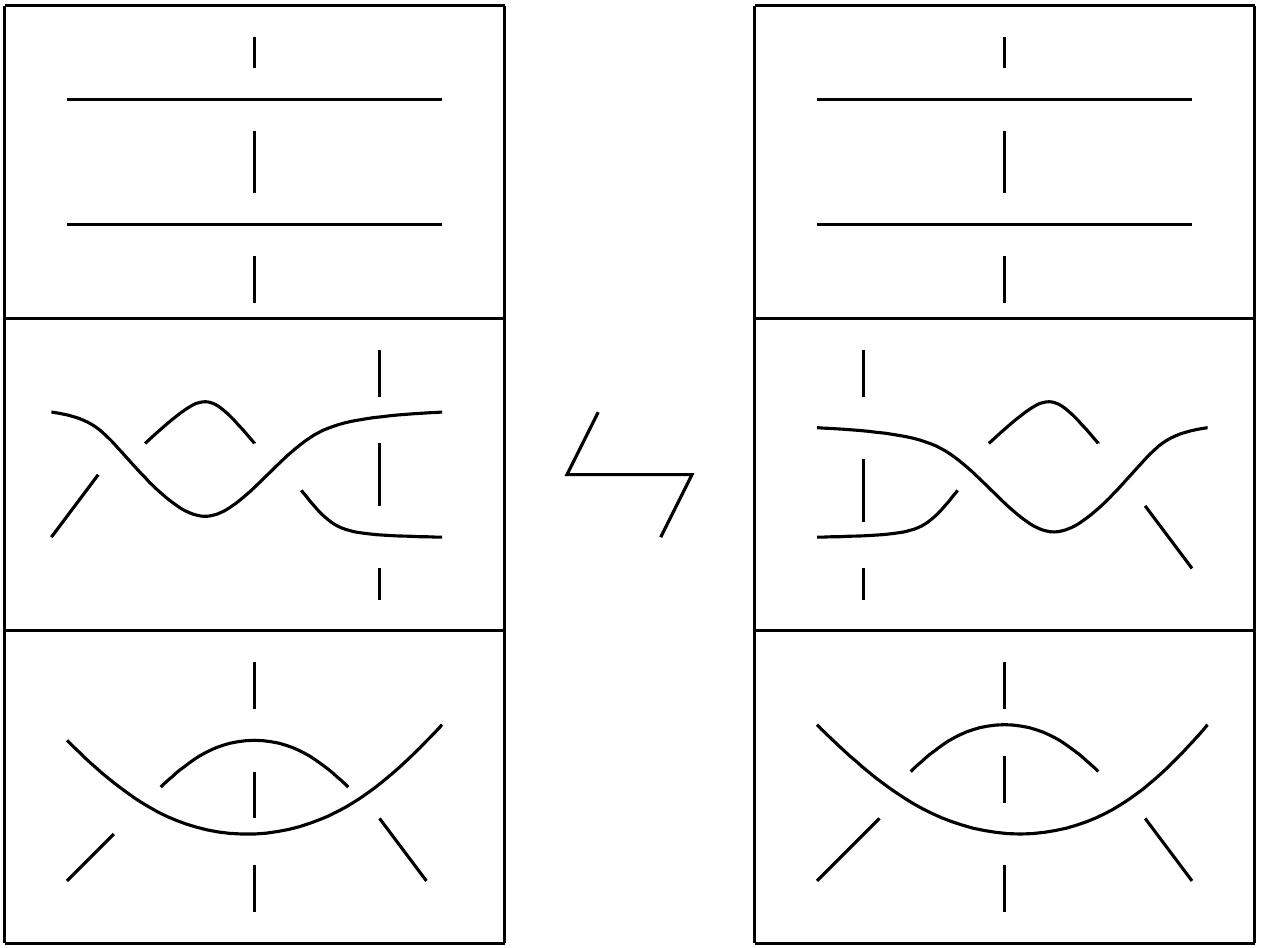}} \hspace{0.7cm} 
\raisebox{-5pt}{\includegraphics[height=1in, width = 1in]{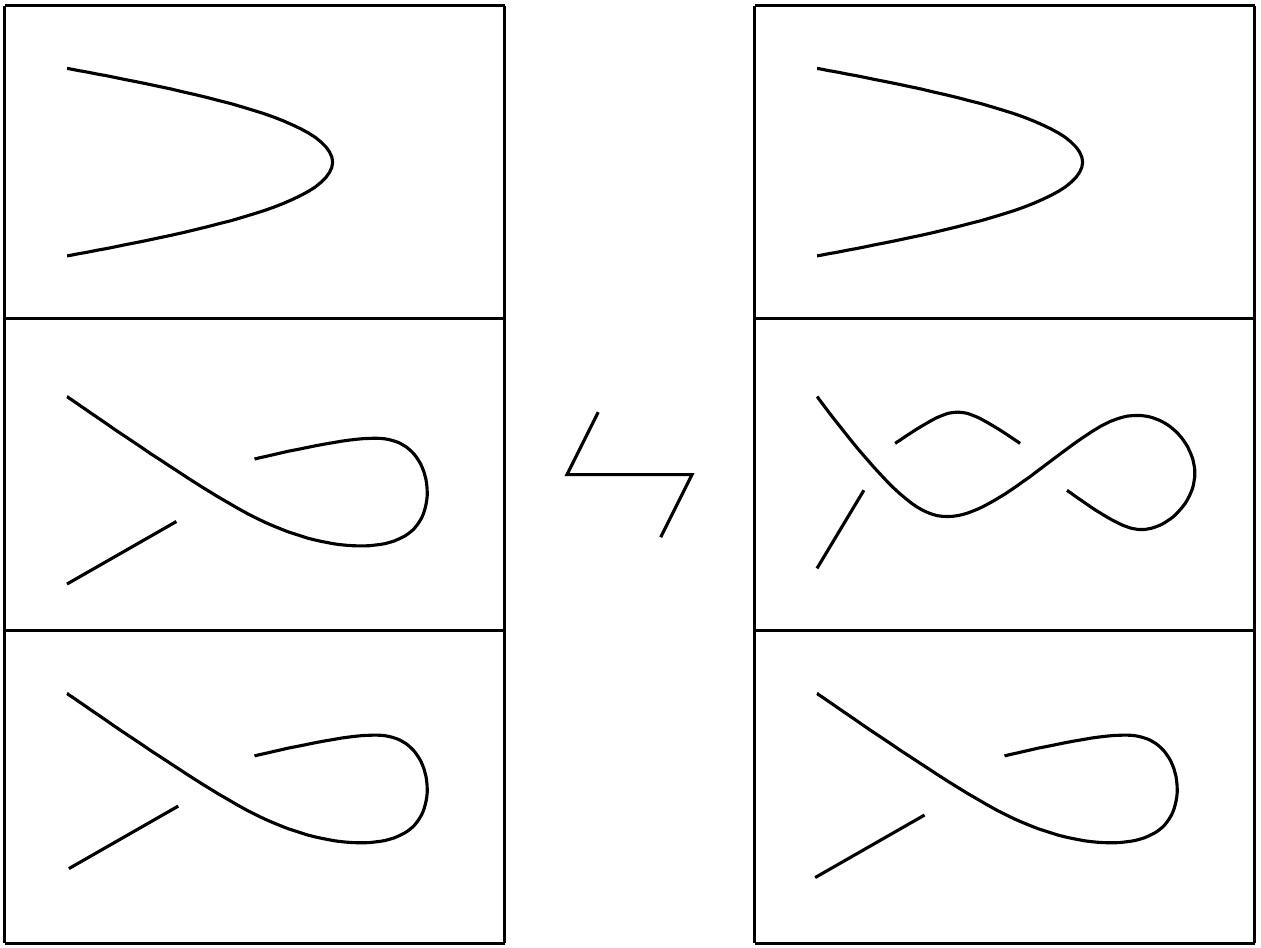}}  \hspace{0.7cm} 
\raisebox{-5pt}{\includegraphics[height=1in, width = 1in]{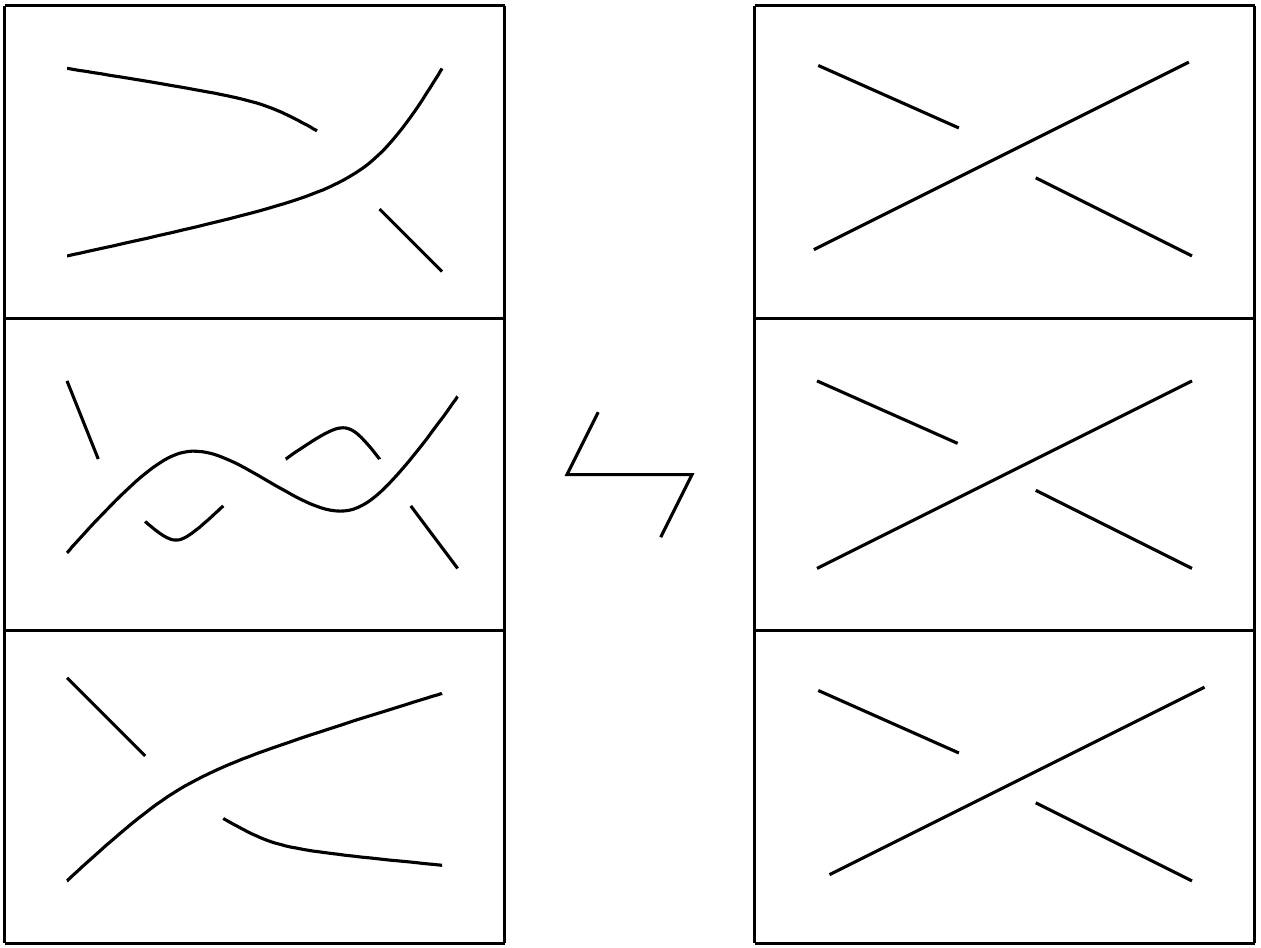}} \hspace{0.7cm}
\raisebox{-5pt}{\includegraphics[height=1in, width = 1in]{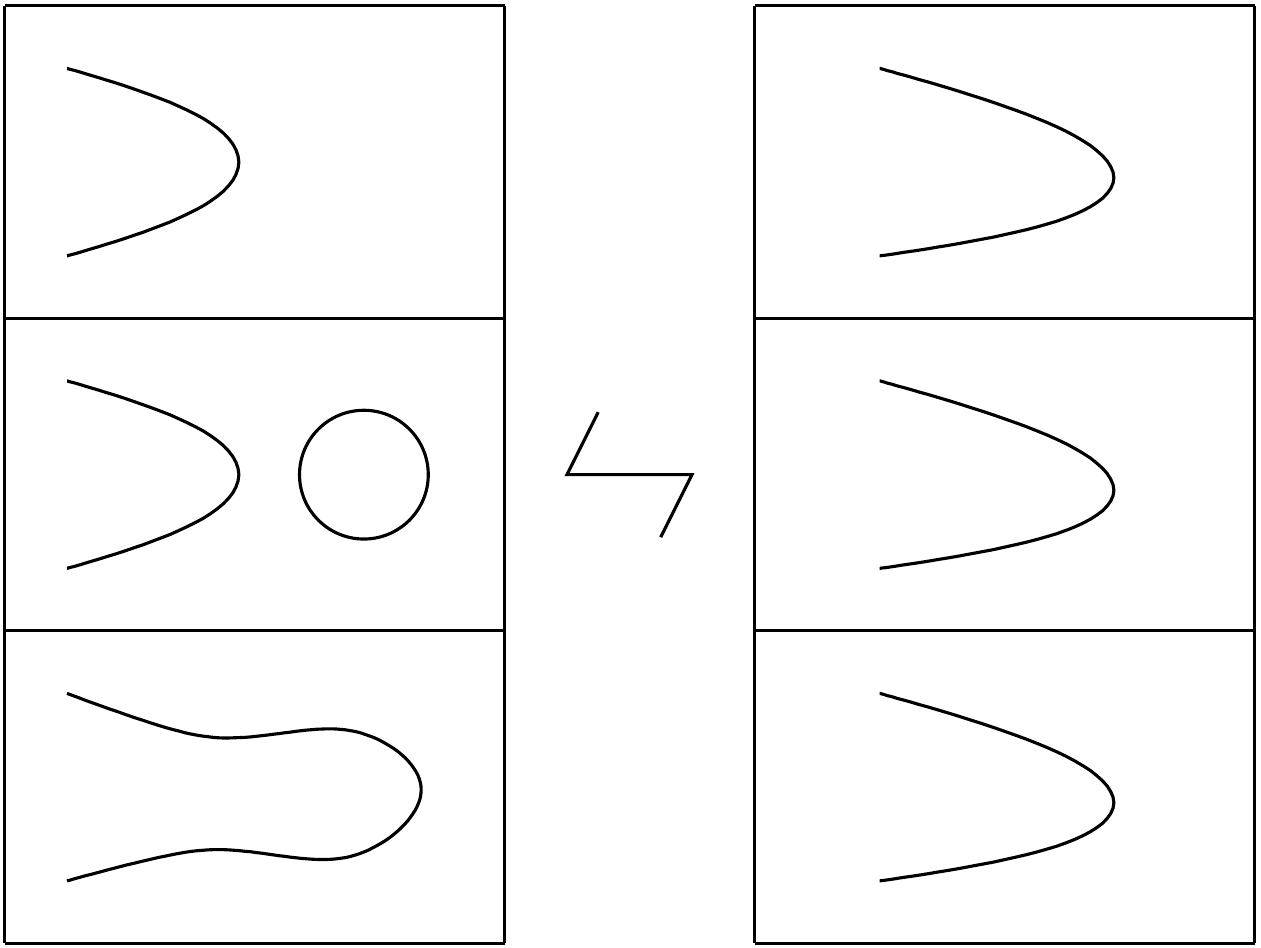}}\]

\[\put(24, -63){\fontsize{8}{8}$\text{MM12}$} 
\put(114, -63){\fontsize{8}{8}$\text{MM13}$}
\put(205, -63){\fontsize{8}{8}$\text{MM14}$}
\put(295, -63){\fontsize{8}{8}$\text{MM15}$}
\raisebox{-55pt}{\includegraphics[height=1.15in]{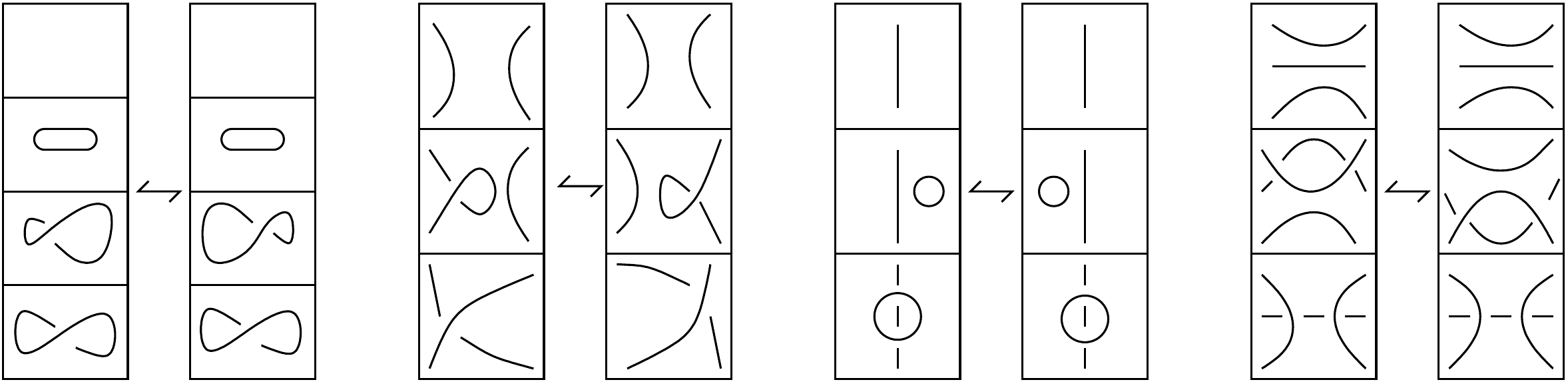}}\]
\caption{Carter-Saito augmented set of movie moves}\label{fig:Carter-Saito}
\end{figure}

Our additional set of movie moves for singular link cobordisms is given in Figure~\ref{fig:new-reidclips}. We have adopted the standard convention of listing movie moves with only one possible set of crossing data. Nonetheless, it is an easy exercise to obtain all versions of these movie moves.

\begin{figure}[ht]
\[\put(35, 60){\fontsize{8}{8}$\text{MM16}$}
\put(155, 60){\fontsize{8}{8}$\text{MM17}$}
\put(273, 60){\fontsize{8}{8}$\text{MM19}$}
\raisebox{-15pt}{\includegraphics[height=1in]{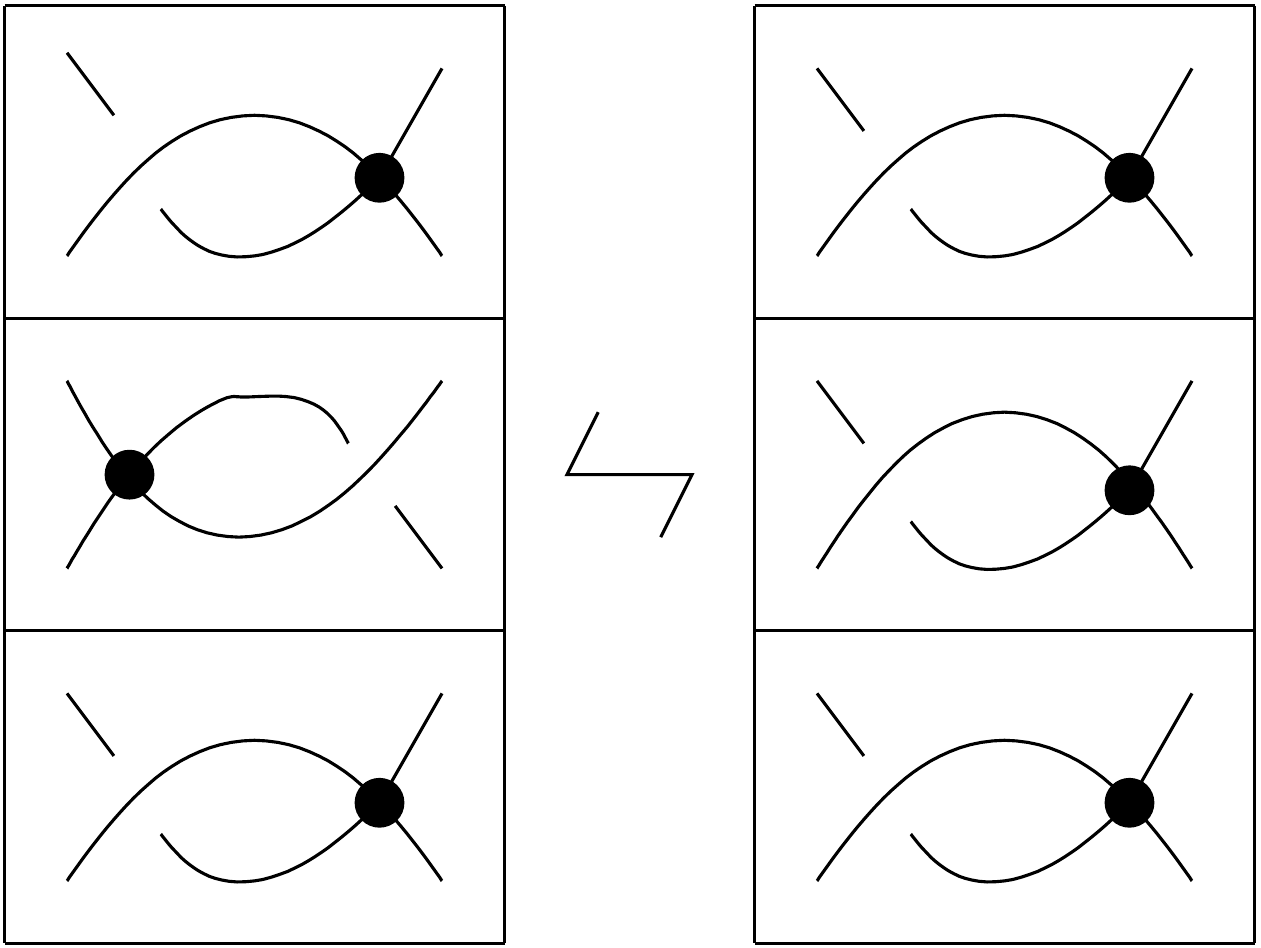}} \hspace{0.8cm} 
\raisebox{-15pt}{\includegraphics[height=1in]{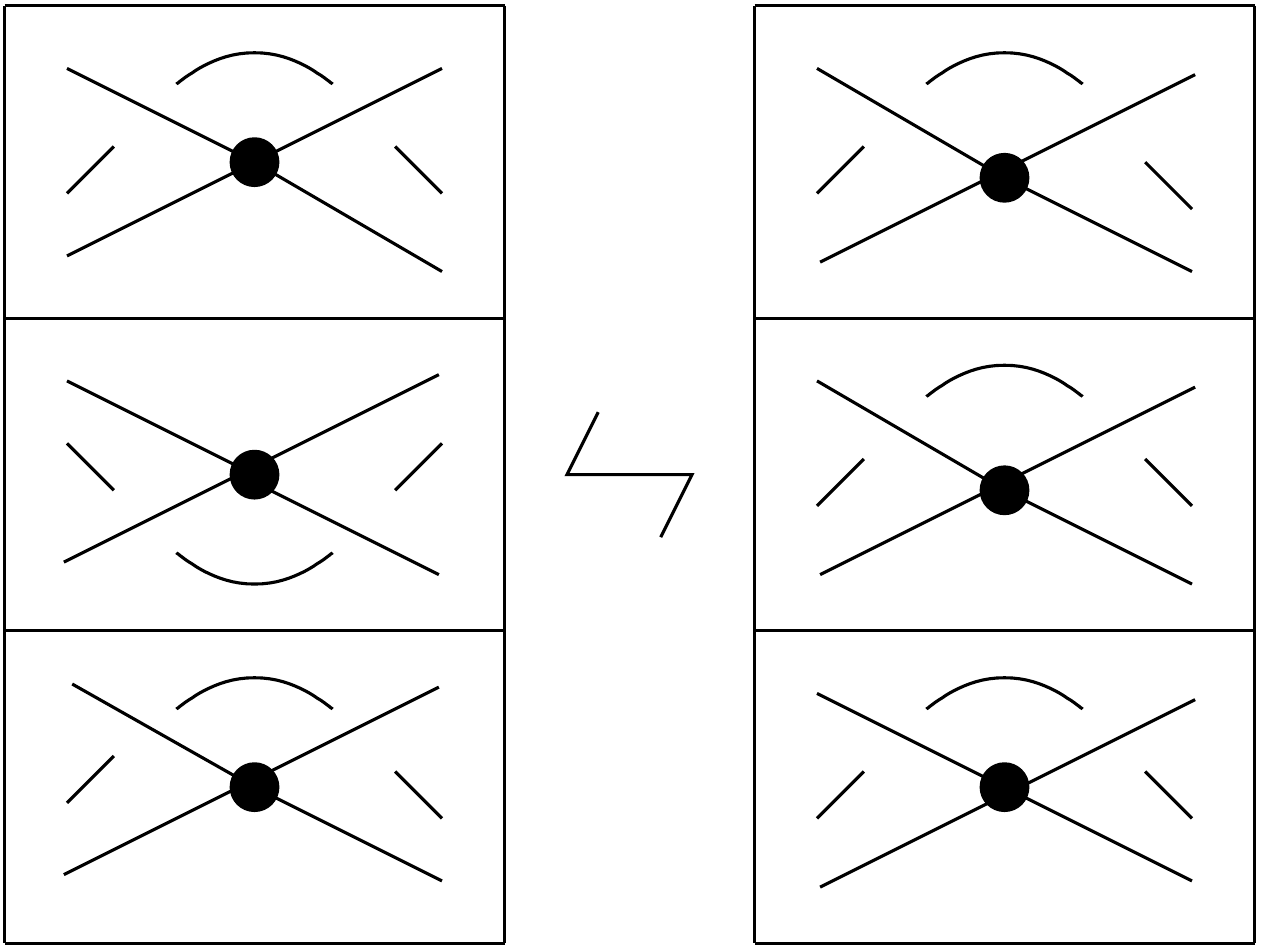}} \hspace{0.8cm} 
\raisebox{-15pt}{\includegraphics[height=1in]{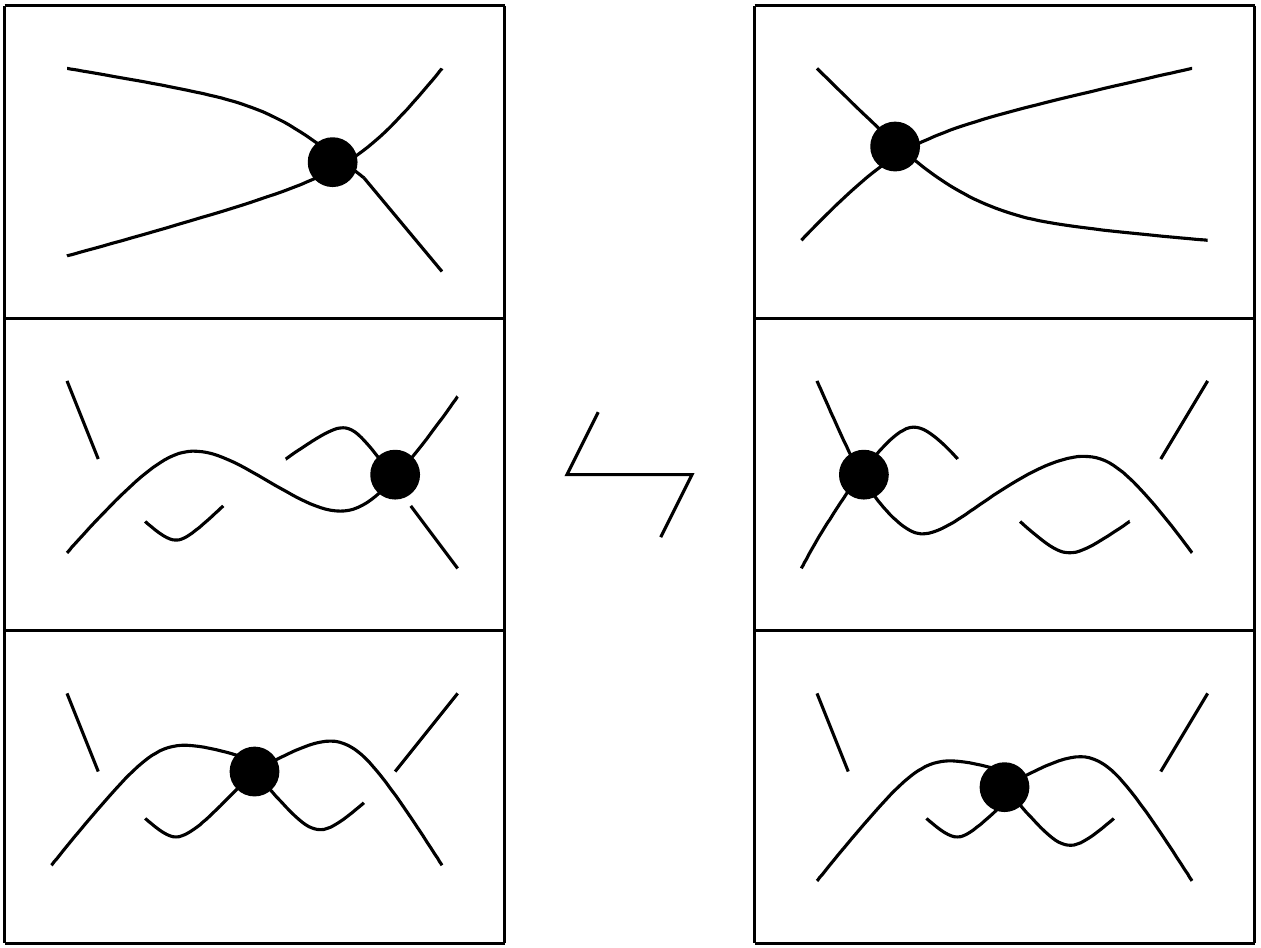}} \]

\[
\raisebox{-15pt}{\includegraphics[height=1.3in]{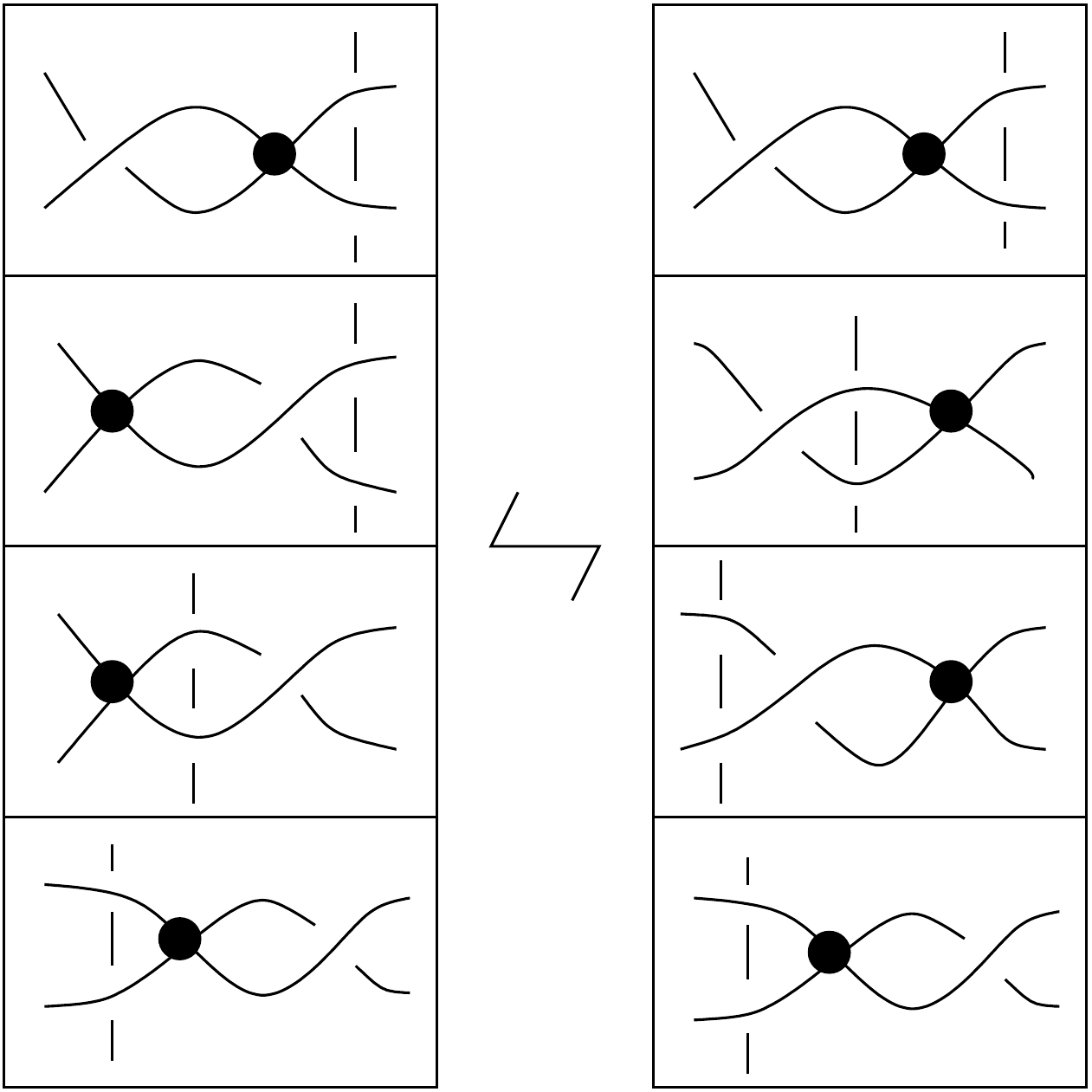}} \hspace{1.5cm}
\raisebox{-15pt}{\includegraphics[height=1in]{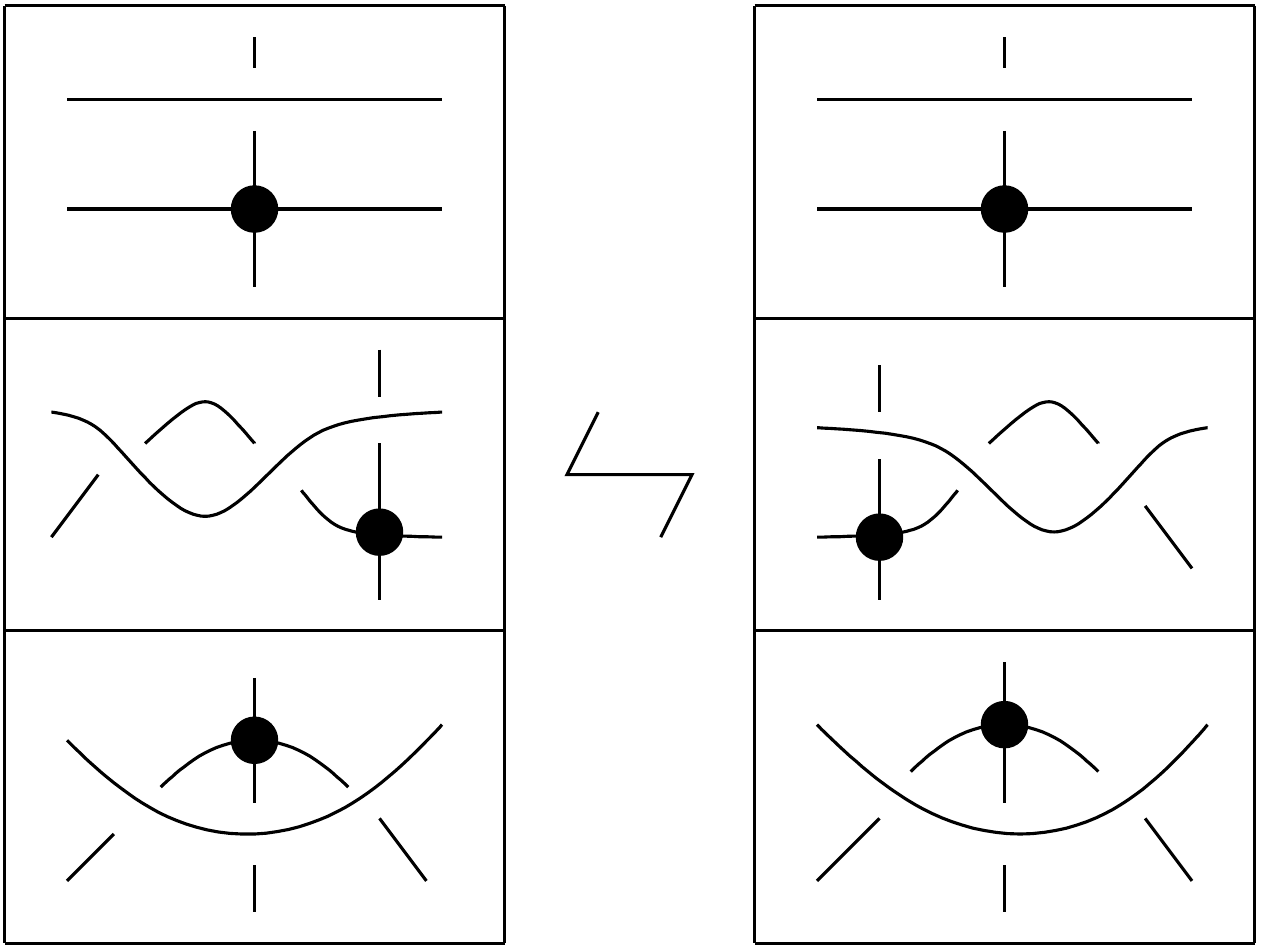}} 
\put(-200, -23){\fontsize{8}{8}$\text{MM18}$} 
\put(-60, -23){\fontsize{8}{8}$\text{MM20}$}
 \]

\[\hspace{-0.5cm}\raisebox{-15pt}{\includegraphics[width = 1.4in, angle=90]{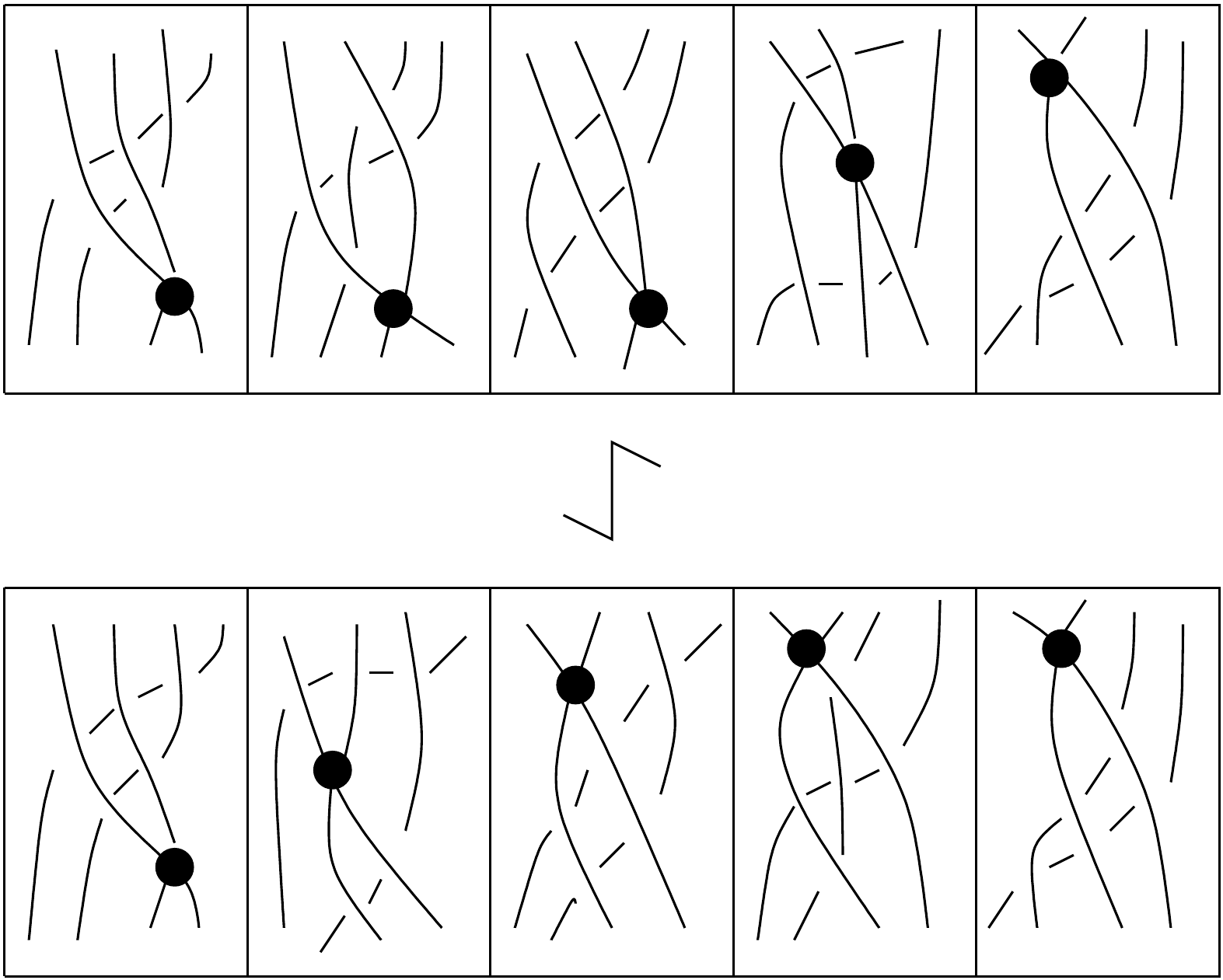}} 
\hspace{1.5cm} \raisebox{-15pt}{\includegraphics[width=1.4in, angle=90]{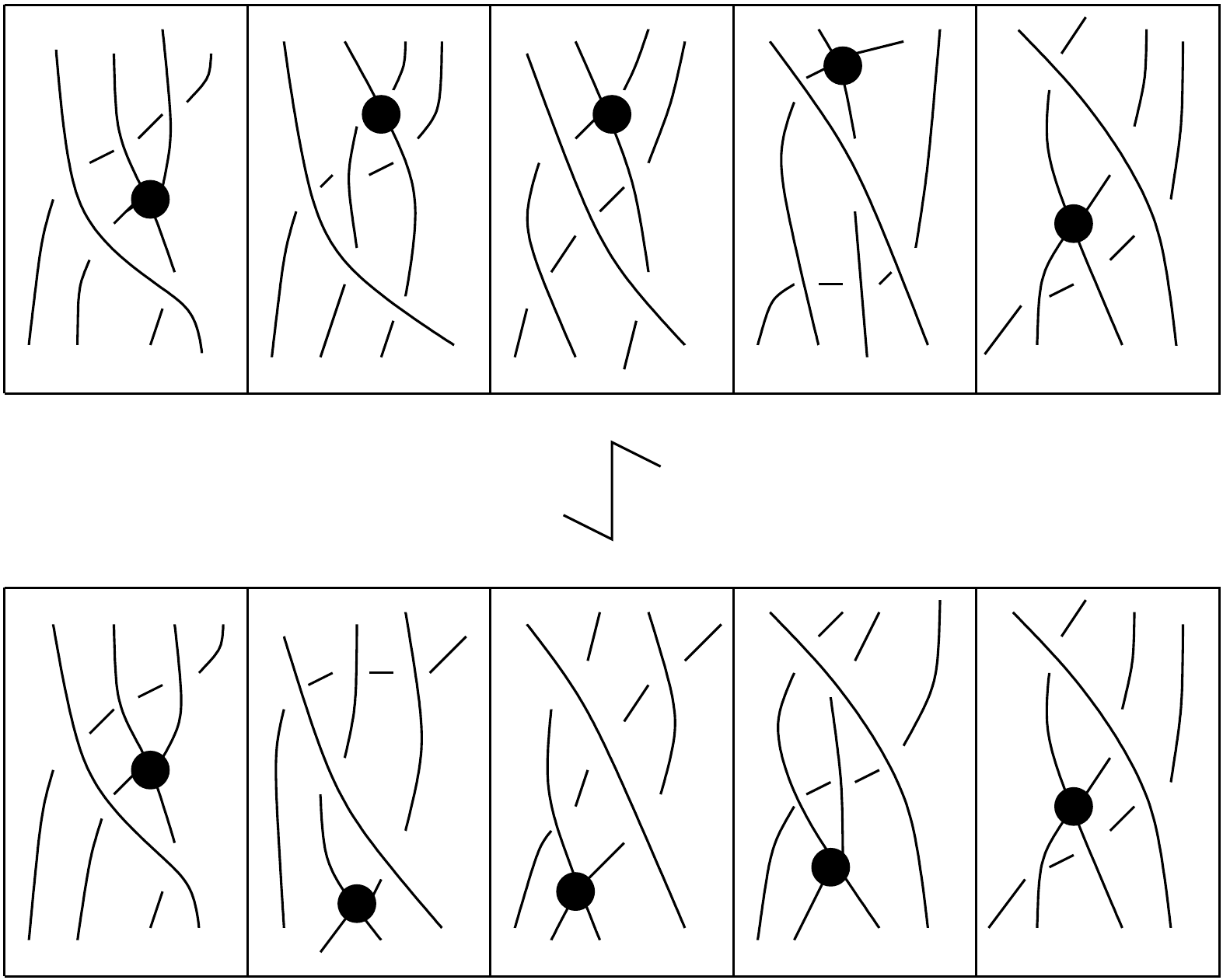}} 
\put(-243, -23){\fontsize{8}{8}$\text{MM21}$}
\put(-75, -23){\fontsize{8}{8}$\text{MM22}$}
\]

\[ \hspace{1.5cm} \raisebox{-15pt}{\includegraphics[width=1.4in, angle = 90]{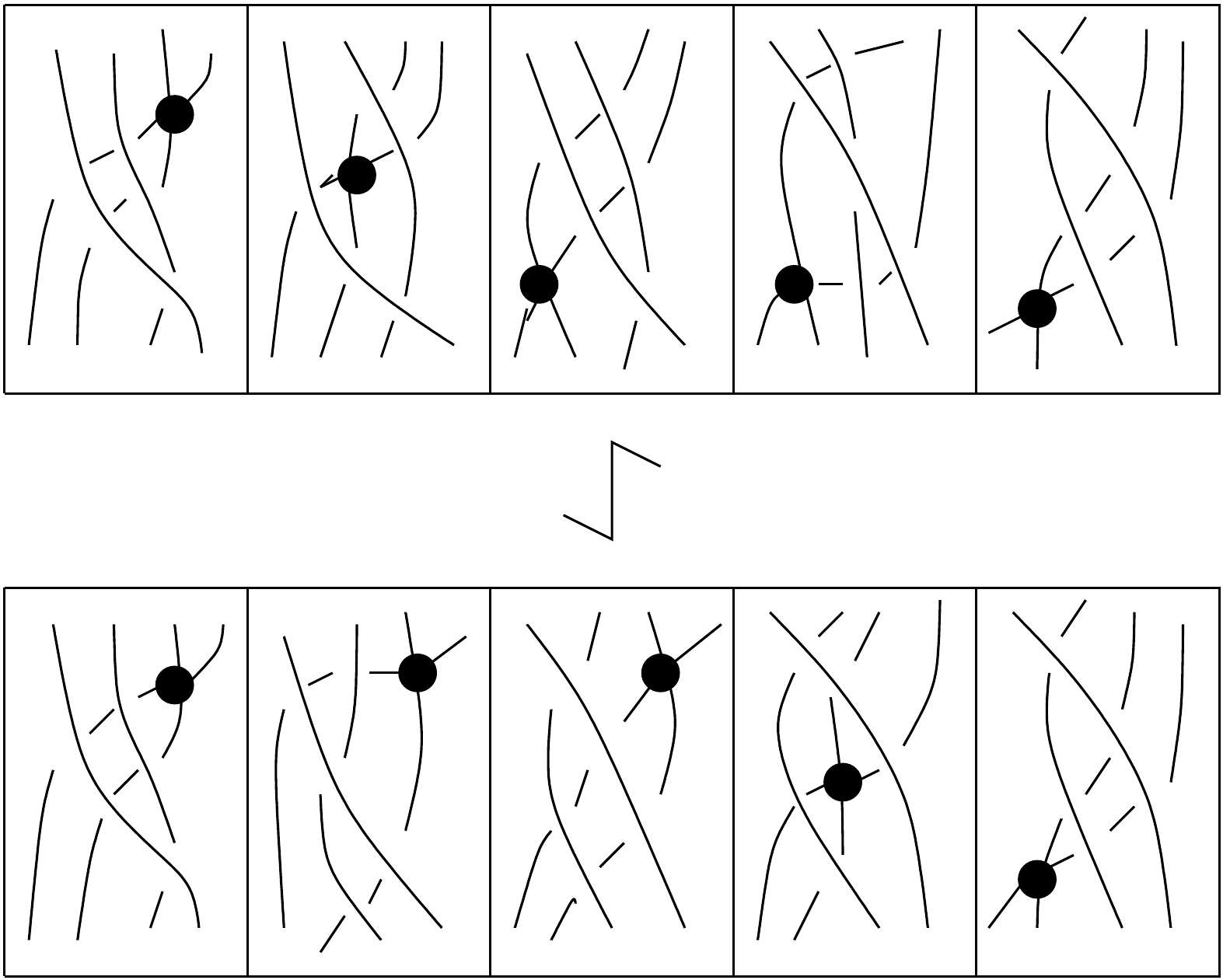}} \hspace{1.5cm}
\raisebox{-15pt}{\includegraphics[width=1.4in, angle=90]{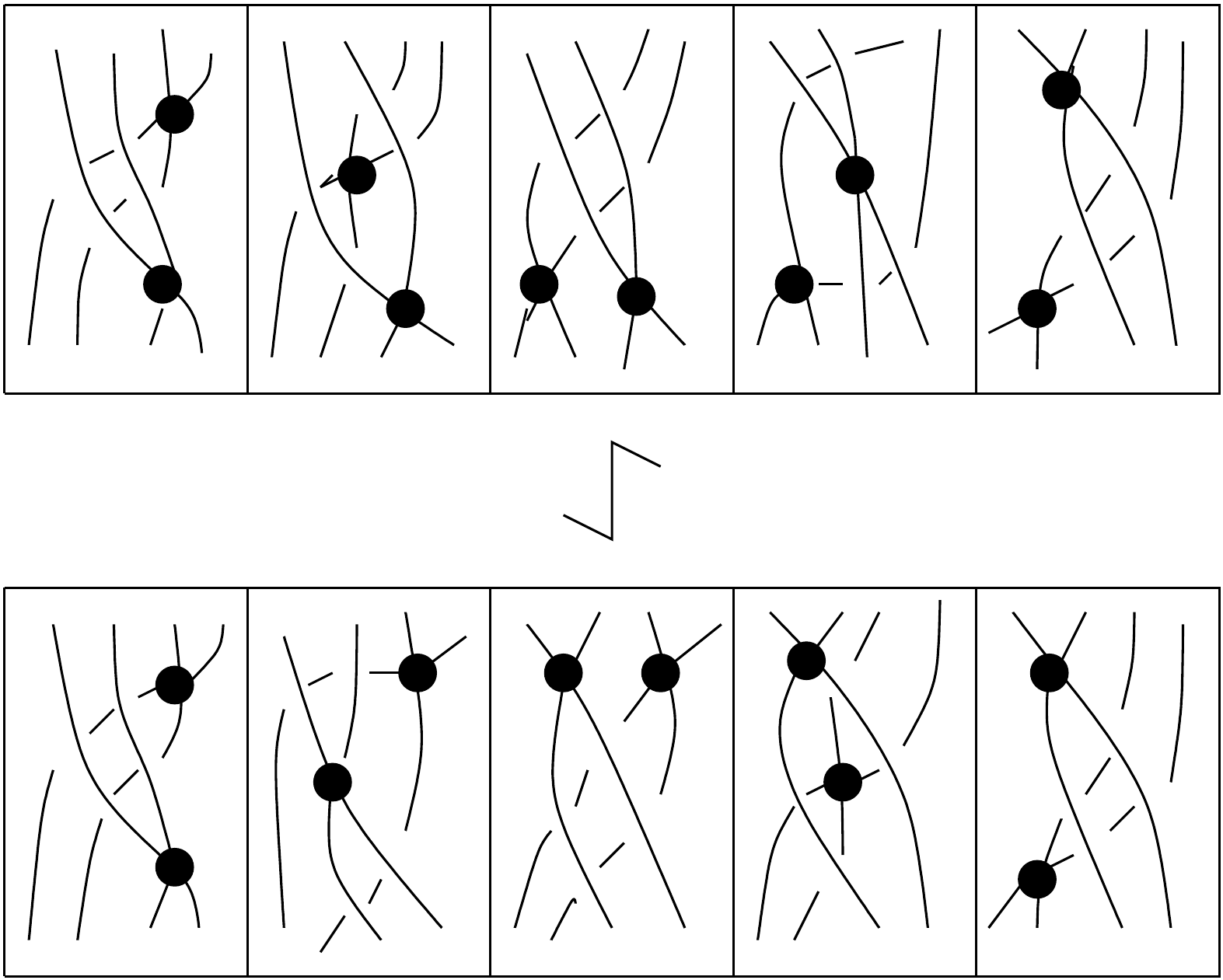}}
\put(-243, -23){\fontsize{8}{8}$\text{MM23}$}
\put(-75, -23){\fontsize{8}{8}$\text{MM24}$}
\]
\caption{Additional movie moves for singular link cobordisms}\label{fig:new-reidclips}
\end{figure}

It is important to note that in our Movie-Move Theorem (Theorem~\ref{thm:movie-moves}) we do not fix a height function on each still of a movie.

The idea of the proof for Theorem~\ref{thm:movie-moves} is to analyze an isotopy of a singular link cobordism as a finite sequence of changes, where each change is one of the finitely many possible changes. We want to ensure that no two changes occur simultaneously, which is warranted by a small perturbation of the isotopy. 

In showing that the list of our movie moves is sufficient to relate diagrams of isotopic singular link cobordisms, we will use a similar approach as in our sketch of the sufficiency of the extended Reidemeister moves to relate diagrams of isotopic singular links. We will study the projection of an isotopy on the retinal plane. As the isotopy passes through a singular projection, the types of possible changes involve fold sets and critical points of the multiple point strata including transverse intersections (in the direction of the isotopy).

Each of the movie moves will correspond to:
\begin{enumerate}
\item[(I)] non-degenerate critical points of the multiple point strata (there are seven such movie moves)--where the isotopy direction provides a height function, or to  
\item[(II)] transverse intersections between two strata (there are seven such movie moves), or to 
\item[(III)] changes in folds and in the relative position of the multiple point strata and folds (there are ten such movie moves). 
\end{enumerate}

We remark that the movie moves of types (I) and (II) above are movie descriptions of the Roseman-type moves for diagrams of isotopic singular link cobordisms. In addition, the movie moves of type (I) correspond to invertibility of each of the extended Reidemeister moves for singular link diagrams. Also, the movies moves of type (III) do not affect the topology of the image.

The 0-dimensional singular points on a generic diagram of a singular link cobordism are: 
\begin{enumerate}
\item[(1.a)] the branch points and crossing-vertex points that result from an $R1$ and $R5$ move, respectively,
\item[(1.b)] triple points and singular intersection points that result from an $R3$ and $R4$ move, respectively.
\end{enumerate}

The 1-dimensional singular sets of a generic diagram of a singular link cobordism consist of: 
\begin{enumerate}
\item[(2.a)] the double-point arcs produced by self-intersections of the singular link cobordism, or
\item[(2.b)] the edge set of the cobordism.
\end{enumerate}
Any facet of a generic diagram is 2-dimensional and therefore, it is non-singular.

During an isotopy, a branch point traces a 1-dimensional set, which we call a \textit{branch-point set}. Similarly, a crossing-vertex point, triple point, and singular intersection point trace 1-dimensional sets during an isotopy, and we refer to these sets as a \textit{crossing-vertex set, triple-point set, and singular intersection set}, respectively. In addition, the 1-dimensional singular sets of a generic diagram evolve in space-time to 2-dimensional sets. To find the analogue of the Roseman moves for singular link cobordisms we need to examine (I) the critical points of the resulting 1-dimensional and 2-dimensional sets, and (II) transverse intersections between two strata. Having fixed a height function on the 3-space into which the cobordism is projected, we need to examine additional changes that  involve folds and cusps (giving rise to the movies moves of type (III)), where these changes do not alter the topology of the projection. 

It is worth noting the analogy between our approach to isotopies of singular link cobordisms and the discussion in Section~\ref{sec:sing-links} about isotopies of singular links.

\subsection{Proof of Theorem~\ref{thm:movie-moves}} In this section we prove our Movie-Move Theorem. Here, when we illustrate a movie move we also depict its corresponding move on diagrams for singular link cobordisms.

Let $D_0$ and $D_1$ be generic diagrams (projections) in $\R^2 \times [0, 1]$ of singular link cobordisms, each given as a sequence of ESSIs. It is easy to see that if $D_0$ and $D_1$ differ by any of the movie moves depicted in Figures~\ref{fig:Roseman}, \ref{fig:Carter-Saito}, and \ref{fig:new-reidclips}, or by interchanges of the levels of distant critical points, then they correspond to isotopic cobordisms. 

Now suppose that $D_0$ and $D_1$ represent isotopic singular link cobordisms. We need to show that $D_1$ is obtained from $D_0$ by a finite sequence of local moves taken from those illustrated in Figures~\ref{fig:Roseman}, \ref{fig:Carter-Saito}, and \ref{fig:new-reidclips}, or by exchanging the order in which ESSIs occur when they occur in disjoint neighborhoods.

The isotopy between $D_0$ and $D_1$ provides a map $\phi \co C \times I \to (\R^2 \times [0, 1]) \times I$, where $I = [0, 1]$ and we need to examine the singularities of this map. We may assume (by allowing the isotopy to be slightly modified, if necessary) that this map has generic singularities on the singular sets of the projection, and that  these singularities lie at different time coordinates, $t \in I$. In additions, by compactness, there will be finitely many singularities. Considering the projection map $q \co (\R^2 \times [0, 1]) \times I \to I$, we may assume, without loss of generality, that $q$ is Morse on $\phi(C \times [0,1])$; in particular, we may assume that the critical points for the singular strata of the isotopy are non-degenerate and that the intersections (in the 4-dimensional space) between strata of the isotopy are generic and transverse.

(I) We start by analyzing the generic critical points for the 1-dimensional and 2-dimensional singular sets of the isotopy. 
The 1-dimensional singular sets (of type (1.a) and (1.b)) are the images under the isotopy of the branch points, crossing-vertex points, triple points, and singular intersection points of the diagram $D_0$.

The Reidemeister move $R1$ has non-symmetric sides and, consequently, there are two types of movie moves corresponding to the non-degenerate critical points of the branch-point sets in the isotopy direction. The movie move $MM1$ corresponds to a minimum (or maximum) for the brach-point set which is a minimum (or maximum) for the 2-dimensional set traced by the double-point arc; the pair of branch points that are created/annihilated during the move $MM1$ are the boundary points of the above mentioned double-point arc (self-intersection arc).
On the other hand, the move $MM2$ corresponds to a minimum (or maximum) for the branch-point set which is a maximum (or minimum) for the 2-dimensional set traced by the double-point arc.
\[ \raisebox{5pt}{\includegraphics[height=0.6in]{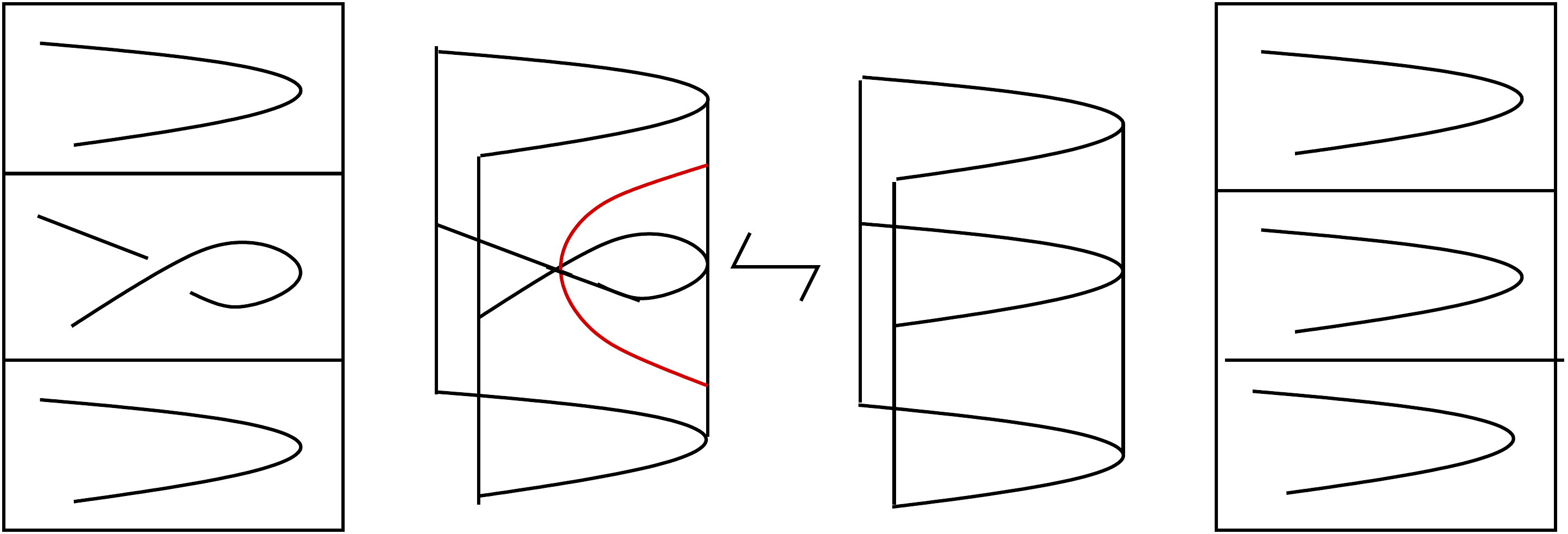}} \hspace{2cm} \raisebox{5pt}{\includegraphics[height=0.6in]{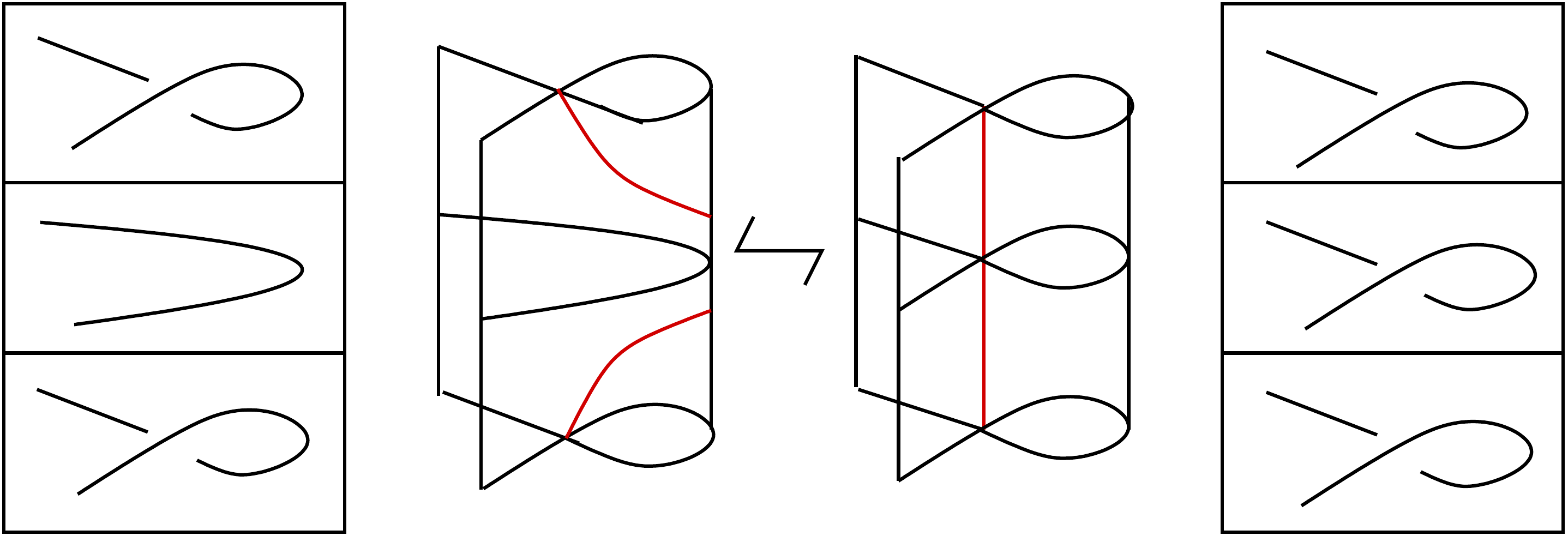}}  
\put(-333, 25){\fontsize{8}{8}$\text{MM1}$} 
\put(-153, 25){\fontsize{8}{8}$\text{MM2}$} \]

The twisted crossing-vertex move $R5$ has symmetric sides, and therefore there is only one movie move corresponding to the invertible behavior of this Reidemeister-type move, namely the movie move $MM16$.  This movie move corresponds to a non-degenerate critical point of the 1-dimensional crossing-vertex point set in the isotopy direction. In particular, the move corresponds to the annihilation/creation of a pair of crossing-vertex points.
\[ \raisebox{5pt}{\includegraphics[height=0.65in]{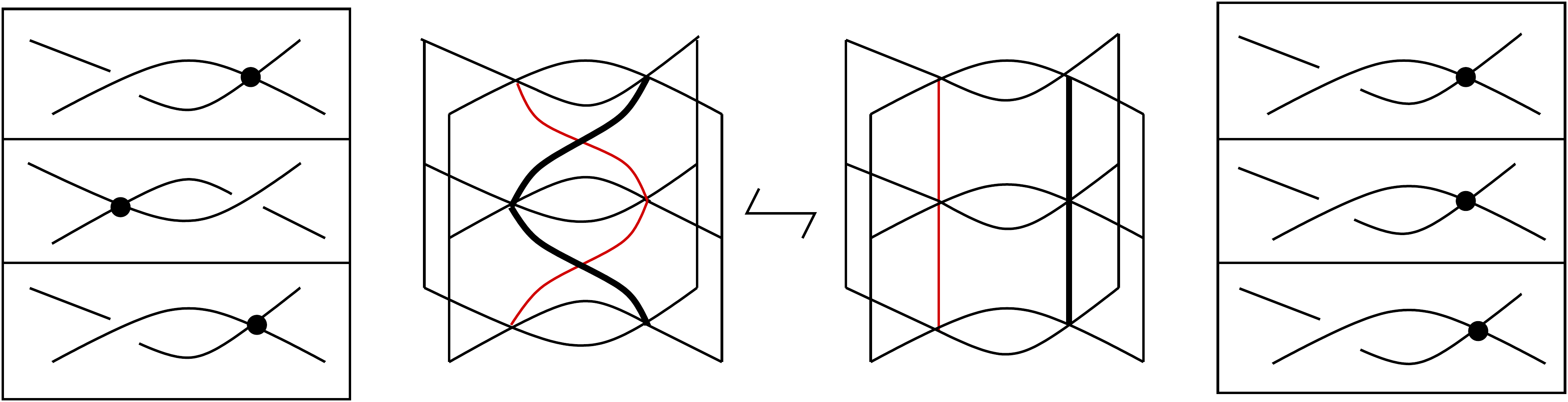}}
\put(-225, 28){\fontsize{8}{8}$\text{MM16}$} \]

The movie moves described so far, $MM1, MM2$ and $MM16$, correspond to critical points for the singular sets of type (1.a). 

We examine now the critical points for the singular sets of type (1.b). There is an optimum (maximum or minimum) point on a triple-point set which corresponds to the annihilation or creation of a pair of triple points during an isotopy. The movie move reflecting this behavior is the move $MM5$ showed below, together with the corresponding move on diagrams:

\[ \raisebox{5pt}{\includegraphics[height=0.8in]{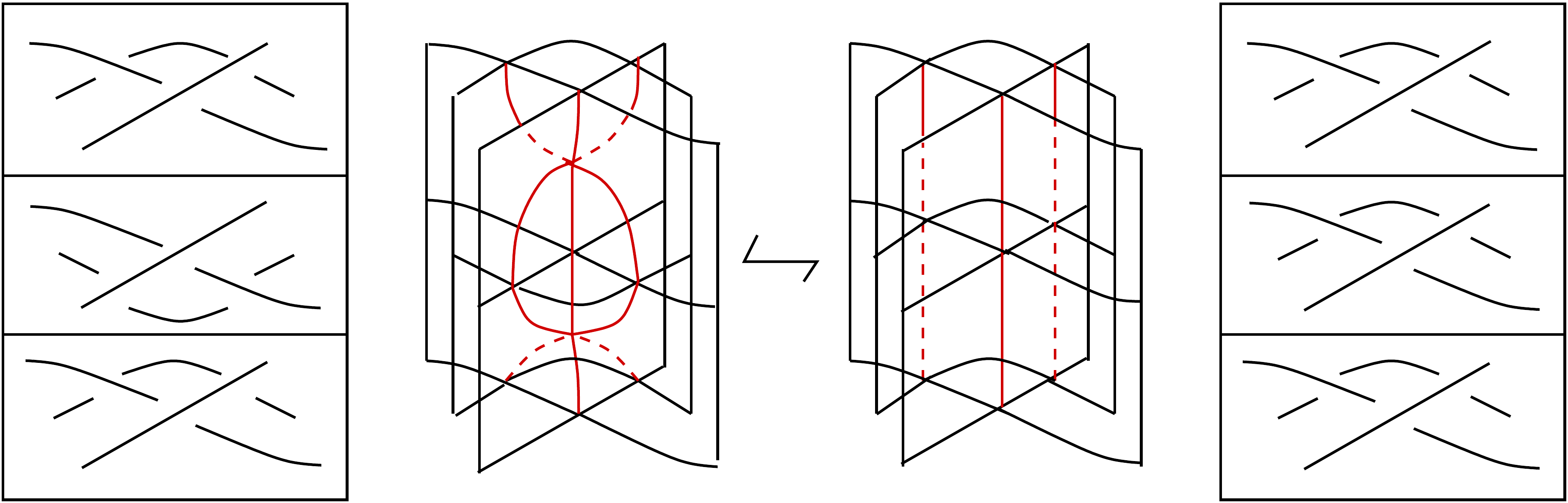}} 
\put(-205, 28){\fontsize{8}{8}$\text{MM5}$}\]

The extended Reidemeister move $R4$ for singular link diagrams is similar to the Reidemeister move $R3$, and the movie move $MM17$ reflects the invertibility property of the move $R4$. A non-degenerate critical point in the isotopy direction creates or annihilates a pair of singular intersection points:

\[ \raisebox{5pt}{\includegraphics[height=0.8in]{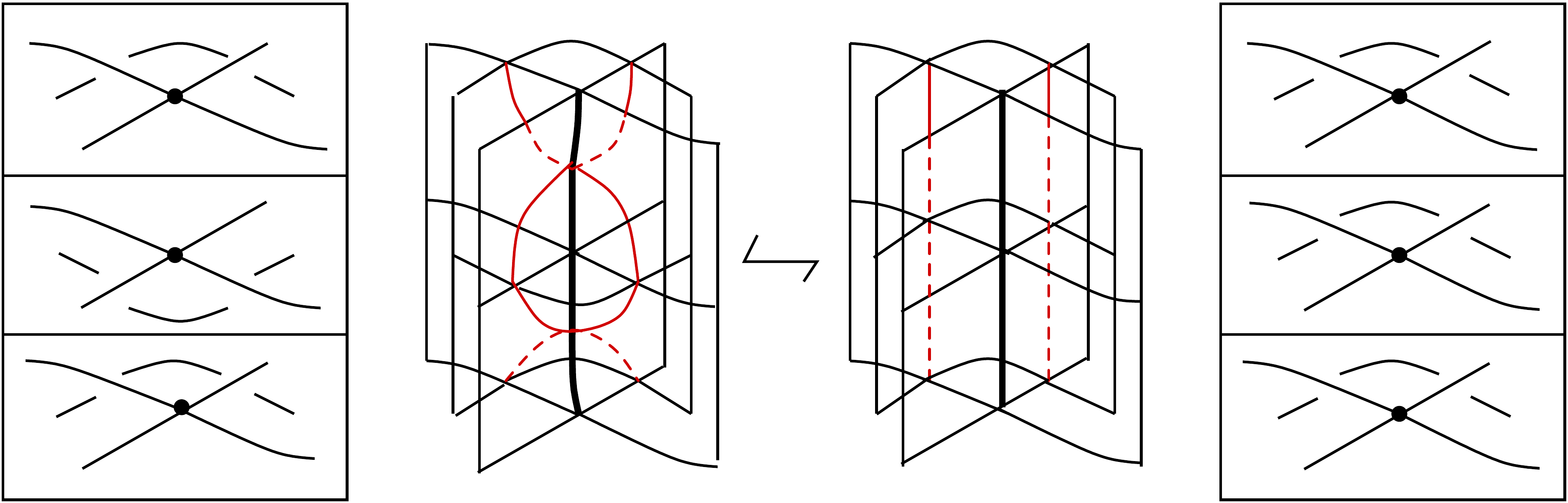}}
\put(-207, 28){\fontsize{8}{8}$\text{MM17}$}
 \]

We proceed now to analyze the 2-dimensional singular sets of the isotopy. These are the traces of the double-point arcs in (2.a) and edges in (2.b) of the diagram $D_0$. In examining the critical points of these 2-dimensional sets, we observe first that the edge set of $D_0$ remains unchanged during the isotopy, and thus there are no critical points for the edge set in (2.b). Therefore, we only need to consider the critical points corresponding to the 2-dimensional sets of type (2.a). The non-degenerate critical points for a surface are maximum/minimum and saddle points, and they account for the movie moves $MM3$ and $MM4$ depicted below along with their corresponding moves for diagrams (these are the \textit{Roseman bubble and saddle moves}; see~\cite{CS, Ca, Ros}):
\[ \raisebox{5pt}{\includegraphics[height=0.6in]{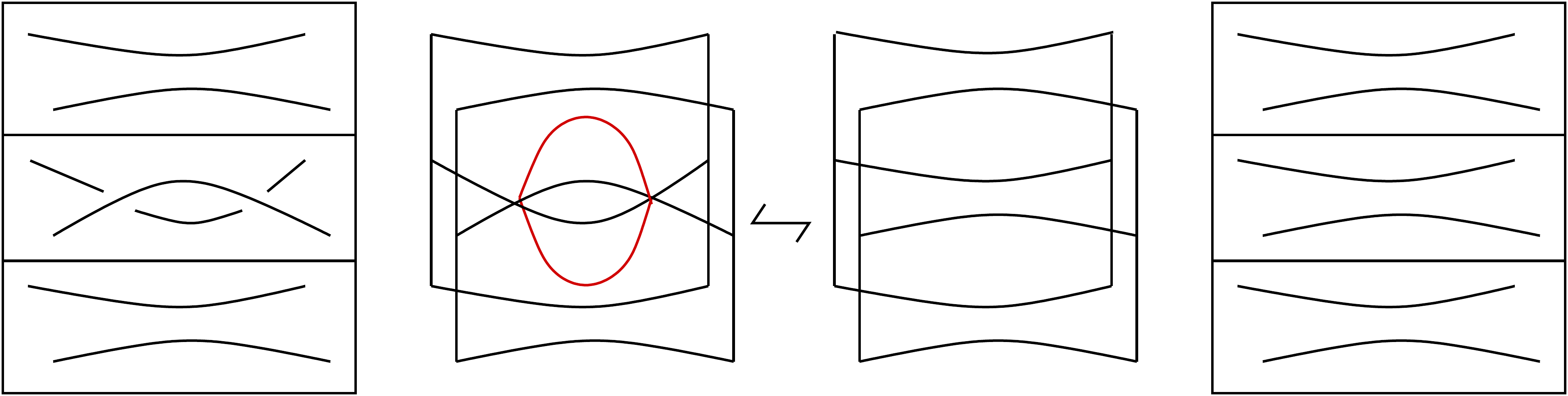}}
\put(-195, 25){\fontsize{8}{8}$\text{MM3}$} \]
\[ \raisebox{5pt}{\includegraphics[height=0.6in]{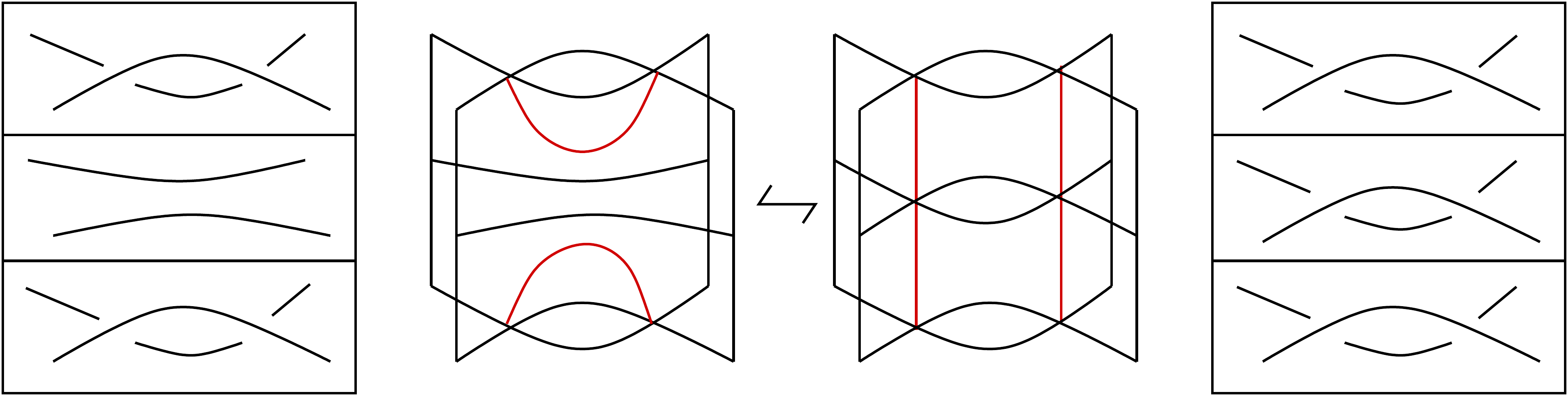}}
\put(-195, 25){\fontsize{8}{8}$\text{MM4}$} \]

(II) The second type of singularities to be considered deal with transverse intersections in the 4-dimensional space-time (the 3-space of the projection times the isotopy direction). A transverse intersection in 4-space occurs between a 1-dimensional arc and a 3-dimensional solid, or between two 2-dimensional sets. We refer to these  intersections as being of type (3-1) and (2-2), respectively.

The type of (3-1) intersections involve arcs formed in space-time by the singular points of type (1.a) and (1.b) and we need to examine their transverse intersections with embedded sheets. The resulting intersections are classified as follows:
\begin{enumerate}
\item[(3-1.B)] A branch point passing through a transverse sheet. In the isotopy direction this corresponds to the transverse intersection between a branch point set and a 3-dimensional set (the  transverse sheet times the isotopy parameter). This gives rise to the movie move $MM8$ depicted below along with the associated move on diagrams (this is one of the seven Roseman moves~\cite{Ros}):

\[ \put(36, 123){\fontsize{8}{8}$\text{MM8}$}
\raisebox{5pt}{\includegraphics[height=1.6in]{MM8-twosides}} \hspace{1cm} \raisebox{25pt}{\includegraphics[height=0.8in]{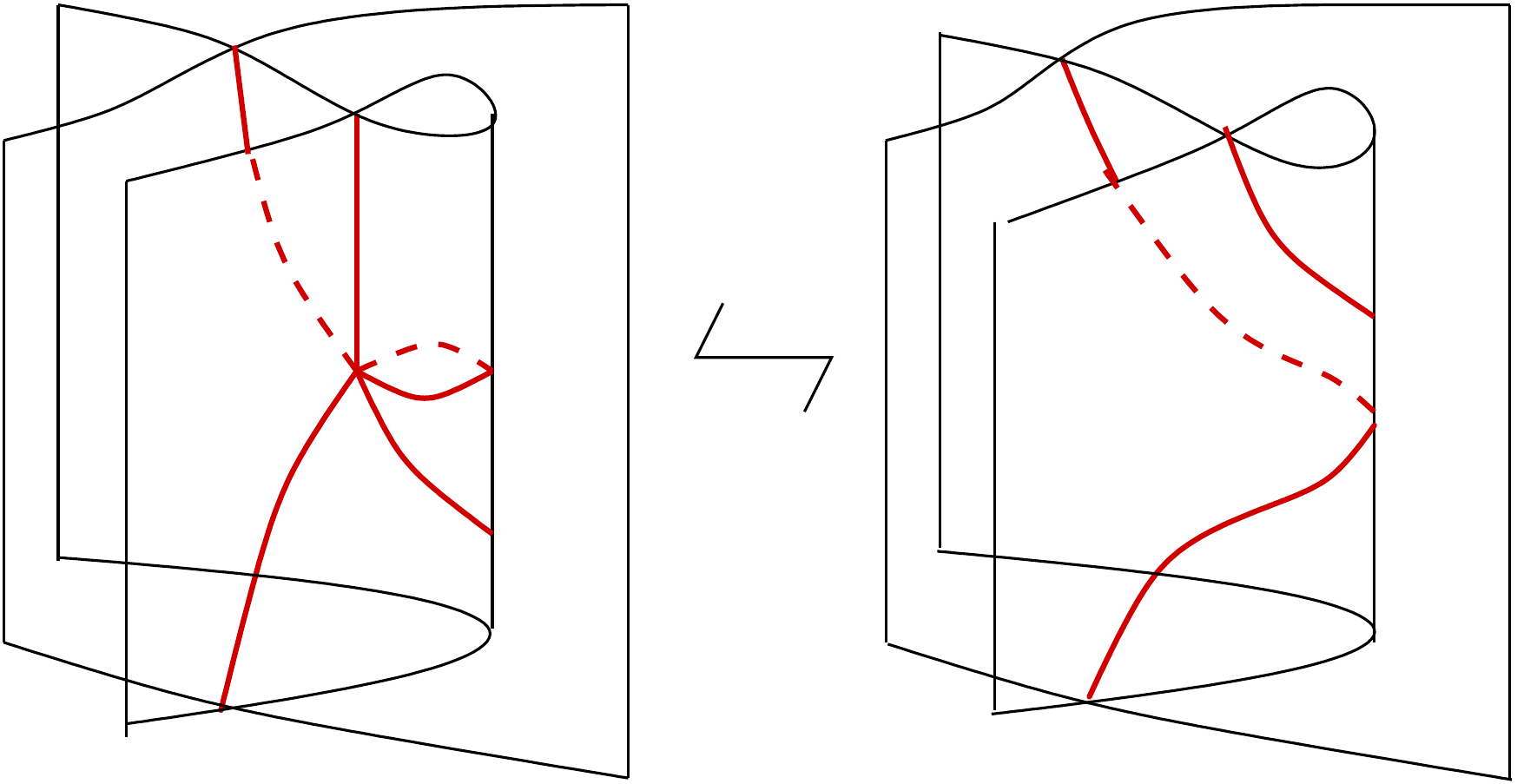}} 
\] 

\item[(3-1.CV)] A crossing-vertex point passing through a transverse sheet. This corresponds to an isolated point created as the transverse intersection between a crossing-vertex point set and a 3-dimensional set, and is represented by the movie move $MM18$:

 \[ \put(30, 90){\fontsize{8}{8}$\text{MM18}$}
\raisebox{0pt}{\includegraphics[height=1.2in]{MM18-twosides}} \hspace{1cm} \raisebox{0pt}{\includegraphics[height=1.1in]{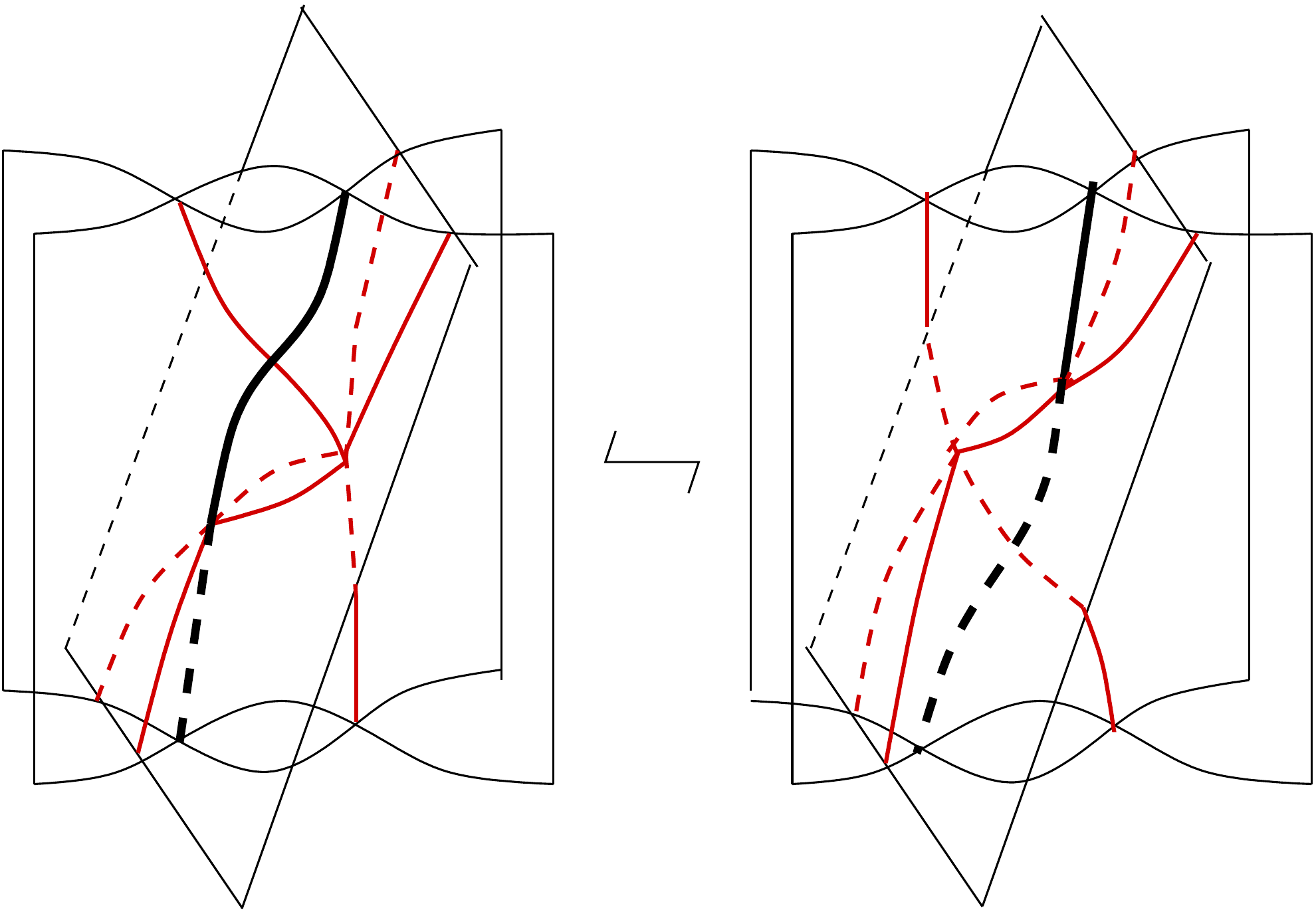}} 
\] 

\item[(3-1.T)] A triple point passing through a transverse sheet. This corresponds to the \textit{tetrahedral move} or \textit{quadruple point move}, and is one of the Roseman moves for isotopic knotted surfaces. The movie parametrization, $MM10$, for this move is as follows:

\[ \put(33, 121){\fontsize{8}{8}$\text{MM10}$}
\raisebox{5pt}{\includegraphics[width = 1.25in]{MM10-twosides}}  \hspace{1cm} \raisebox{15pt}{\includegraphics[height=1in]{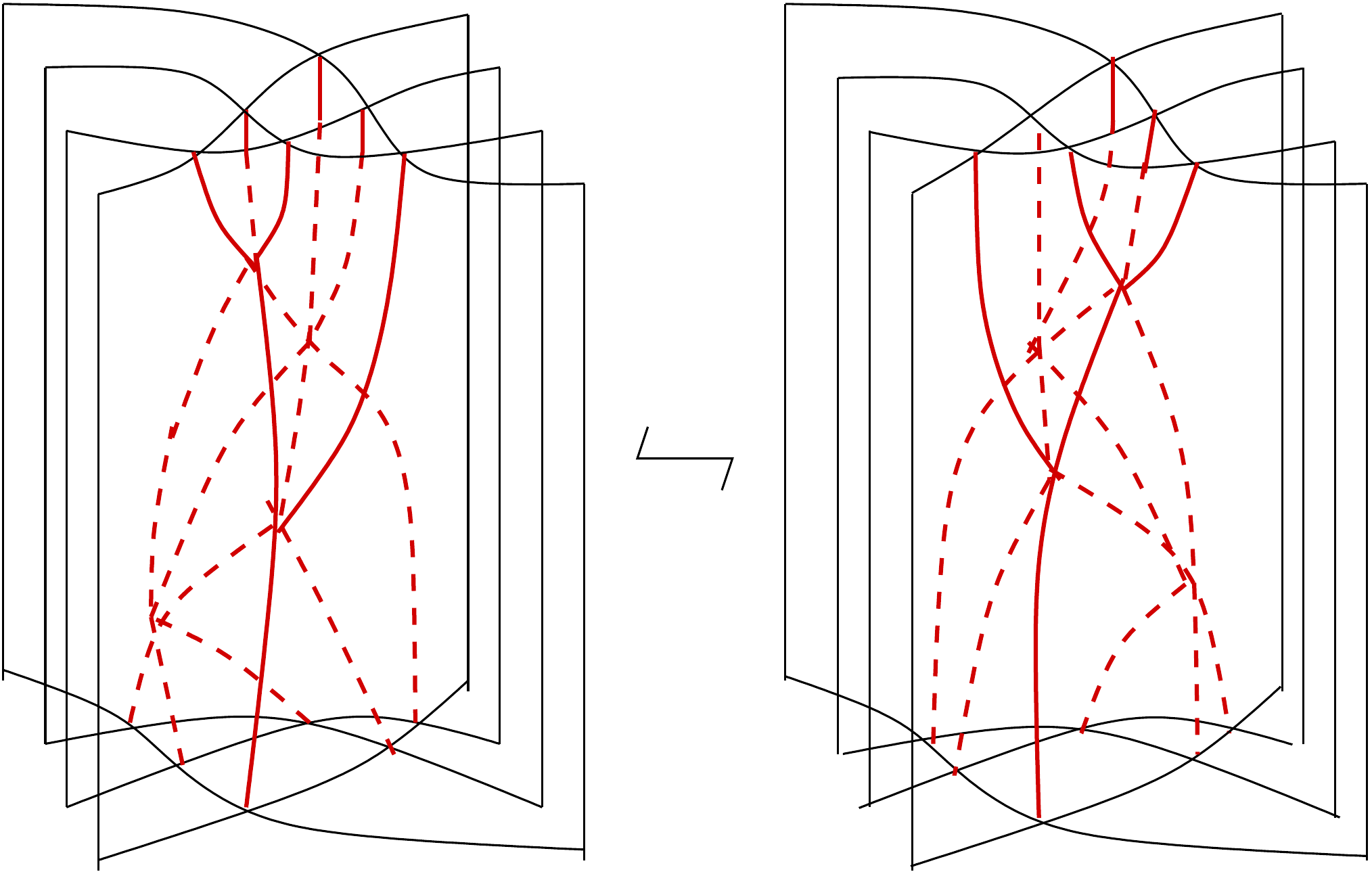}} 
\] 
Two consecutive stills in the movies representing the two sides of this movie move differ by a type-III Reidemeister move, which results in an isolated triple point. There are four possible triple points (formed by any three of the four strands involved in the movie move) and any of these points passes through a transverse sheet. The triple point evolves in space-time to become a 1-dimensional arc and the transverse sheet evolves into a 3-dimensional solid; their transverse intersection in 4-space is an isolated quadruple point. 

Note that there is another interpretation for the critical behavior of the movie move $MM10$, which will be discussed below in (2-2.DD). 

\item[(3-1.SI)] A singular intersection point passing through a transverse sheet. There are three movie moves that correspond to this type of transverse intersection, namely the moves $MM21, MM22$ and $MM23$. These moves are somewhat similar to the move $MM10$, and we refer to them as the \textit{IIX-move}, the \textit{IXI-move}, and the \textit{XII-move}, respectively. An edge $X \times [0, 1]$ can pass over, under, or between a pair of transverse sheets, and there are two ways to interpret the critical behavior of these moves. One way is to consider the intersection of an edge with one of the two sheets to obtain a singular intersection point (which evolves to be a 1-dimensional arc in the isotopy direction) that can pass through the other sheet (which traces a 3-dimensional solid in space-time). So, there is an isolated point as the transverse intersection between a 1-dimensional arc and a 3-dimensional solid in space-time.
\[ \put(33, 121){\fontsize{8}{8}$\text{MM21}$}
\raisebox{5pt}{\includegraphics[width = 1.25in]{MM21-twosides}}  \hspace{1cm} \raisebox{15pt}{\includegraphics[height=1in]{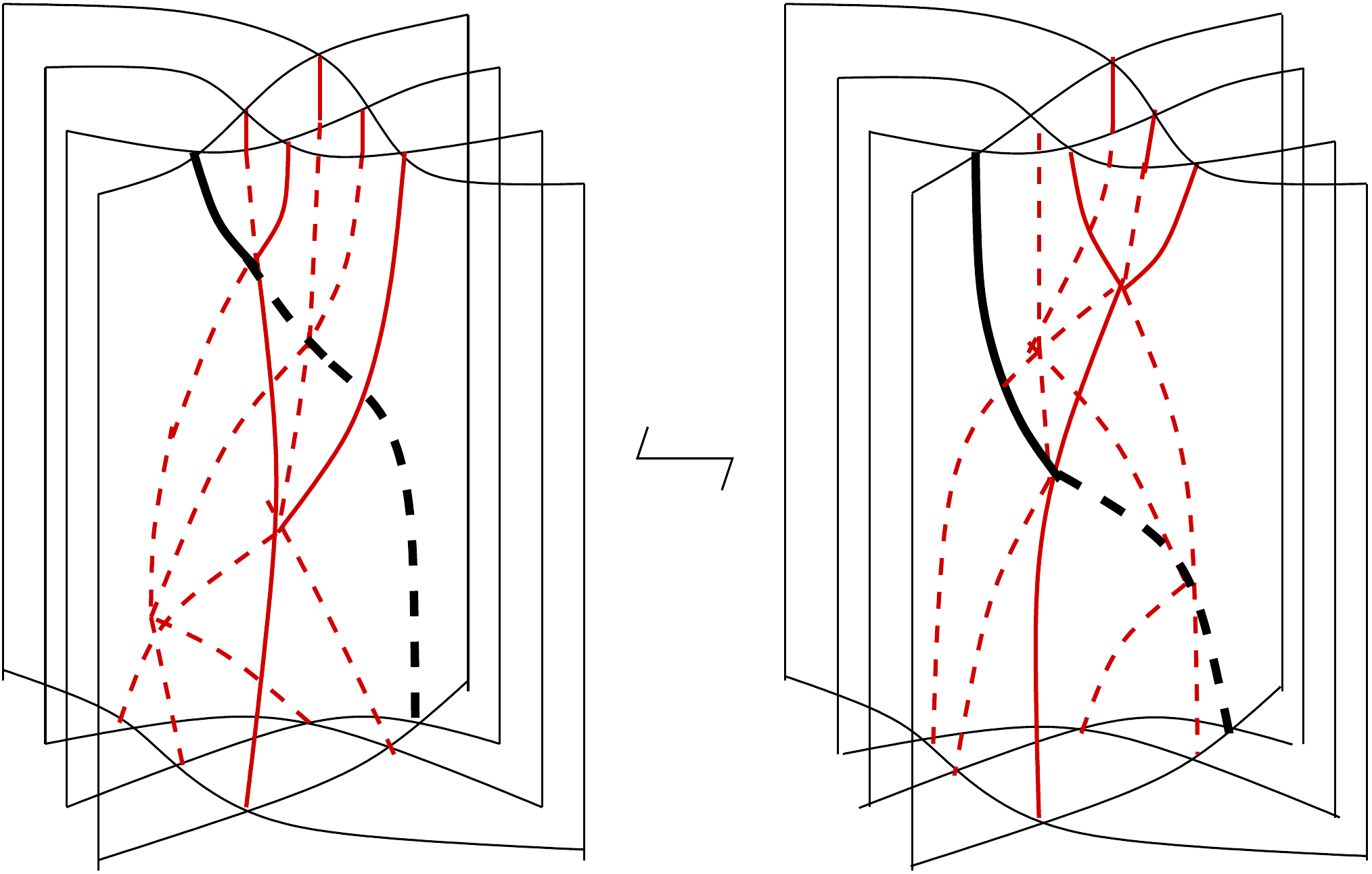}} 
\] 
\[\put(42, 90){\fontsize{8}{8}$\text{MM22}$}
\put(170, 90){\fontsize{8}{8}$\text{MM23}$}
\raisebox{0pt}{\includegraphics[width=1.2in, angle=90]{MM22-twosides}} \hspace{0.7cm}
\raisebox{0pt}{\includegraphics[width=1.2in, angle =90]{MM23-twosides}} \]
Above we only depicted the move on diagrams corresponding to the movie move $MM21$, since these three movie moves project to the same diagrams. Note also that the case (2-2.XD) given below provides another interpretation for these movie moves.
\end{enumerate}

\pagebreak

The type of (2-2) intersections are as follows:

\begin{enumerate}
\item[(2-2.XX)] The intersection between two edges of the cobordism. This critical behavior is represented by the movie move $MM24$ given below:
\[\put(33, 116){\fontsize{8}{8}$\text{MM24}$}
\raisebox{0pt}{\includegraphics[width=1.25in]{MM24-twosides}} \hspace{1cm} \raisebox{15pt}{\includegraphics[height=1in]{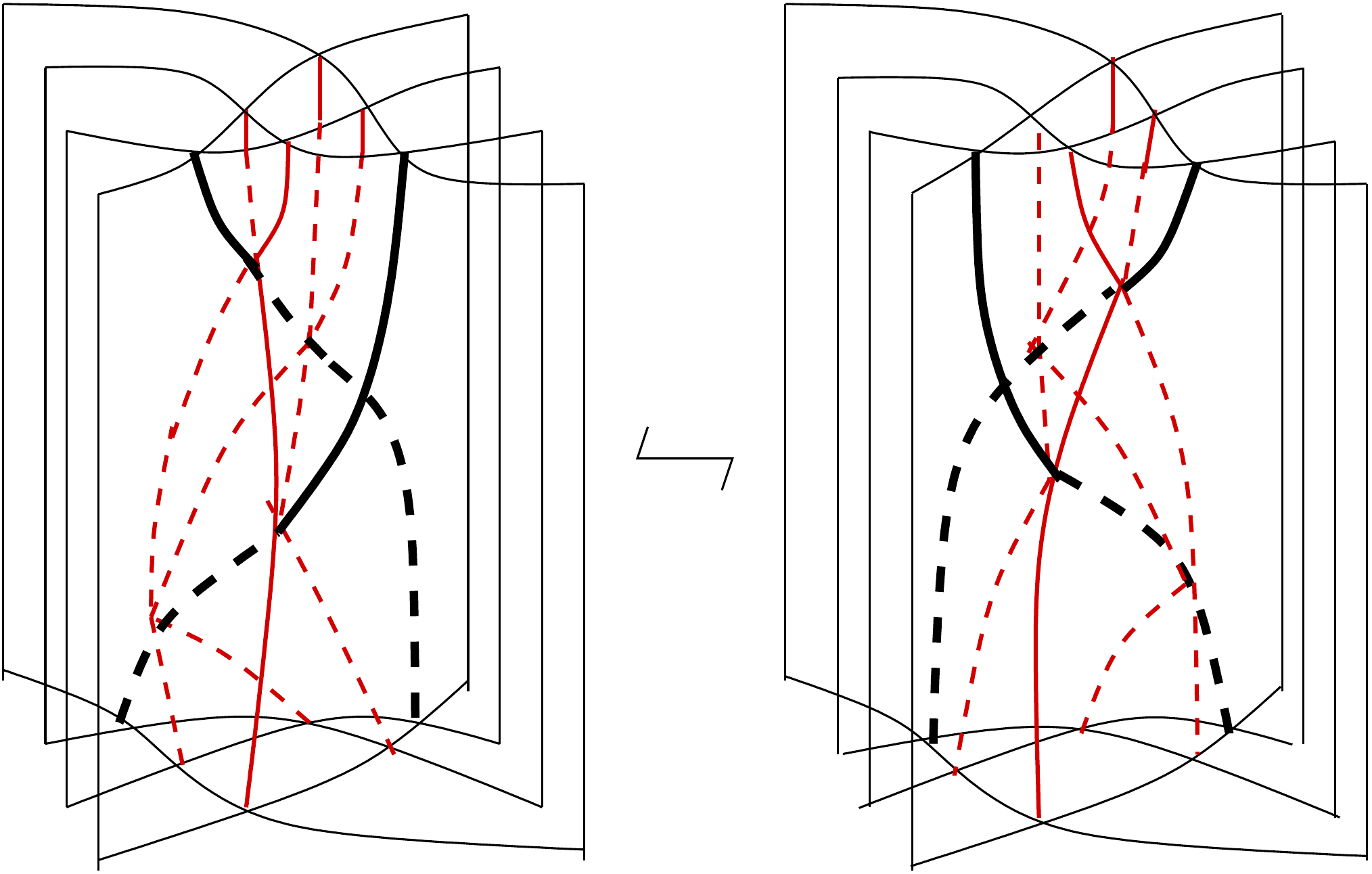}} 
 \]
The movie move $MM24$, also called the \textit{XX-move}, deals with two edges of a cobordism passing through each other in space-time. An edge of the cobordism evolves into a 2-dimensional set in space-time, and the transverse intersection of a pair of 2-dimensional disks in 4-space is an isolated point.

\item[(2-2.DD)] The intersection between two double-point arcs. This type of critical behavior for the isotopy is similar to the previous case; one only needs to replace an edge with a double-point arc. This  critical behavior is represented by the tetrahedral move $MM10$, regarded this time as the creation of an isolated point formed as the transverse intersection between a pair of 2-dimensional sets (double-point arcs times the isotopy parameter).

\item[(2-2.XD)] The intersection between an edge of the cobordism and a double-point arc. There are three movie moves that correspond to this type of critical behavior for the isotopy: the \textit{IIX-, IXI-} and \textit{XII-moves}. It is easy now to see that each of the moves $MM21$ through $MM23$ can also be interpreted as corresponding to an isolated point of the isotopy which is the transverse intersection between a pair of 2-dimensional sets (an edge and, respectively, a double-point arc times the isotopy parameter).
\end{enumerate}

The moves so far exhaust the Morse critical points for singular strata of the isotopy and transverse intersections between strata (where the isotopy direction is regarded as a height function). These are the movie parametrizations of the Roseman-type moves for diagrams of isotopic singular link cobordisms

(III)  As the isotopy passes through a singular projection certain additional changes occur among the fold sets and the multiple point sets which do not affect the topology of the projection. When watching the isotopy in the retinal plane, we may see cusps and fold lines in the facets of the projection (which involve optima and saddles for disks), as well as folds in the double-point arcs (which involve optimal points for curves). According to our definition of cobordant singular links and conventions on regular diagrams for singular link cobordisms, there are no folds or cusps for the edge set of a cobordism times the isotopy direction. 

We have the following types of changes:

\begin{enumerate}
\item[(i)] Changes that involve cusps. There is a cusp on the fold line that corresponds to the cancellation of a birth/death of a circle and a saddle, which is represented by the movie move MM11:


 \[ \put(23, 82){\fontsize{8}{8}$\text{MM11}$}
 \raisebox{0pt}{\includegraphics[height=1.1in]{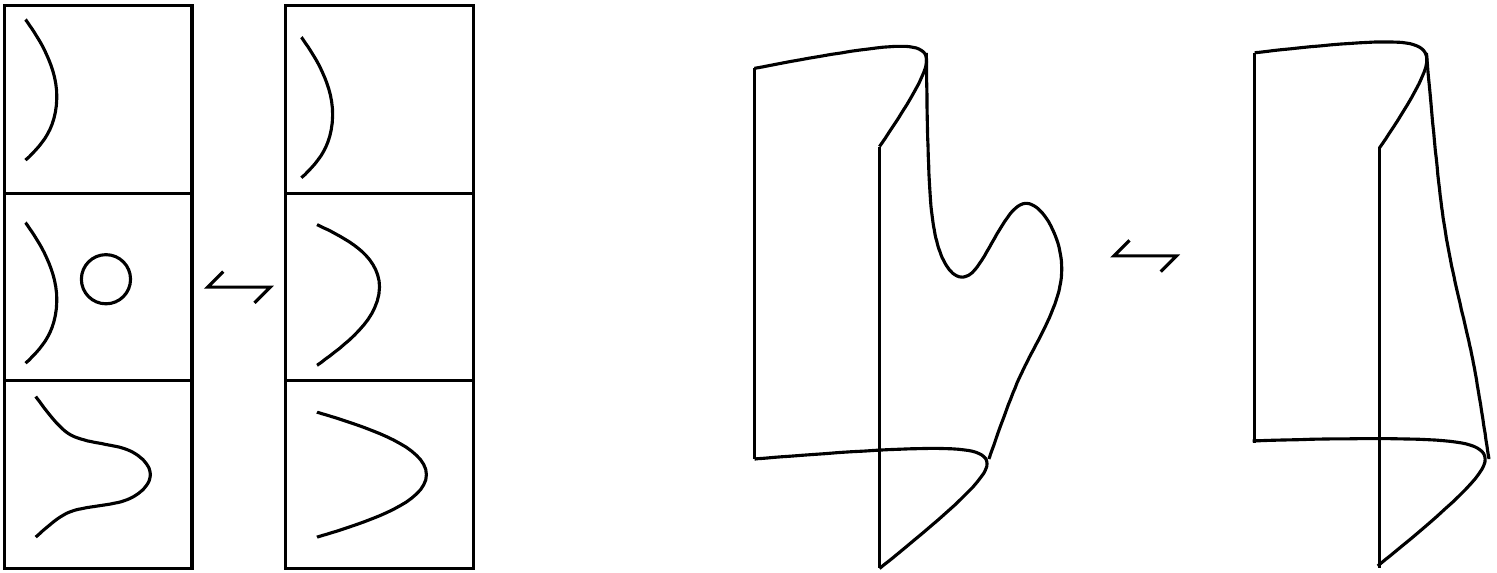}} \]
An optimum point for a double-point arc can be pushed through a branch point, as described by the movie move $MM7$:
\[  \put(25, 75){\fontsize{8}{8}$\text{MM7}$} 
\raisebox{0pt}{\includegraphics[height=1in, width=1in]{MM7-twosides}}
\hspace{1cm} \raisebox{0pt}{\includegraphics[height=0.9in]{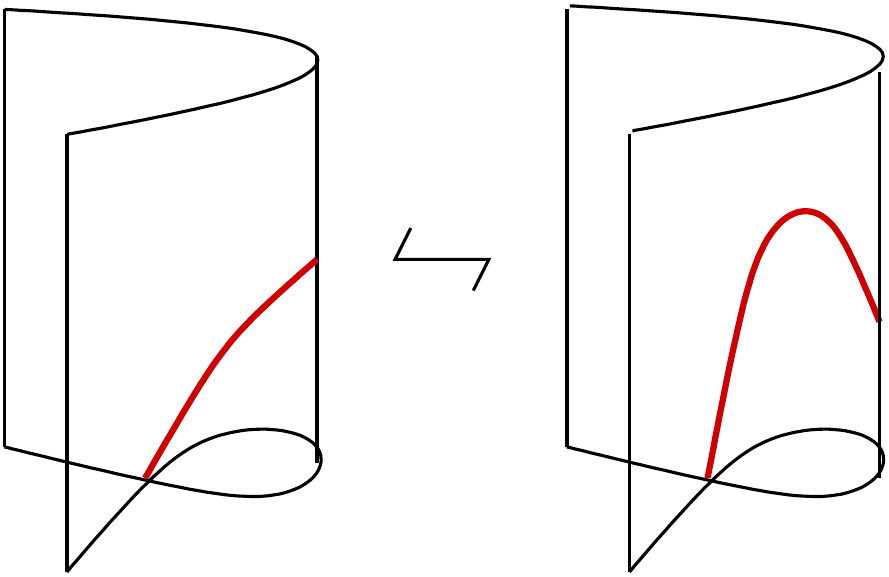}}
 \]
 In addition, a pair of canceling optima for a double-point curve can change in a cusp-like manner; this is represented by the movie move $MM9$:
 \[ \put(25, 75){\fontsize{8}{8}$\text{MM9}$} 
 \raisebox{0pt}{\includegraphics[height=1in, width=1in]{MM9-twosides}}
 \hspace{1cm} \raisebox{0pt}{\includegraphics[height=0.9in]{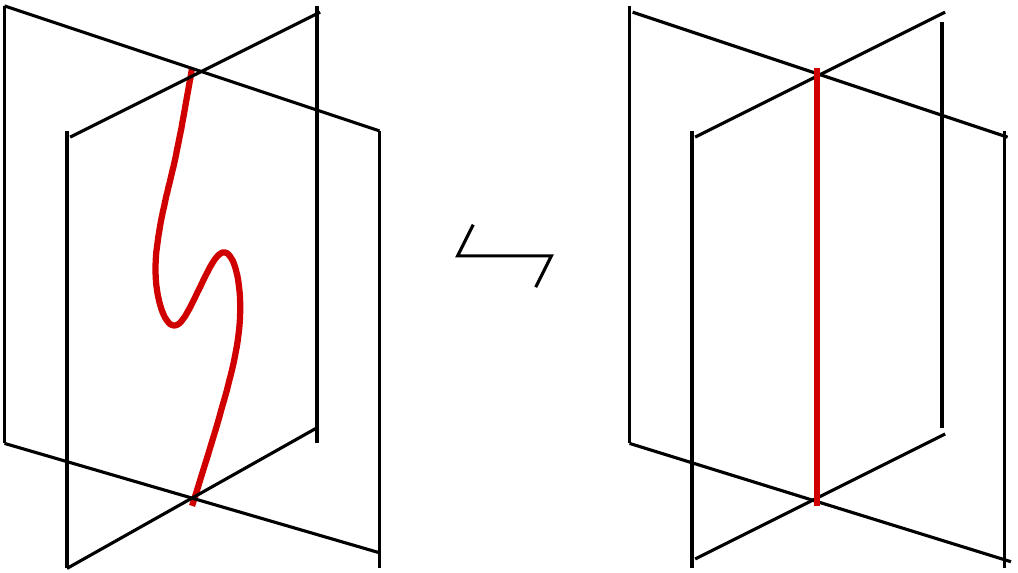}}
 \]
  
We remark that there are no analogue moves for the edge set as the moves $MM7$ and $MM9$, since our definition of cobordant singular links does not allow maximum or minimum points on edges of a cobordism. The reason for this convention is that the extended Reidemeister moves for singular link diagrams do not contain type I and type II moves involving vertices.

\item[ii)] Changes that affect the relative position of double-point arcs in relation to an optimum or saddle in a facet of the cobordism. First, a branch point can pass over an optimum or a saddle in a facet. In the isotopy direction, this corresponds to passing the surface of double-points transversely through an optimum fold line or saddle fold line, respectively, which are represented by the movie moves $MM12$ and $MM13$, respectively:
\[\raisebox{0pt}{\includegraphics[height=1.1in]{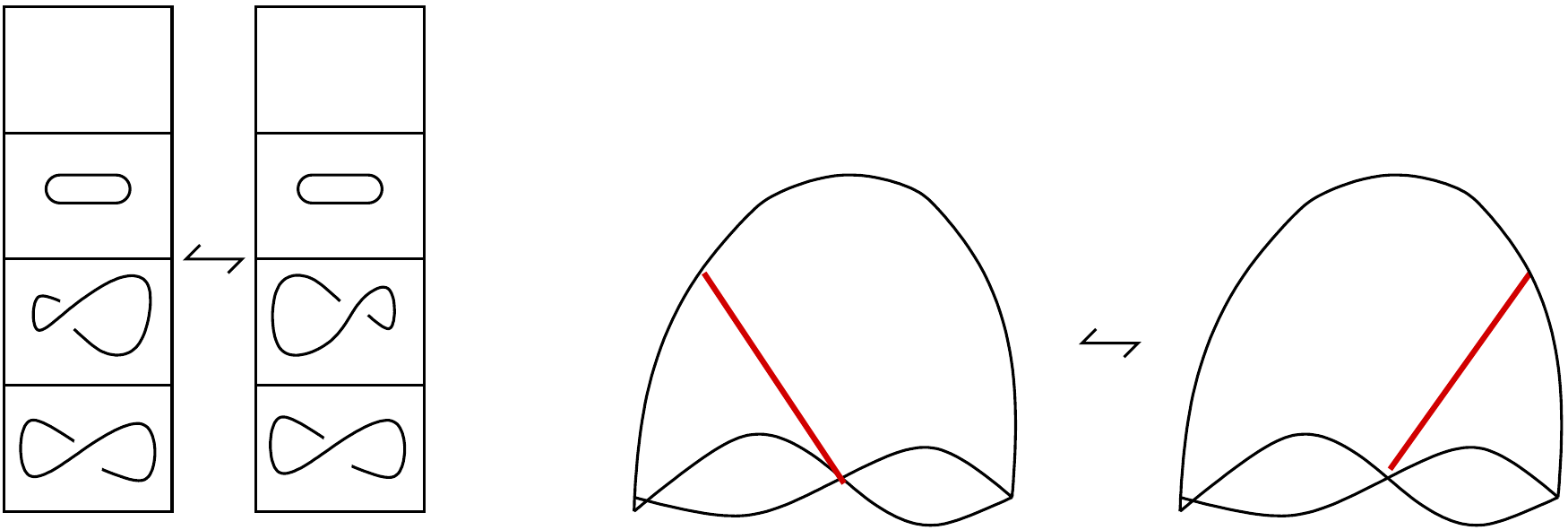}} 
   \put(-215, 82){\fontsize{8}{8}$\text{MM12}$}\]
 \[\raisebox{0pt}{\includegraphics[height=1.1in]{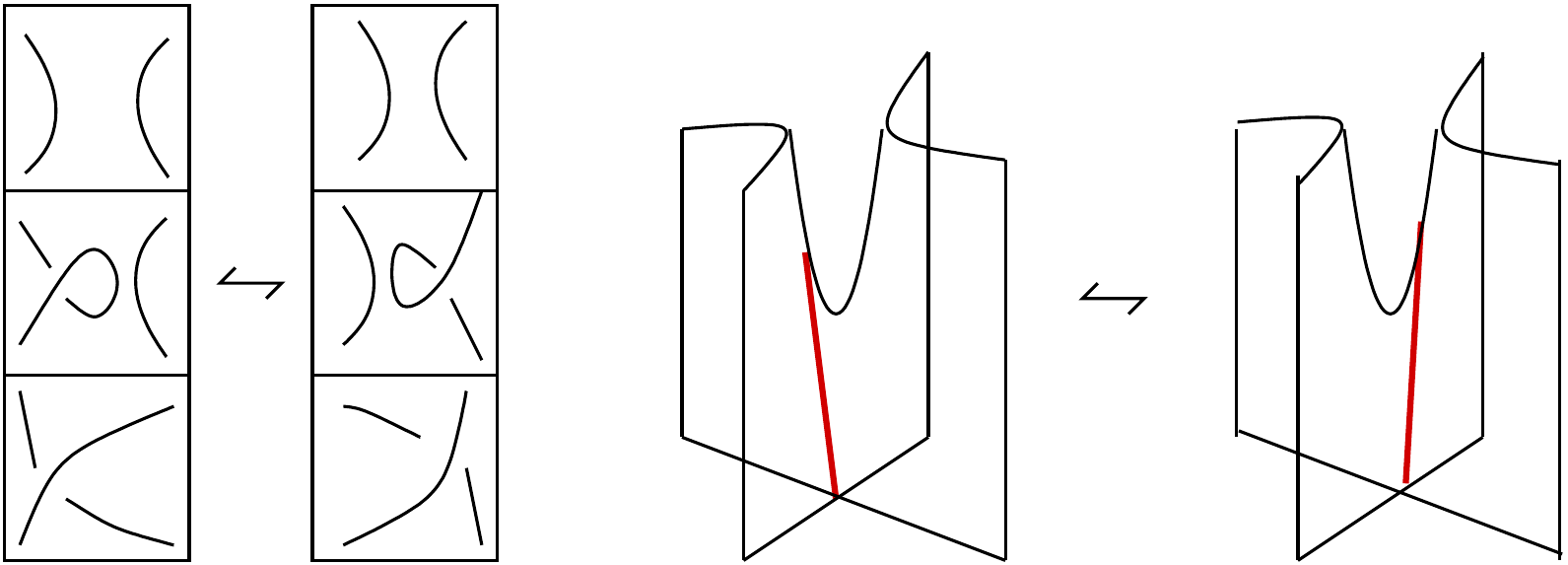}} 
   \put(-197, 82){\fontsize{8}{8}$\text{MM13}$}\]
Second, an optimum of a double-point curve may pass over an optimum or a saddle in a facet. These changes are represented by the movie moves $MM14$ and $MM15$:
  \[\raisebox{0pt}{\includegraphics[height=1.1in]{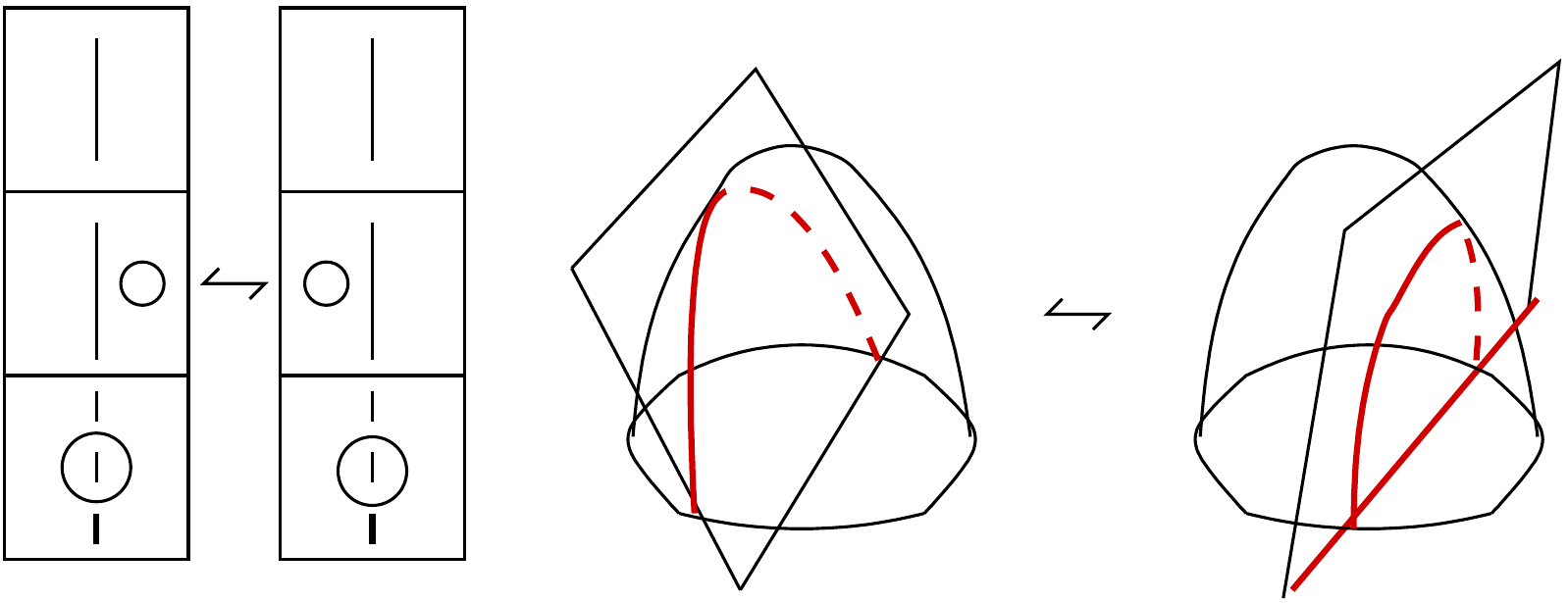}} 
   \put(-185, 83){\fontsize{8}{8}$\text{MM14}$}\]
  \[\raisebox{0pt}{\includegraphics[height=1.1in]{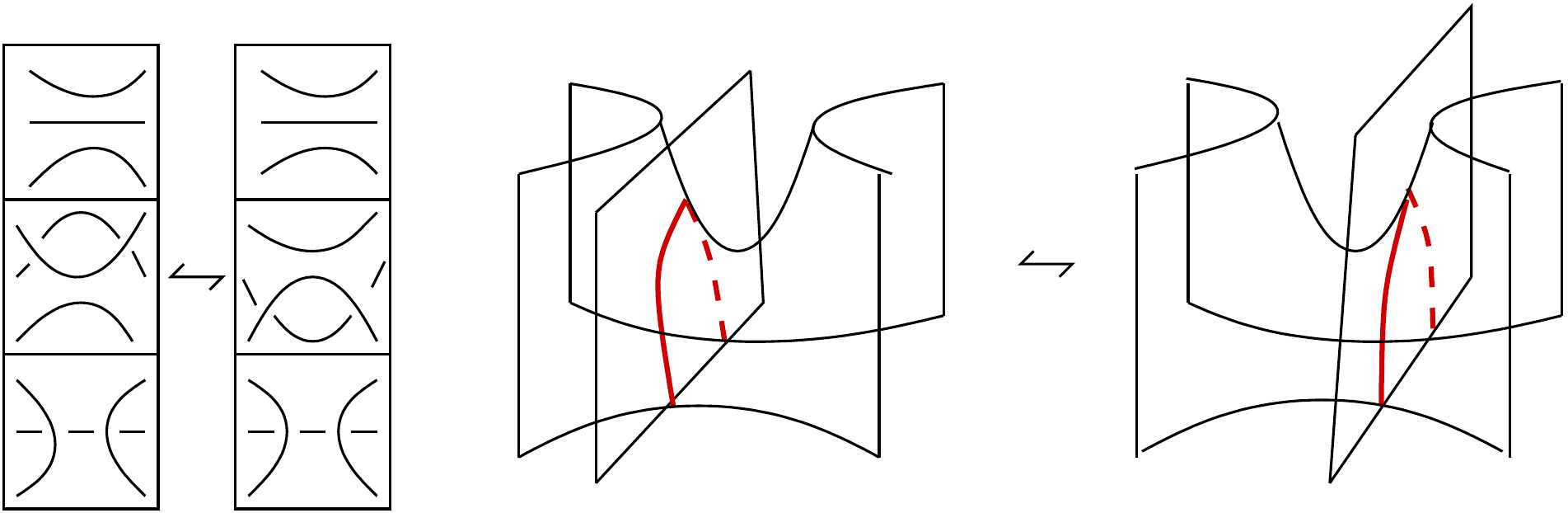}} 
   \put(-222, 77){\fontsize{8}{8}$\text{MM15}$}\]
Equivalently, the moves $MM14$ and $MM15$ occur when, in the isotopy direction, a double-point curve passes over a fold line near an optimum or, respectively, a saddle in a facet.

\item[(iii)] Changes that affect the relative position of singular points (triple point, crossing-vertex point, or singular intersection point) in relation to the optimal point of a double-point arc.  Such a singular point can be pushed over an optimum point on a double-point arc.
The case for a triple point is represented by the movie move $MM6$:
 \[ \put(25, 75){\fontsize{8}{8}$\text{MM6}$} 
 \raisebox{0pt}{\includegraphics[height=1in, width=1in]{MM6-twosides}}
 \hspace{1cm} \raisebox{0pt}{\includegraphics[height=1in]{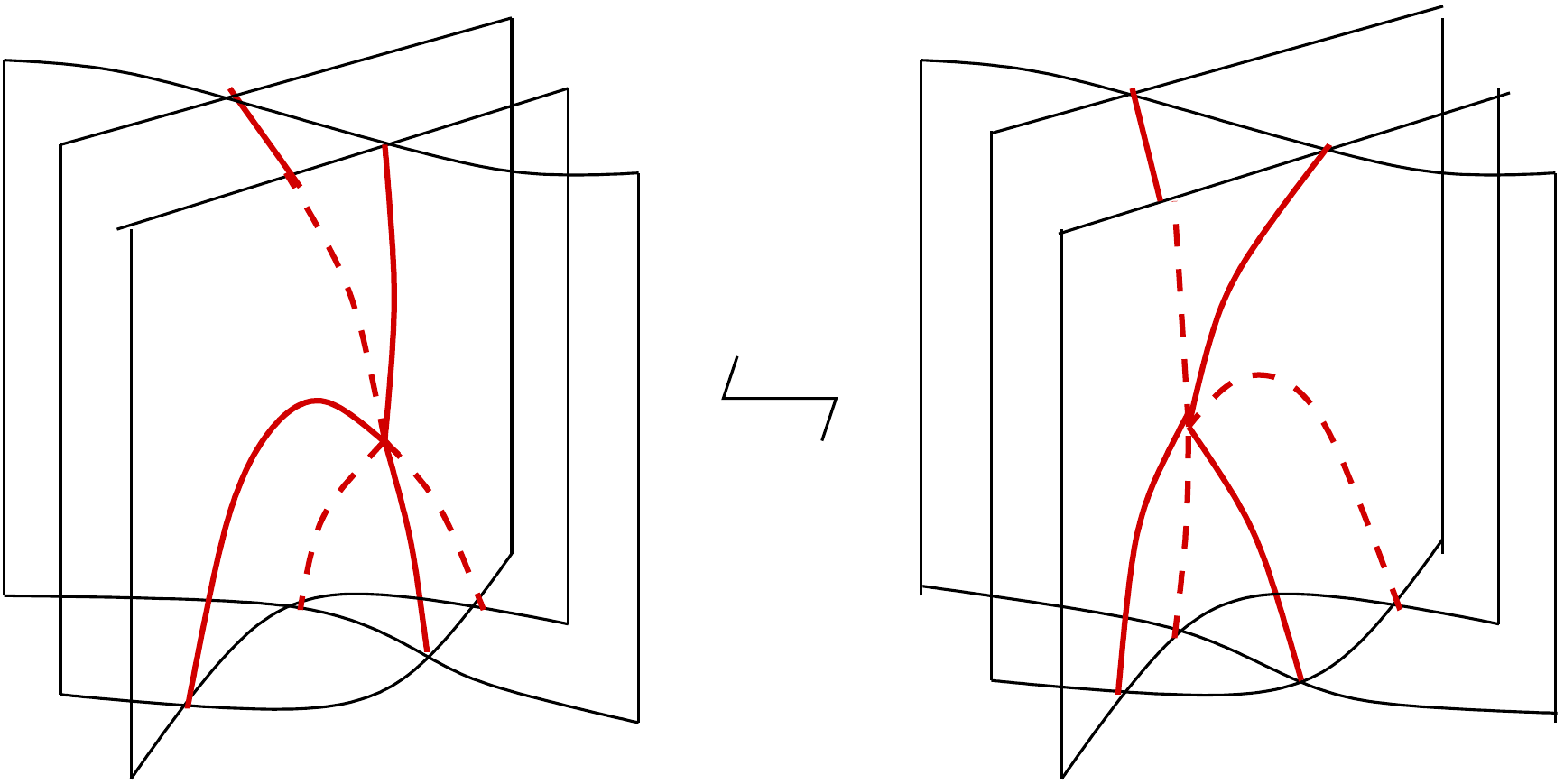}}
 \]
Similarly, a crossing-vertex point can pass through an optimum on a double-point arc as represented by the movie move $MM19$:
 \[ \put(25, 75){\fontsize{8}{8}$\text{MM19}$} 
 \raisebox{0pt}{\includegraphics[height=1in, width=1in]{MM19-twosides}}
 \hspace{1cm} \raisebox{0pt}{\includegraphics[height=0.9in]{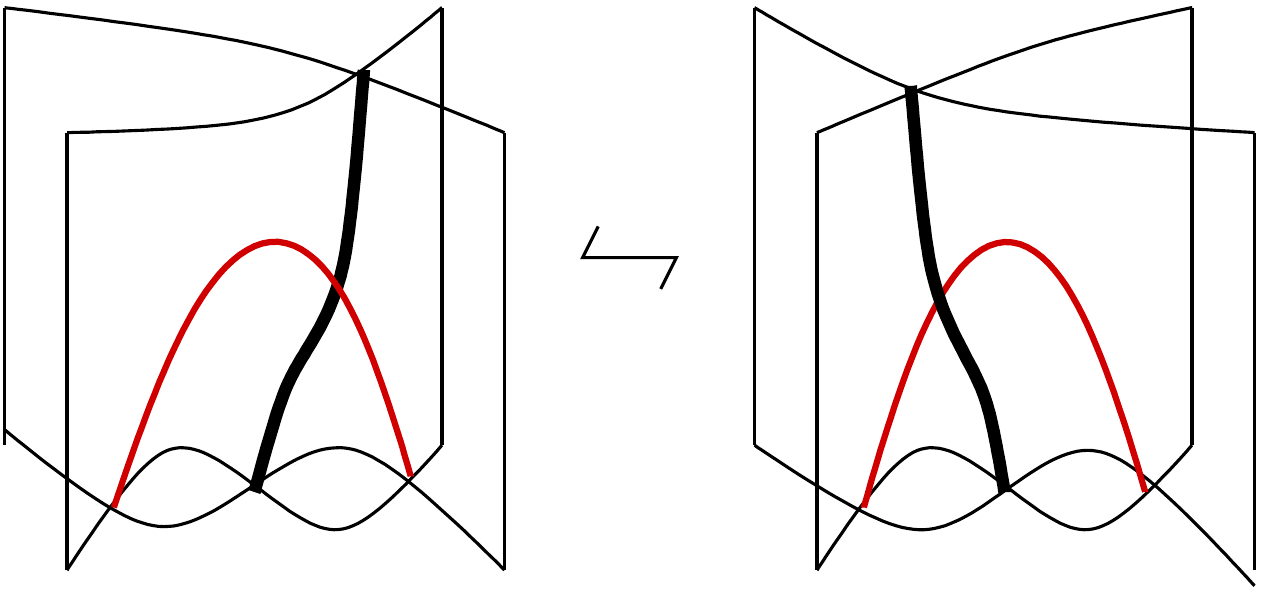}}
 \]
Finally, a singular intersection point can be moved pass an optimum on a double-point arc; this corresponds to the movies move $MM20$:
 \[ \put(25, 75){\fontsize{8}{8}$\text{MM20}$} 
 \raisebox{0pt}{\includegraphics[height=1in, width=1in]{MM20-twosides}}
 \hspace{1cm} \raisebox{0pt}{\includegraphics[height=1in]{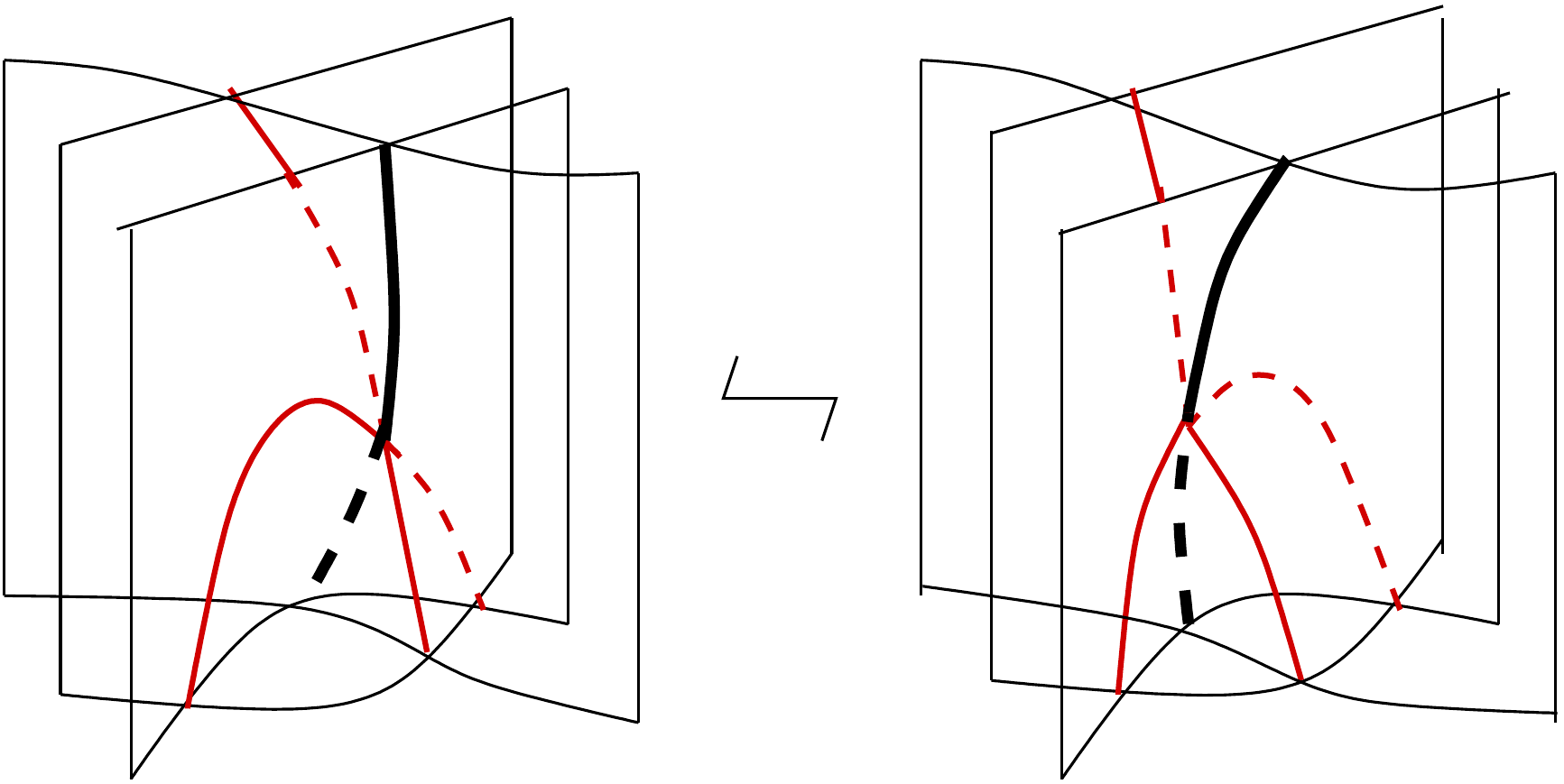}}
 \]

\end{enumerate}
The list (i)--(iii) above exhaust the list of possible changes in folds and transverse intersections between fold lines and multiple point sets. This completes the proof.


\section{Concluding remarks and comments} \label{sec:remarks}

\begin{remark}
We proved Theorem~\ref{thm:movie-moves} by a classification of codimension one singularities in the retinal plane. With the exception of the moves $MM19$, and $MM20$ the other eight movie moves of type (III) are the Carter-Saito movie moves~\cite{CS} that supplement the movie parametrizations of the Roseman moves~\cite{Ros} for knotted surfaces to the case when there is a height function in the 3-space of the projection. In addition, the seven movie moves of type (I) and (II) that supplement the set of movie parametrizations of the seven Roseman moves are $MM16$ through $MM18$ and $MM21$ through $MM24$.

It was noted and showed in \cite{Ca0} that the movie moves $MM12$ and $MM14$ follow from the movie moves $MM11$ and $MM13$ and, respectively, $MM11$ and $MM15$, together with interchanging the levels of distant critical points. This is due to a theorem first discovered by John Fischer (see \cite[Figure 6.17]{Ca0}).
\end{remark}

\begin{remark}
Moves in which distant critical points can change levels were not listed among the movie moves, but are included in our Movie-Move Theorem. Figure~\ref{fig:exchanges} shows such moves for movies. For example, a pair of crossing-vertex points can change their relative position (see the first movie move in~Figure~\ref{fig:exchanges}), as can happen with a  singular intersection point and a crossing-vertex point (this corresponds to the second movie move in Figure~\ref{fig:exchanges}).
Similarly, there is a movie move corresponding to changing the relative heights of a triple point and a crossing-vertex point, or of a pair of triple points (the third and fourth movie moves, respectively, in Figure~\ref{fig:exchanges}). These are only a few examples of moves that correspond to exchanging the order in which ESSIs occur, when they occur in disjoint neighborhoods.

\begin{figure}[ht]
\[
\raisebox{-15pt}{\includegraphics[width = 1in, height=1in]{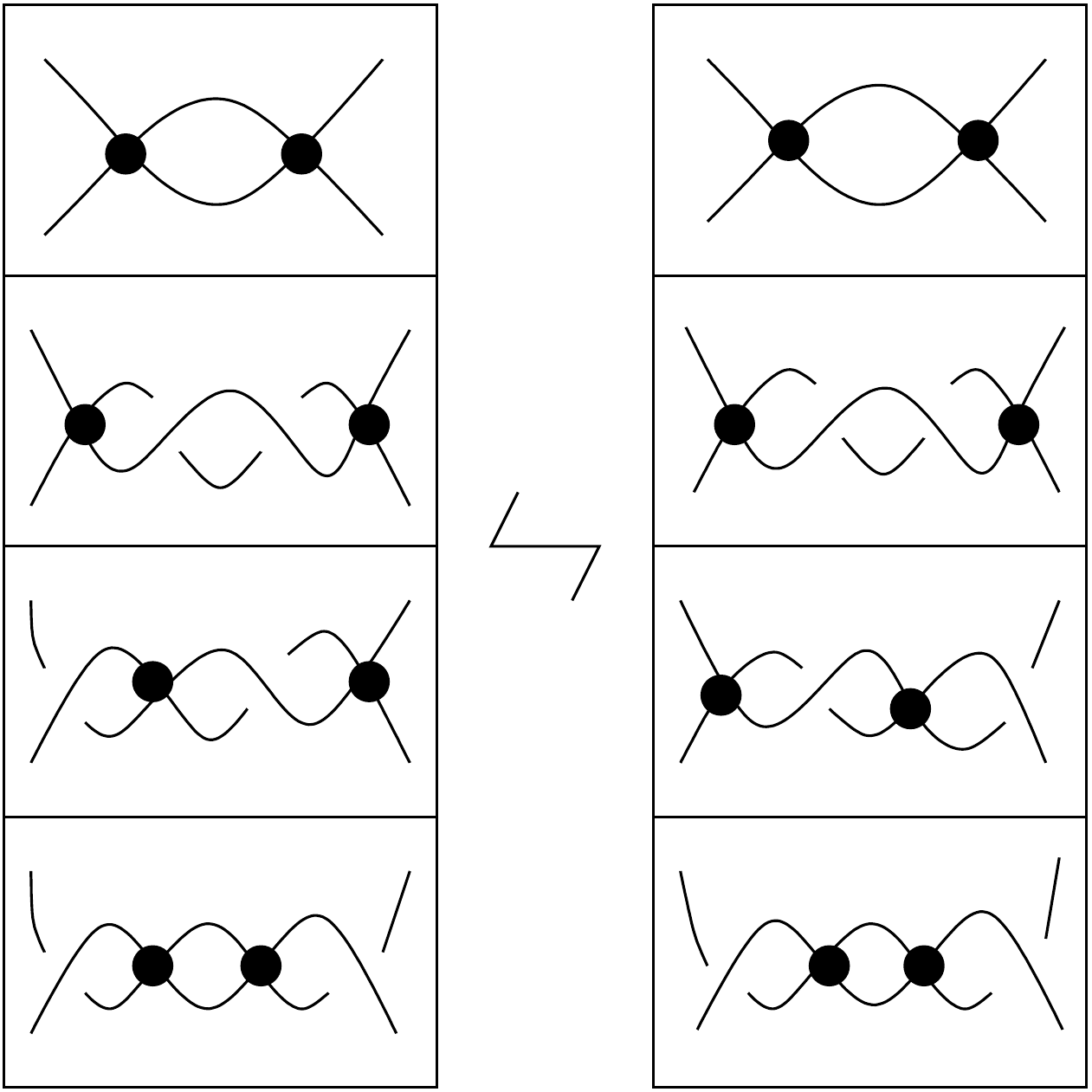}} \hspace{0.7cm} \raisebox{-15pt}{\includegraphics[width = 1in, height=1in]{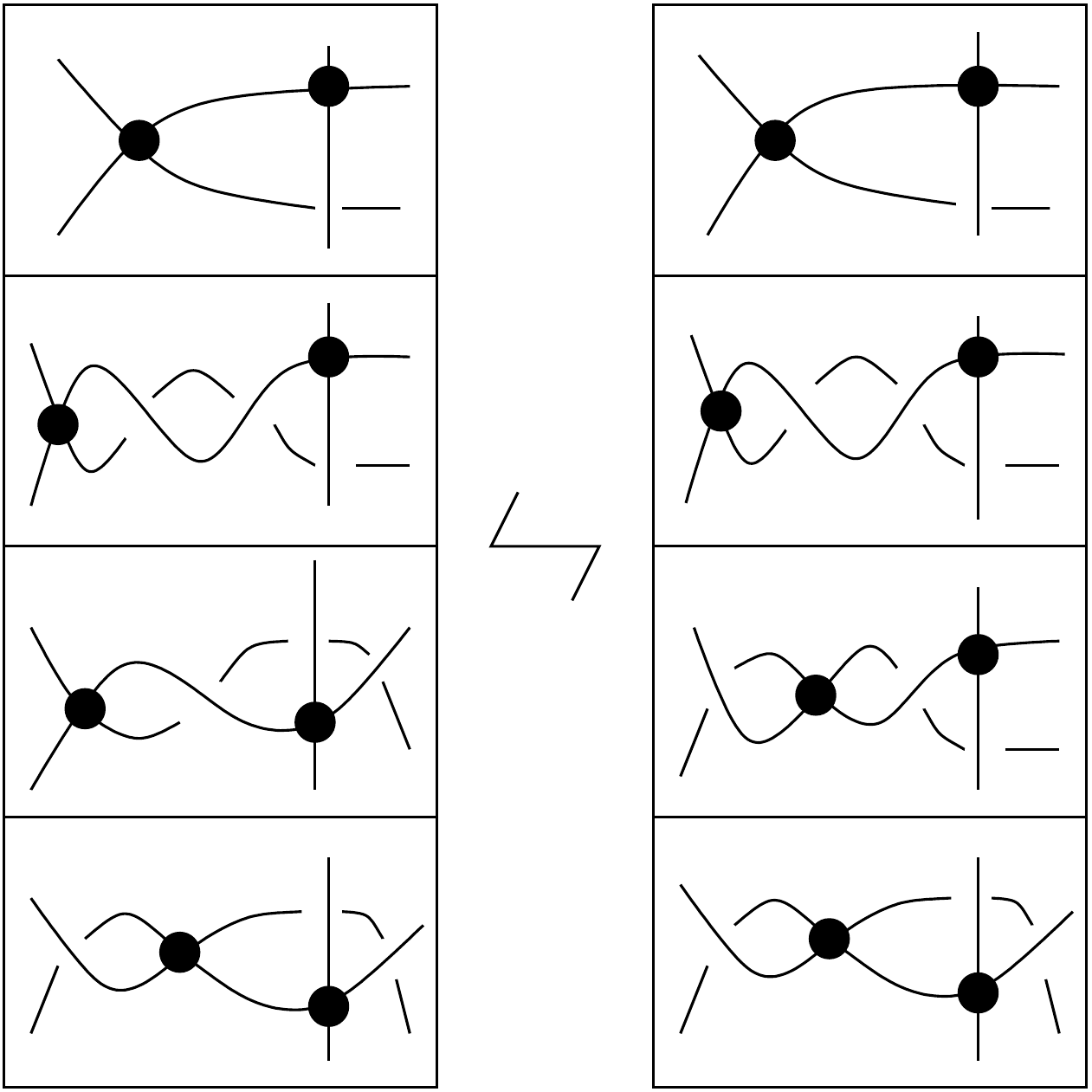}} \hspace{0.7cm}
\raisebox{-15pt}{\includegraphics[height=1in]{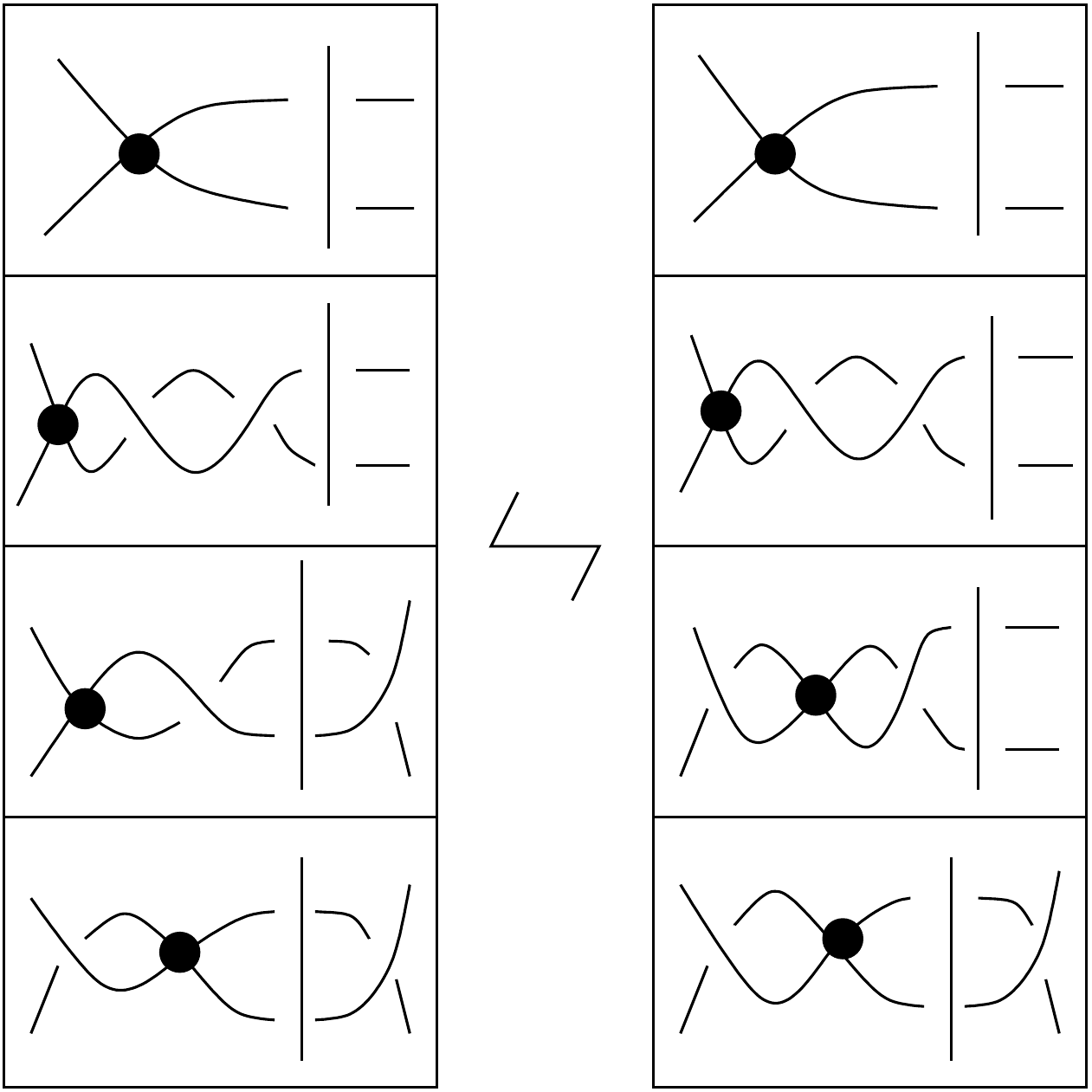}} \hspace{0.7cm} \raisebox{-15pt}{\includegraphics[height=1in]{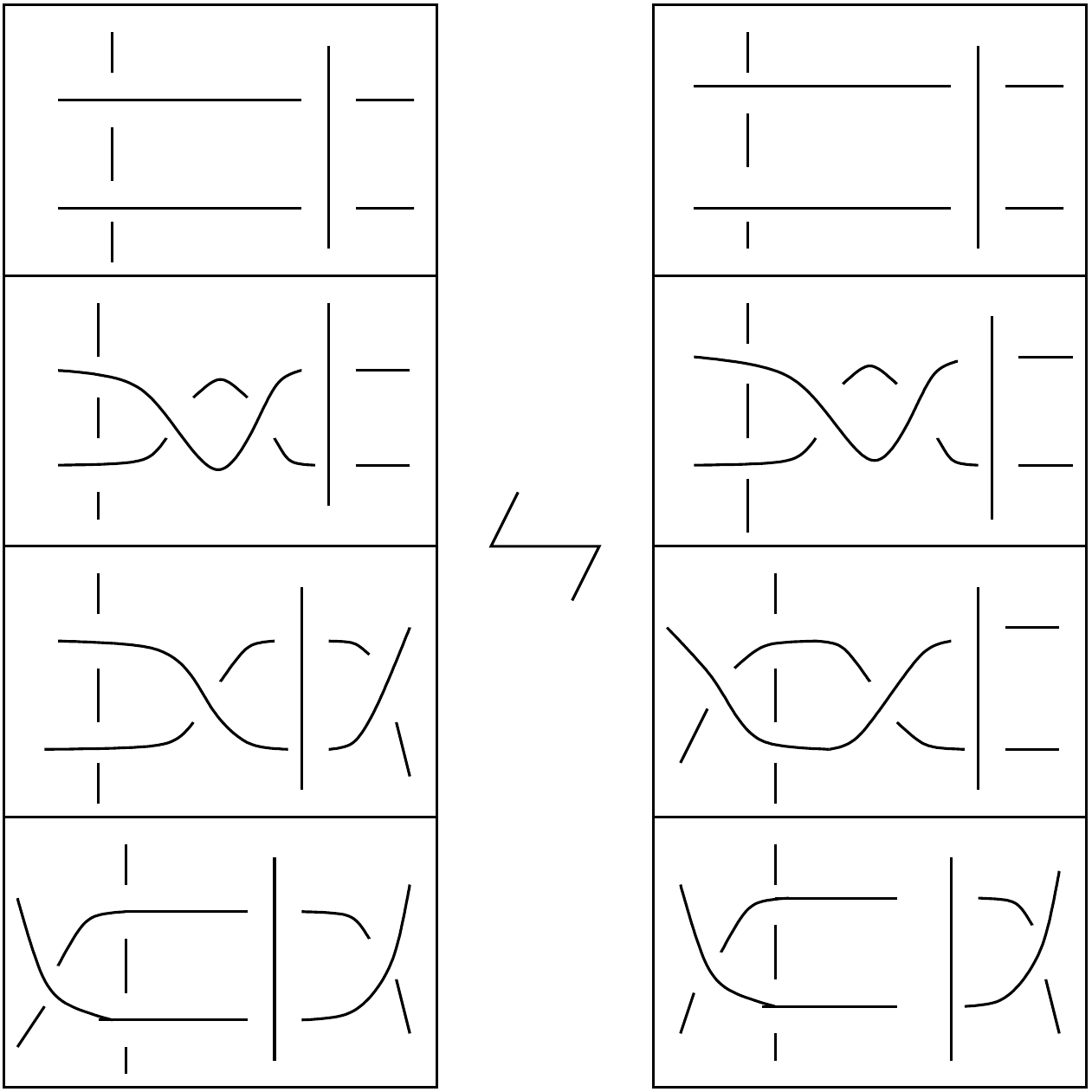}}  \]
\caption{Examples of movie moves in which critical points change levels} \label{fig:exchanges} 
\end{figure}
\end{remark}

\begin{remark}
In this paper we focused on cobordisms between (non-empty) singular links, and said that two singular links are cobordant if they are related via singular link isotopy and Morse transformations (which gave rise to the notion of ESSIs). Since the extended Reidemeister moves for singular link diagrams do not contain type II (and type I) moves involving vertices, we did not include the bigon move among the ESSIs (this move was given in Figure~\ref{fig:bigon}). If one desires to allow cobordisms between two singular links where at least one of these links is the empty link, then the bigon move must be included in the list of ESSIs. In that case, the corresponding list of movie moves of type (III) would be larger, as one would need to take into account the bigon move and its interaction with the ESSIs.

In addition, in this paper we did not attempt to provide a categorical description of singular link cobordisms and their isotopies such as that given in~\cite{CRS} for knotted surfaces. For that, one would need to impose a height function in each still (of a movie description for a singular link cobordism) to obtain an interpretation of the movie moves in the language of 2-categories. In such a setup, the list of movie moves of type (III) would be much larger.
\end{remark}

\textbf{Motivation for this paper.} This paper originated from author's interest in extending the Khovanov homology~\cite{Kh} to singular links, and then asking whether the resulting homology theory satisfies the functoriality property with respect to singular link cobordisms. This problem will be considered in a subsequent paper. Another source of motivation and inspiration was the paper of Scott Carter~\cite{Ca0}, which provides the Roseman-type moves for embedded foams in 4-space.

\textbf{Acknowledgements.} 
This project has started during author's participation in the \textit{Faculty-Undergraduate Research Student Teams} (FURST) Program supported by the NSF Grant DMS--1156273, and has grown from long and useful discussions with Heather M. Russell. The author is indebted to Masahico Saito for illuminating conversations and suggestions. She is also thankful to Scott Carter and Charles Frohman for being patient to answer many questions on the topic.


\end{document}